\documentclass[12pt, a4paper,leqno]{amsart}
\usepackage{amsfonts}
\usepackage{amsmath}
\usepackage{amssymb}
\usepackage{mathrsfs}
\usepackage{graphicx}
\usepackage{amsthm}
\usepackage[unicode=true]{hyperref}
\usepackage{booktabs}
\usepackage{array}
\usepackage{esint}

\providecommand{\U}[1]{\protect\rule{.1in}{.1in}}
\calclayout
\allowdisplaybreaks
\newtheorem{theorem}{Theorem}[section]
\theoremstyle{plain}

\newtheorem{corollary}[theorem]{Corollary}

\newtheorem{problem}{Problem}
\newtheorem{proposition}[theorem]{Proposition}

\numberwithin{equation}{section}
\theoremstyle{definition}
\newtheorem{definition}[theorem]{Definition}

\theoremstyle{remark}
\newtheorem{remark}{Remark}
\def\R{\mathbb{R}}
\def\T{\mathbb{T}}
\def\N{\mathbb{N}}
\DeclareMathOperator*{\argmin}{arg\,min}
\begin{document}
\title[Quantitative $\Gamma$-Convergence]{Quantitative $\Gamma$-Convergence:
A Variational Extrapolation Framework}
\author{Oscar Dominguez}
\address{Oscar Dominguez\\
Departamento de Matem\'aticas \\
CUNEF Universidad \\
Madrid, Spain.}
\email{oscar.dominguez@cunef.edu}
\author{Mario Milman}
\address{Mario Milman \\
Instituto Argentino de Matem\'atica\\
Buenos Aires, Argentina.}
\email{mario.milman@icloud.com}
\urladdr{https://sites.google.com/site/mariomilman}
\thanks{O.D. is supported in part by the grants RYC2022-037402-I, PID2023-152878NB-I00 and CNS2025-166944 funded by the Spanish Ministry of Science and
Innovation.}

\begin{abstract}
This paper introduces a new extrapolation framework within the Calculus of
Variations. We develop a  robust quantitative theory of the   $\Gamma$-convergence approach of De Giorgi by combining it  with  Gagliardo's  work on relative completions in  interpolation theory. A central new tool
is the extrapolation of  \emph{variational $K$-functionals}, which associates to an arbitrary family   of
functionals a dynamic penalty method that simultaneously characterizes their $\Gamma$%
-limits and yields explicit rates of convergence for minimizers. As a
technical foundation we prove compactness extrapolation theorems also
involving Gagliardo completions. To illustrate the mechanism,  we apply the technique  to variational models arising 
in image processing, notably the Rudin--Osher--Fatemi (ROF) and
Chan--Esedoglu models. Our main result in this setting is a contribution to
the \emph{inverse regularity problem} associated to discrete approximations of the ROF model: we show that the required regularity of the observed image  that guarantees a prescribed optimal rate of convergence is measured in terms of the  decay of  
variational $K$-functionals. Surprisingly,  our characterization is
explicitly captured by using the recently introduced spaces  of  Brezis, Van
Schaftingen, and Yung. Moreover, our methods unify, and considerably extend, classical
pointwise convergence and $\Gamma$-versions of the
Bourgain--Brezis--Mironescu and Maz'ya--Shaposhnikova formulae, as well as related 
results in a variety of settings that involve BMO-type functionals, heat kernels, Sobolev spaces on metric measure 
spaces, ...  By bridging De Giorgi's variational
convergence with Gagliardo's completion through variational $K$-functionals
and extrapolation of compactness, this work establishes a quantitative
approach to the Calculus of Variations.
\end{abstract}

\maketitle
\tableofcontents

\section{Introduction}

\subsection{Overview}

The derivation of quantitative rates of convergence for minimizers of energy
functionals remains a limitation of De Giorgi's theory of $\Gamma $%
-convergence. While $\Gamma$-convergence guarantees stability of minimizers
and convergence of energies, it does not intrinsically supply rates of
approximation. Standard tools from interpolation theory, built around the
classical $K$-functional between two fixed normed spaces, are not directly suited
to sequences of functionals in which the penalization itself evolves with
the approximation parameter.

In this paper, we address this limitation by developing a quantitative theory
that merges De Giorgi's variational framework with Gagliardo's theory of
relative completions. The central device is the \textit{variational }$K$-%
\textit{functional,} which replaces the two fixed norms of the classical
functional by a sequence of functionals (not necessarily norms). This construction yields a novel  penalty
characterization of $\Gamma$-limits that is both dynamic and quantitative.

The development of these tools allows us to reformulate $\Gamma$-convergence
within the framework of Gagliardo relative completion from interpolation
theory. In our setting, the two theories merge naturally, allowing the
description of the $\Gamma$-convergence of sequences of functionals in terms
of their \textit{generalized relative Gagliardo completions}. Armed with this
equivalence \ we are able to develop new results on extrapolation of
compactness, analyze rates of convergence of variational limits and bridge
the gap between qualitative stability and quantitative approximation in
variational analysis.

Next, we provide an atlas of the contents of the paper. Our presentation here
deviates from the order of the paper: We start by  introducing the applications to
image processing and then outline the contents of the preceding\footnote{%
Here we follow Gabriel Garc\'ia M\'{a}rquez's novella \textit{Chronicle of a
Death Foretold, }(G. Rabassa, Trans.), 2003 Vintage Books, that starts with
a murder and the plot is its reconstruction.} sections 
 with  the supporting
foundational theoretical material, as well as many applications that
illustrate the scope of our theory.

\subsection{Image processing: the inverse regularity problem}\label{Section1.2}

As a preview, and somewhat informally (we refer to Section \ref{Section9.2}
for formal definitions and notation), we present an application of our
theory to the inverse regularity problem for the celebrated Rudin--Osher--Fatemi (ROF) model \cite{ROF}  in image denoising. Recall that, for a noisy image $f \in L^2(\T^2)$, the ROF model relies on the functional  
\begin{equation}\label{ROFIntro}
E(u):=\frac{1}{2}\Vert f-u\Vert_{L^{2}}^{2}+\text{TV}(u),
\end{equation}
where $\text{TV}$ refers to the (anisotropic) total variation (cf. \eqref{PF18.2}). For each $n \in \N$, let $E_n$ be the approximation of  $E$ associated to a suitable discretization of $\text{TV}$  defined on the space of piecewise constant functions on an $n \times n$ grid  (cf.  (\ref{PF18.6})). 
Let $\bar{u}$ and $\bar{u}_n$ denote the unique minimizers of $E$ and $E_n$ with corresponding values 
\begin{equation*}
 \bar{E}:= E(\bar{u}), \qquad \bar{E}_n := E_n(\bar{u}_n).
 \end{equation*}
Then   we show that  optimal rates of convergence of both $\bar{E}_n \to \bar{E}$ and $\bar{u}
_{n}\rightarrow\bar{u}$  are controlled by a variational $K$-functional of the
observed data $f$, which in turn can be estimated  by a classical $K$-functional (cf. \eqref{eq:classical_K_intro} below). To be more precise, we establish the following result. 

\begin{theorem}\label{Thm1Intro}
	Assume that\footnote{%
The assumption $p\leq2$ is quite natural. Otherwise, in view of the
classical embedding $\dot{W}^{1,p}\subset C$ if $p>2$, we would impose continuity on the observed image, which is not very interesting.} $p\in(1,2]$ and $\gamma = 2-\frac{2}{p}$. Then there exists a constant $C > 0$, which is independent of $f$, such that 
\begin{equation}
|\bar{E}-\bar{E}_n|^{\frac{1}{2}} + \Vert\bar{u}-\bar{u}_{n}\Vert_{L^{2}}\leq C K(n^{-\gamma},f;L^{2},\dot {W}%
^{1,p})  \label{enlaprimera}
\end{equation}
for every $n \in \mathbb{N}$. In addition, if\footnote{In light of \cite[Proposition 1 and Remark 3]{BIM},  the assumption $\bar{u} \in L^\infty$ is automatically satisfied for a general datum $f \in L^2$. However, it is not possible to provide an estimate of type $\|\bar{u}\|_{L^\infty} \leq c \|f \|_{L^2}$.} $\bar{u} \in \dot{\text{BV}} \cap L^\infty$ then
\begin{equation}\label{29.9.1}
	\Vert\bar{u}-\bar{u}_{n}\Vert_{L^{2}}^2 \leq  C n^{-1/2} K(n^{-\gamma},f;L^{2},\dot {W}%
^{1,p}) (\|\bar{u}\|_{L^\infty} \emph{TV}(\bar{u}))^{1/2}.  
\end{equation}
\end{theorem}

  In particular, Theorem \ref{Thm1Intro} allows us to solve the following basic question: Given any  $\alpha\in(0,1),$
what is the required regularity of the observed image $f$ that guarantees the prescribed rate of convergence
\begin{equation}\label{31.8.1}
\Vert\bar{u}-\bar{u}_{n}\Vert_{L^{2}} = \mathcal{O} (n^{-\alpha})? 
\end{equation}
In view of (\ref{enlaprimera}) and the definition of the real interpolation
spaces for $\theta\in(0,1)$ (cf. \eqref{Deflim} below) 
\begin{equation*}
f\in(L^{2},\dot{W}^{1,p})_{\theta,\infty}\iff  K(t,f;L^{2},%
\dot{W}^{1,p}) = \mathcal{O} ( t^{\theta}), \qquad \forall t > 0, 
\end{equation*}
we see that the desired rate of approximation of the minimizers $%
\mathcal{O}(n^{-\alpha })$ with $\alpha\in(0,\gamma)$ is achieved provided that %
\begin{equation}
f\in(L^{2},\dot{W}^{1,p})_{\frac{\alpha}{\gamma},\infty}.   \label{lasegunda}
\end{equation}
Hence the above question can be naturally reduced to computing the interpolation spaces in \eqref{lasegunda}. In the Hilbert setting $p=2$ (and so $\gamma=1$), this can be done in a simple way: Using the well-known fact that 
\begin{equation}\label{2.9.2}
	K(t, f; L^2, \dot{W}^{1, 2}) \approx \omega(f, t)_2,  \qquad \forall t > 0, 
\end{equation}
where\footnote{Given two nonnegative quantities $A$ and $B$, the standard notation $A \lesssim B$ means that there exists a constant $c$, independent of all essential parameters, such that $A \leq c B$. We write $A \approx B$ if $A \lesssim B \lesssim A$.}   $\omega(f, t)_2$ denotes  the classical modulus of smoothness of $f \in L^2$ (cf. \eqref{15.9.1}), then 
\begin{equation}\label{2.9.4}
(L^{2},\dot {W}%
^{1,2})_{\alpha,\infty}=\dot{B}_{2,\infty}^{\alpha} \qquad \text{(Besov spaces)}.
\end{equation}
As a combination of \eqref{enlaprimera} and \eqref{2.9.2}, we obtain 
\begin{equation*}
	\Vert\bar{u}-\bar{u}_{n}\Vert_{L^{2}}  \lesssim \omega(f, n^{-1})_2
\end{equation*}
for every $n \in \mathbb{N}$ and $f \in L^2$. In particular, if $f \in \dot{B}^\alpha_{2, \infty}$ then  (cf. Theorem \ref{Coro8.32})
\begin{equation}\label{27.9.4}
	\Vert\bar{u}-\bar{u}_{n}\Vert_{L^{2}}  = \begin{cases}
	
	 \mathcal{O} (n^{-\alpha})& \text{if} \qquad \alpha \in [\frac{1}{2}, 1), \\
	 \mathcal{O}(n^{-\frac{2 \alpha +1}{4}}) & \text{if} \qquad \alpha \in (0, \frac{1}{2}), 
	 \end{cases}
\end{equation}
 and if $f \in L^\infty \cap \dot{BV}$ then 
\begin{equation}\label{7.9.2}
	\Vert\bar{u}-\bar{u}_{n}\Vert_{L^{2}}  \lesssim n^{-1/2} \|f\|_{L^\infty}^{1/2} \text{TV}(f)^{1/2};
\end{equation}
 cf. Corollary \ref{Thm191.} for a sharper version in terms of fractional Sobolev spaces. Furthermore, the decay expressed in \eqref{27.9.4} is optimal for every $\alpha \in (0, 1)$.  We mention that the optimal decay $\mathcal{O}(n^{-1/2})$ for $f \in L^\infty$ given by \eqref{7.9.2}  has been recently achieved by Chambolle--Pock \cite{CP} and Caillaud--Chambolle  \cite{CC} in the more delicate isotropic setting and under an additional technical assumption on existence of Lipschitz dual fields associated to \eqref{ROFIntro}; cf. also \cite{LLW, WL} for first contributions along this line of research. Concerning the anisotropic setting  based on approximation by continuous piecewise bilinear functions,  the estimate \eqref{7.9.2} was obtained in \cite{BNS}. Working with the anisotropic functionals $E_n$ in Theorem \ref{Thm1Intro}, a very recent contribution by Kirisits and Setterqvist \cite{KSe} gives the rate of convergence $\|\bar{u}-\bar{u}_n\|_{L^2} = \mathcal{O}(n^{- \frac{q-2}{2(q-1)} })$ under  $f \in L^q, q \geq 2$. 

However, the above approach does not work anymore outside the Hilbert setting (i.e., working with $(L^2, \dot{W}^{1, p})$ with $p \neq 2$) due to the lack of a suitable  analog of \eqref{2.9.2}. Despite this obstruction, as already pointed out,    the decay of $K(t, f; L^2, \dot{W}^{1, p})$ as $t \to 0$ may be captured by the interpolation norm associated to  \eqref{lasegunda}. Then the key ingredient is a recent observation obtained in  \cite{DSSVSY}, where the authors found a characterization of the desired interpolation spaces in terms of  the new 
scale of function spaces $\dot{B}_{q}^{\beta}(\gamma,\infty)$ introduced  by Brezis, Van Schaftingen, and Yung \cite{BVY, BVY2}  via  weak-type variants of classical Gagliardo (semi)norms; cf. \eqref{BVSY} below for the precise definition. Armed now with such a tool, we are able to give an  answer to the inverse regularity problem in image processing and show that  Brezis--Van Schaftingen--Yung spaces provide a natural framework to obtain error bounds for discretization of the total variation in the ROF denoising model. 

\begin{theorem}\label{Thm2Intro} Assume that $p \in(1, 2)$ and $\alpha\in\big(0, 2-\frac{2}{p%
} \big)$. Let 
\begin{equation*}
\beta= \frac{\alpha}{2-\frac{2}{p}}, \qquad\frac{1}{q} = \frac{1-\beta}{2} + 
\frac{\beta}{p}, \qquad\frac{1}{\gamma} = \frac{1}{2}-\frac{1}{p}. 
\end{equation*}
If $f \in\dot{B}^{\beta}_{q}(\gamma, \infty)$ then 
\begin{equation}  \label{25.3Nov}
\|\bar{u}-\bar{u}_{n}\|_{L^{2}} \lesssim \begin{cases} n^{-\alpha} \|f\|_{\dot{B}^{\beta}_{q}(\gamma, \infty)} & \qquad \text{if}  \qquad \alpha \geq \frac{1}{2}, \\
n^{-\frac{2 \alpha +1}{4}}   \|f\|_{\dot{B}^{\beta}_{q}(\gamma, \infty)}^{1/2} (\|\bar{u}\|_{L^\infty} \emph{TV}(\bar{u}))^{1/4} & \qquad \text{if} \qquad \alpha < \frac{1}{2}, 
\end{cases}
\end{equation}
for all $n \in \mathbb{N}.$ In particular, if $f \in\dot{W}^{\beta, q}$ then \eqref{25.3Nov} holds.
\end{theorem}

\begin{remark}
Since the appearance of the works \cite{BVY, BVY2, BSVSY}, the spaces $\dot{B}^\beta_q(\gamma, \infty)$ have attracted  a great deal of attention due to several important reasons; cf. the survey paper \cite{BSVSY2}.  
We should also mention   the earlier work of Bourgain and Nguyen  \cite{BoN, N2}, where a  distinguished  case of parameters was already considered. In particular,  it was shown as a combination of  \cite{BoN, N2, BVY, GY, BSVSY, BSVSY2} that these spaces provide     derivative-free characterizations of standard Sobolev (semi)norms, more precisely, for $\gamma \neq 0$, 
\begin{equation}\label{2.9.3}
\dot{B}^{\beta}_{q}(\gamma,\infty) (\R^N)
=
\begin{cases}
\dot{W}^{1,q}(\R^N) & \qquad  \text{if} \qquad  \beta=1, \qquad q \in (1, \infty), \\
\dot{BV}(\R^N) & \qquad \text{if} \qquad \beta = q = 1, \qquad \gamma \in \R \backslash [-1, 0], \\
L^q(\R^N) &\qquad  \text{if} \qquad  \beta=0, \qquad q \in [1, \infty). 
\end{cases}
\end{equation}
Note that, for the purposes of this work (cf. Theorem \ref{Thm2Intro}), the most interesting case concerns the fractional regime $\beta \in (0, 1)$. In sharp contrast to the  case $\beta = 0, 1$ exhibited in \eqref{2.9.3}, the spaces $\dot{B}^{\beta}_{q}(\gamma,\infty), \beta \in (0, 1),$ are not part of  classical scales of function spaces  such as fractional Sobolev spaces, Besov spaces or Triebel--Lizorkin spaces, but they constitute an independent scale; cf. \cite{DSSVSY}. 
\end{remark}

\begin{remark}
 The borderline case $p=2$ in Theorem \ref{Thm2Intro}  refers (at least,
formally) to  the Besov setting $f \in \dot{B}^\alpha_{2, \infty}$, cf. \eqref{2.9.4}.  
\end{remark}

Next we explain the different components of our method and indicate where to
find the results in the paper.

\subsection{Gagliardo completions and extrapolation}

The notion of Gagliardo relative completions was introduced by Gagliardo 
\cite{GG} in the context of his foundational work on interpolation theory.
While the original definition involves the convergence of sequences with
specific properties, over time an equivalent definition, using $K$%
-functionals, became more popular and is now prevalent in the literature.
Curiously, it is precisely the interplay between these two approaches that
was decisive in our research.

Given a compatible\footnote{Loosely speaking, compatible means that $A_0 + A_1$ makes sense.} pair of  normed\footnote{We shall also meet cases where $A_1$ is only equipped with a semi-norm.}  spaces $(A_{0},A_{1}),$ the \emph{$K$-functional}
is defined by a penalty method that plays one norm against the other. For $%
f\in A_{0}+A_{1}$ and $t>0,$ we let%
\begin{equation}
K(t,f;A_{0},A_{1}):=\inf\{\left\Vert f_{0}\right\Vert _{A_{0}}+t\left\Vert
f_{1}\right\Vert _{A_{1}}:f=f_{0}+f_{1},f_{i}\in A_{i},i=0,1\}. 
\label{eq:classical_K_intro}
\end{equation}
The \emph{Gagliardo (relative) completions} are defined by the conditions (cf.  
\cite[p. 295]{BS} and \cite[p. 296]{BK}): For $f \in A_0 + A_1$ and $i = 0, 1$, 
\begin{equation}
\left\Vert f\right\Vert _{G_{i}(A_{0},A_{1})}:=\sup_{t>0}\frac{%
K(t,f;A_{0},A_{1})}{t^{i}}.   \label{sec:ante}
\end{equation}
Note that, by monotonicity properties of $K$-functionals,  
\begin{equation}\label{3.9.1}
	\|f\|_{G_0(A_0, A_1)} = \lim_{t \to \infty} K(t, f; A_0, A_1) \qquad \text{and} \qquad \|f\|_{G_1(A_0, A_1)} = \lim_{t \to 0} \frac{K(t, f; A_0, A_1)}{t}. 
\end{equation}
Let us denote by  $\|
\cdot \|_{(A_{0},A_{1})_{\theta,p}^{\blacktriangleleft}}$  the (normalized) interpolation norms with $\theta \in (0, 1)$ and $p \in [1, \infty]$; cf. \eqref{2.9.5}.  

The connection of Gagliardo's relative completions to extrapolation theory
can be seen from the  pointwise limits obtained in \cite{M-Lim, KMX}   (more precisely, their slight  refinement given in Proposition \ref{prop:lim}). Namely, if $f \in A_{i^\ast}$ with $i^\ast = 1-i, i=0, 1$, then 
\begin{equation}
\lim_{\theta\rightarrow i}\Vert
f\Vert_{(A_{0},A_{1})_{\theta,p}^{\blacktriangleleft}}=\Vert
f\Vert_{G_{i}(A_{0},A_{1})}. 
\label{gextra}
\end{equation}
In particular, applying \eqref{gextra} to the specific pair $(L^{p}(\R^N),\dot{W}^{1, p}(\R^N)),$ $p\in (1, \infty)$, leads to new approaches to asymptotic versions of 
the celebrated  ``BBM-formula" by Bourgain--Brezis--Mironescu \cite{BBM2} and ``MS-formula" by  Maz'ya--Shaposhnikova \cite{MS}
\begin{equation}\label{14.9.4}
\lim_{s\rightarrow1^-}(1-s)^{1/p}\left\Vert f\right\Vert _{\dot{W}^{s,p}(\mathbb{R}%
^{N})}=C_{N,p}\left\Vert \nabla f\right\Vert _{L^{p}(\mathbb{R}^{N})},
\end{equation}%
\begin{equation*}
\lim_{s\rightarrow0^+}s^{1/p}\left\Vert f\right\Vert _{\dot{W}^{s,p}(\mathbb{R}%
^{N})}=c_{N,p}\left\Vert f\right\Vert _{L^{p}(\mathbb{R}^{N})}, 
\end{equation*}
respectively. The limit case $p=1$ also holds in the above formulae provided that $\|\nabla f\|_{L^p(\R^N)}$ in \eqref{14.9.4} is replaced by $\|f\|_{\dot{BV}(\R^N)}$, cf. \cite{Dav}. 

The $\Gamma$-convergence version of the BBM  formula was obtained by Ponce \cite{P} but, prior to our work, $\Gamma$-convergence has not
been studied using extrapolation methods.

We were led to the notion of variational $K$-functional when we discovered a
special case of the connection between De Giorgi's theory and Gagliardo's
relative completions.

\subsection{A new key idea: the connection De Giorgi--Gagliardo}

De Giorgi's $\Gamma$-convergence \cite{DG} deals with limits of functionals
defined on a normed\footnote{$\Gamma$-convergence can be formulated in the more general setting of metric spaces. However, for the purposes of this paper, it is enough to focus on normed spaces.} space $X$. Given a sequence  $%
\{F_{n}\}_{n\in \mathbb{N}}$ of functionals  $F_{n}:X\longrightarrow [0, +\infty],$
its  lower and upper $\Gamma$-limits are defined  
for each $f\in X$ by 
\begin{equation}
\Gamma_X-\liminf_{n\rightarrow\infty}F_{n}(f):=\inf\Big \{\liminf_{n\rightarrow
\infty}F_{n}(f_{n}):\{f_{n}\}_{n\in\mathbb{N}}\subset X\text{ with }%
\lim_{n\rightarrow\infty}\left\Vert f-f_n\right\Vert _{X}=0\Big\}, 
\label{intro1}
\end{equation}%
\begin{equation}
\Gamma_X-\limsup_{n\rightarrow\infty}F_{n}(f):=\inf\Big\{\limsup_{n\rightarrow
\infty}F_{n}(f_{n}):\{f_{n}\}_{n\in\mathbb{N}}\subset X\text{ with }%
\lim_{n\rightarrow\infty}\left\Vert f-f_n\right\Vert _{X}=0\Big\}. 
\label{intro2}
\end{equation}
Then $\{F_{n}\}_{n\in \mathbb{N}}$ is $\Gamma$-convergent if these two
limits are equal for all $f\in X,$ in which case we let $F(f)$ be the
functional that computes their common value at $f$ and write $\Gamma_X-\lim_{n \to \infty} F_{n} = F.$ For a detailed study of $\Gamma$-convergence, we refer to the monographs \cite{Bra, DALM}. 

On the other hand, let us recall Gagliardo's original definition of relative completions,
framed in terms of convergence of suitable sequences (cf. \cite[Lemma 1]{C})
\begin{equation}\label{2.9.7}
\left\Vert f\right\Vert _{G_{i}(A_{0},A_{1})}=\inf\Big\{\sup_{n%
\in \mathbb{N}}\left\Vert f_{n}\right\Vert _{A_{i}}:\{f_{n}\}_{n \in \mathbb{N}}\subset
A_{i}\text{ with } \lim_{n \to \infty} \left\Vert f-f_n\right\Vert _{A_{i^\ast}} = 0%
\Big\} 
\end{equation}
for $f \in A_0 + A_1$ and $i=0, 1$. Now, it is an easy exercise to check that
\begin{align}
\left\Vert f\right\Vert _{G_{i}(A_{0},A_{1})} & =\inf\Big\{\limsup_{n%
\to \infty}\left\Vert f_{n}\right\Vert _{A_{i}}:\{f_{n}\}_{n \in \mathbb{N}}\subset
A_{i}\text{ with } \lim_{n \to \infty} \left\Vert f-f_n\right\Vert _{A_{i^\ast}} = 0%
\Big\} \nonumber \\
& =  \inf\Big\{\liminf_{n%
\to \infty}\left\Vert f_{n}\right\Vert _{A_{i}}:\{f_{n}\}_{n \in \mathbb{N}}\subset
A_{i}\text{ with } \lim_{n \to \infty} \left\Vert f-f_n\right\Vert _{A_{i^\ast}} = 0%
\Big\}.    \label{intro3}
\end{align}
Although innocent-looking, the reformulation of the original definition  \eqref{2.9.7} given by \eqref{intro3} enables a direct comparison with \eqref{intro2}. However, there is still a technical point to be clarified, namely, there are three different norms (associated to $A_0, A_1$ and $A_0 + A_1$) playing simultaneously  in \eqref{intro3}, unlike in \eqref{intro1} where only one norm (associated to $X$) is involved.  This issue can be circumvented as follows: Let $\left\{
F_{n}^{i}\right\} _{n\in \mathbb{N}}$ be the (constant) sequence of 
functionals defined on $A_{i^\ast}$ by 
\begin{equation*}
F_{n}^{i}(f)=%
\begin{cases}
\left\Vert f\right\Vert _{A_{i}}, & f\in A_{i} \cap A_{i^\ast}, \\ 
+\infty, & f \in A_{i^\ast} \backslash A_i,%
\end{cases}   
\end{equation*}
for every $n \in \mathbb{N}$.    It can be easily seen  from (\ref{intro1}), \eqref{intro2} and (\ref{intro3}) that%
\begin{equation}
\Gamma_{A_{i^\ast}}-\lim_{n\rightarrow\infty}F_{n}^{i}(f)=\left\Vert f\right\Vert
_{G_{i}(A_{0},A_{1})},\qquad f\in A_{i^\ast}.   \label{aextender}
\end{equation}
In particular, this result  tells us that  the
topology of Gagliardo completions fits into the variational setting of $\Gamma$-convergence. Moreover, this  suggests that there could be a deeper structural link between
Gagliardo's extrapolation and variational limits.

\begin{problem}
\label{core_problem} Is it possible in general to characterize De Giorgi $%
\Gamma$-convergence through a parameterized penalty method? Specifically,
does there exist a dynamic variant of the $K$-functional capable of
recovering $\Gamma$-limits while simultaneously providing a mechanism to
extract explicit rates of convergence for minimizers?
\end{problem}

The classical $K$-functional is too rigid to resolve  Problem \ref%
{core_problem}; our solution instead relies on letting the penalty evolve
dynamically.

\subsection{$\Gamma$-convergence as Gagliardo extrapolation  of variational $K$-functionals}

Concerning the derivation of (\ref{aextender}) note that the sequence of
functionals, whose $\Gamma$-limit we seek, is part of the $K$-functional
that we control. This informs our constructions. The primary tool introduced
in this paper is a  family of $K$-functionals constructed by
replacing one of the two endpoint spaces of the classical definition (cf. \eqref{eq:classical_K_intro}) by a
sequence of  functionals. Given a sequence  $%
\{F_{n}\}_{n\in \mathbb{N}}$ of functionals  $F_n : X \longrightarrow [0, + \infty],$ for each  $t > 0$ and $f\in X,$ we let 
\begin{equation}
K(t,f;X,F_{n}):=\inf_{g\in X}\{\left\Vert f-g\right\Vert _{X}+tF_{n}(g)\}. 
\label{varK_intro}
\end{equation}
This functional acts as a dynamic extension of the classical $K$-functional,
replacing fixed target norms with an evolving sequence of functional
penalties. Although the definition of variational $K$-functional given in \eqref{varK_intro} mimics in an obvious way the classical construction of $K$-functional, we stress that many fundamental properties of the latter strongly rely on norm properties and then the absence of norms makes the analysis of variational $K$-functionals more delicate.   The basic properties of variational $K$-functionals and their connection with $\Gamma$-convergence are
developed in Section \ref{section:GagliardoGamma}. In particular, we show that $\Gamma$-convergence can be understood as the Gagliardo completion of variational $K$-functionals, providing a satisfactory answer to   Problem \ref{core_problem}.

\begin{theorem}
\label{teoarriba}
Let $\{F_{n}\}_{n\in\mathbb{N}}$ be a sequence of
functionals on $X$. Then  for all $f\in X$,
\begin{align*}
\sup_{t>0}\liminf_{n\rightarrow\infty}\frac{K(t,f;X,F_{n})}{t} & \leq
\Gamma_X-\liminf_{n\rightarrow\infty}F_{n}(f) \\
& \hspace{-3cm}\leq\Gamma_X-\limsup_{n\rightarrow\infty}F_{n}(f)\leq
\inf_{\{t_n\}_{n \in \mathbb{N}}}\limsup_{n\rightarrow\infty} \frac{K(t_n,f;X,F_{n})}{t_n},
\end{align*}
where the infimum runs over all infinitesimals $\{t_n\}_{n \in \mathbb{N}} \subset (0, \infty)$.  
\end{theorem}

As a consequence, $\Gamma$-convergence corresponds to the extrapolation characterization  \eqref{3.9.1} of the Gagliardo completion $G_1$ when the role of classical $K$-functionals is replaced by variational $K$-functionals. Hence, when combined with \eqref{aextender}, Theorem \ref{teoarriba} provides a new convergence approach that encompasses both De Giorgi and Gagliardo theories.  Remarkably,  one of the main advantages of this point of view is that  $\Gamma$-convergence can now be put into the powerful  quantitative setting provided by $K$-functionals, cf. Section \ref{Section1.2}.

\subsection{Variational $K$-functionals and extrapolation of $\Gamma$-convergence}

We now illustrate how the variational $K$-functional can be applied to obtain $%
\Gamma$-convergence results. Here we shall consider in detail the
case when $\theta\rightarrow1^{-},$ the case $\theta\rightarrow0^{+}$ will
then follow by symmetry. Next we show the $\Gamma$-version of (\ref{gextra}) as a simple application of Theorem \ref{teoarriba}. 

\begin{theorem}\label{ThmExtrapolIntro}
Let $p \in [1, \infty]$. If $f \in A_0$ then\footnote{As usual, we adopt the convention that the formula \eqref{9.9.4} for $f \in A_0 \backslash G_1(A_0, A_1)$ reads as $\Gamma_{A_0}-\lim_{\theta\to 1^-} \left\Vert f\right\Vert _{(A_{0},A_{1})_{\theta
,p}^{\blacktriangleleft}} = + \infty$.} 
\begin{equation}\label{9.9.4}
\Gamma_{A_0}-\lim_{\theta\rightarrow1^{-}}\Vert
f\Vert_{(A_{0},A_{1})_{\theta,p}^{\blacktriangleleft}}=\left\Vert
f\right\Vert _{G_{1}(A_{0},A_{1})}. 
\end{equation}
\end{theorem}
\begin{proof}
Consider the 
family of functionals $F_\theta, \theta \in (0, 1)$, defined on $A_0$ given by 
\begin{equation}\label{9.9.1}
F_{\theta}(f):= \begin{cases}
 \|f\|_{(A_0, A_1)_{\theta, p}^\blacktriangleleft} & \text{if} \qquad  f\in A_0 \cap (A_0, A_1)_{\theta, p}, \\ 
+\infty & \text{if} \qquad  f \in A_0 \backslash (A_0, A_1)_{\theta, p}. %
\end{cases} 
\end{equation}
  To implement the variational $K$-functional method we
need two elementary inequalities (see the proof of Theorem \ref{Thm1}  
below for full details): If $f \in A_0$ then 
\begin{equation}
K(t^{1/\theta},f;A_{0},A_{1})\leq K(t,f;A_{0},F_{\theta}), \qquad \forall t > 0, 
\label{elmal}
\end{equation}
and%
\begin{equation}
\limsup_{n\rightarrow\infty}nK\Big(\frac{1}{n},f; A_{0},F_{\theta_{n}} \Big) \leq\left%
\Vert f\right\Vert _{G_{1}(A_{0},A_{1})},   \label{elbien}
\end{equation}
where $\theta_n \to 1^{-}$. 
It readily follows that 
\begin{align*}
	\|f\|_{G_1(A_0, A_1)} &= \sup_{t \in (0, \infty)} \frac{K(t, f; A_0, A_1)}{t} \\
	& = \sup_{t \in (0, \infty)} \liminf_{\theta \to 1^-}  \frac{K(t^{\frac{1}{\theta}}, f; A_0, A_1)}{t^{\frac{1}{\theta}}} \\
	& \leq \sup_{t \in (0, \infty)} \bigg( \liminf_{\theta \to 1^-} \frac{K(t, f; A_0, F_\theta)}{t} \bigg) \bigg(\liminf_{\theta \to 1^-} t^{1-\frac{1}{\theta}} \bigg) \qquad \text{(by \eqref{elmal})} \\
	& \leq \Gamma_{A_0}-\liminf_{\theta \to 1^-} F_\theta(f) \qquad \text{(by Theorem \ref{teoarriba})} \\
	& \leq \Gamma_{A_0}-\limsup_{\theta \to 1^-} F_\theta(f) \\
	& \leq \limsup_{n\rightarrow\infty}nK\Big(\frac{1}{n}; f,A_{0},F_{\theta_{n}} \Big) \qquad \text{\Big(by Theorem \ref{teoarriba} with $t_n = \frac{1}{n}$ \Big)}   \\
	& \leq \|f\|_{G_1(A_0, A_1)} \qquad \text{(by \eqref{elbien})}. 
\end{align*}
\end{proof}

We refer to Sections \ref{section:GagliardoGamma} and \ref%
{SectionGammaExtrapol} for full development.

\subsection{Some applications of extrapolation of $\Gamma$-convergence: a unified perspective to BBM and MS formulae}

We illustrate the scope of our theory by considering in detail examples in
different contexts for which our methodology nevertheless produces a unified
perspective. In
Section \ref{sec:app&exples} we apply Theorem \ref{ThmExtrapolIntro} to
many familiar pairs of function spaces. In particular, we derive  
$\Gamma$-convergence results for:
\begin{itemize}
	\item BBM formulae for general Sobolev spaces built on rearrangement invariant spaces\footnote{Loosely speaking, Banach spaces of functions whose  norms depend only on the size of functions. This class contains many customary spaces such as Lebesgue spaces, Lorentz spaces, Orlicz spaces, Lorentz--Zygmund spaces, ...  We refer the reader to \cite{BS} for a detailed account.} on $\R^N$, thus considerably extending Ponce's approach  \cite{P};
	\item BBM formulae for Sobolev spaces on metric measure spaces. In particular, we show that the recent  approach to Sobolev spaces on Cheeger spaces proposed in  \cite{ARB}  follows from our abstract theory; 
	\item BBM formulae associated to non-local operators such as the fractional Laplacian via Butzer seminorms;  
	\item convergence of BMO-type functionals related to the total variation, cf. \cite{ARBDN}; 
	\item extension of the MS formulae obtained in  \cite{DLNP} (restricted to characteristic functions) and \cite{CDLKNP} (associated to $L^2$) to  any rearrangement invariant norm on   $\R^N$;
	\item  heat-content type energies related to the total variation and Sobolev seminorms, as recently investigated in \cite{GS}.  
\end{itemize}

The implementation of our method follows a unified perspective: while the
technicalities in each case can be very different, the starting procedure
and the tasks to perform to arrive at the results always call for a careful
choice of a pair of spaces and   control of the corresponding $K$-functionals.

\subsection{Computing exact limit formulae}

A glaring drawback of the extrapolation methodology is that  $K$-functionals can rarely be computed exactly and for most pairs of spaces\ we only
have equivalent expressions that are up to multiplicative constants. This is
reflected in the corresponding asymptotic formulae, which we have formulated in
terms of the $K$-functionals and therefore the results obtained are not exact formulae, but only up
to equivalence constants. However,  we show that in most of the settings there
is a conceptually simple fix. Informally, when we
arrive at expressions that are within constants of the $K$-functional we
seek, we can work directly with those expressions and prove the
corresponding limit theorems for them and normalize them to achieve exact
limits. This is developed in Section \ref{SectionExact} where we
explain the method in detail and illustrate it with  examples. 

\subsection{Equi-coercivity via extrapolation of compactness}

It is well known that equi-coercivity is a key property paired with $\Gamma$-convergence to ensure stability of minimizers. In Section \ref{SectionCompactness}, we show that equi-coercivity is also part of the general extrapolation theory. In particular, we obtain the following extrapolation of compactness.

%
%
%
%
%

\begin{theorem}\label{ThmCwiGen}
Suppose that $(A_0, A_1)$ is a pair of Banach spaces such that\footnote{The symbol $\hookrightarrow$ stands for the continuous embedding.}  
\begin{equation*}
	A_1 \hookrightarrow A_0 \qquad \text{is compact}.
\end{equation*}
Let $p\in\lbrack1,\infty
],\{\theta_{n}\}_{n \in \mathbb{N}}\subset(0,1)$ with $\lim_{n\rightarrow\infty}\theta_{n}=1$,
and let $\{u_{n}\}_{n \in \mathbb{N}}$ be a sequence such that%
\begin{equation*}
\liminf_{n\to \infty}\left\Vert u_{n}\right\Vert
_{(A_{0},A_{1})_{\theta_{n},p}^{\blacktriangleleft}}<\infty. 
\end{equation*}
Then there exists $u\in A_{0}$ such that $%
u_{n}\rightarrow u$ in $A_{0}$ as $n\rightarrow\infty$ (up to a subsequence).
\end{theorem}

The proof of Theorem \ref{ThmCwiGen}  combines Cwikel's one sided
compactness theorem\ \cite{C2} with special
reiteration formulae in extrapolation theory (cf. \cite{KM}). Then we show that many  coercivity results in different settings that have been treated independently in the  recent literature (to mention some of them, cf. \cite{BBM2, P, AABPT, ARBDN, GS, ARB}) follow immediately from Theorem \ref{ThmCwiGen} when the pair $(A_0, A_1)$ is specified, thereby providing  a unified perspective. See Section \ref{Section7.1} for further details.

\subsection{Variational extrapolation} 

In Section \ref{Section8}, we show stability of minimizers in variational analysis via extrapolation and Gagliardo completions. To be more precise, we establish the following continuity result for interpolated functionals defined on the rather general setting of Banach function spaces (cf. \cite[Chapter 1]{BS}).

\begin{theorem}\label{ThmVarIntro}
Let $B$ and $A_0$ be Banach function spaces over a common finite measure space such that $B \hookrightarrow A_0$ and let $A_1$ be a Banach space satisfying  
\begin{equation}\label{23.9.1}
	A_0 \cap A_1 \hookrightarrow A_0 \qquad \text{is compact}.
\end{equation}
In addition, we assume that $B$ is uniformly convex. For each $f \in B, q, \lambda > 0,  p \in [1, \infty), \theta \in (0, 1)$, consider the interpolated functional given by
$$
	E_\theta (u):= \|f-u\|^q_B + \lambda \|u\|_{(A_0, A_1)_{\theta, p}^{\blacktriangleleft}}.  
$$
	If $\bar{u}_\theta$ are near-minimizers of $E_\theta$ then $\bar{u}_\theta \to \bar{u}$ in $B$ as $\theta \to 1^{-}$ (up to a subfamily), where $\bar{u}$ is a minimizer  of the Gagliardo functional
	\begin{equation}\label{23.9.4}
			E_1 (u):= \|f-u\|^q_B + \lambda \|u\|_{G_1(A_0, A_1)}.  
	\end{equation}
	Moreover
	\begin{equation}\label{20.9.9}
		E_1(\bar{u}) = \lim_{\theta \to 1^{-}} E_\theta(\bar{u}_\theta). 
	\end{equation}
\end{theorem}

\begin{remark}
	\begin{enumerate}
		\item[(i)] In view of applications to image processing (cf. Section \ref{Subsection1.11}  below), the assumption $B \hookrightarrow A_0$ seems to be quite natural.  Note that $A_1$ is  not necessarily a Banach function space in the sense of \cite{BS}. In particular, Sobolev spaces or more generally Besov--Triebel--Lizorkin spaces are admissible. 
		\item[(ii)] The uniform convexity of $B$  may be weakened to include non-uniformly convex spaces such as $L^1$. On the other hand, the fact that $B$ is uniformly
convex seems to be general enough to cover many classical spaces such as the 
$L^{p}$ spaces with $1<p<\infty$, Hilbert spaces, Lorentz spaces $ 
\Lambda^{p}(w)$ with $1<p<\infty$ and a regular weight $w$ (cf. \cite{H} for
further details), suitable Orlicz--Lorentz spaces, ....
	\end{enumerate}
\end{remark}

Basic ingredients in the proof of Theorem \ref{ThmVarIntro} are extrapolation of $\Gamma$-convergence (cf. Theorem \ref{ThmExtrapolIntro}) and extrapolation of compactness (cf. Theorem \ref{ThmCwiGen}).

\subsection{Convergence of filters in image processing: from non-local to local models}\label{Subsection1.11} 
Starting from the influential  works of Buades--Coll--Morel \cite{BCM} and Gilboa--Osher \cite{GO, GO2},   there has been an increasing interest in denoising models arising in image processing based on filters that approximate the local total variation in \eqref{ROFIntro} by certain non-local versions.  Then Aubert--Kornprobst \cite{AK} and Brezis--Nguyen \cite{BN} established the convergence of non-local to local filters.

As explained in Section  \ref{sec:asymp}, Theorem \ref{ThmVarIntro}    is partly motivated  by  convergence results of non-local  to local filters in image processing.   In our setting, the role of non-local filters is played by  the (normalized) interpolation functionals $\|\cdot\|_{(A_0, A_1)_{\theta, p}}^{\blacktriangleleft}$ and the local filter corresponds to the Gagliardo completion $G_1(A_0, A_1)$. To be more precise, an application of Theorem \ref{ThmVarIntro} related to  the special choice 
$$B = L^2, \qquad A_0 = L^1, \qquad A_1 = \dot{W}^{1, 1} \qquad \text{(on a fixed cube in $\R^N$)}$$
together with the important fact that 
$$
	G_1(L^1, \dot{W}^{1, 1}) = \dot{BV}
$$
give  $L^2$-stability of minimizers from the functionals (cf. \cite{AK, NO})
\begin{equation*}
	E_\theta(u) = \|f-u\|_{L^2}^2 + (1-\theta) \|u\|_{\dot{W}^{\theta, 1}}
\end{equation*}
to the classical ROF model  (cf. \eqref{ROFIntro})  
as $\theta \to 1^-$. On the other hand, the counterpart associated to  the choice 
$$B = L^1, \qquad A_0 = L^1, \qquad A_1 = \dot{W}^{1, 1},$$
in Theorem \ref{ThmVarIntro} gives the corresponding $L^1$-stability of minimizers of (cf. \cite{Be})
\begin{equation*}
	E_\theta(u) = \|f-u\|_{L^1} + (1-\theta) \|u\|_{\dot{W}^{\theta, 1}}
\end{equation*}
 to the Chan--Esedoglu denoising model \cite{CE} as $\theta \to 1^{-}$. See Corollary \ref{CoroIP1} for the precise statement.  As already mentioned,  these convergence results were  obtained in \cite{AK} by using completely different tools; cf. also \cite{BN} dealing with a more  challenging setting of both non-local and non-convex approximations (Yaroslavsky filters) of the total variation.  Accordingly,  Theorem \ref{ThmVarIntro} shows that  convergence results for filters in image processing are  part of a general phenomenon based on  extrapolation and Gagliardo completions.

\subsection{Brief interlude: historical context and references}

The celebrated limit formulae of Bourgain--Brezis--Mironescu \cite{BBM2} and
Maz'ya--Shaposhnikova \cite{MS} have had a major impact in research on Sobolev spaces 
and related topics. One can find further references in the books \cite{BM}
and \cite{Pbook}. The pointwise limits  of real interpolation norms as the
parameters approach the endpoints (cf. \eqref{gextra}) and their relevance to
deriving the BBM and MS formulae were apparently first
systematically studied in \cite{M-Lim}, and further developed in \cite{KMX} and 
\cite{DM}. Building on the foundational extrapolation theory of
Jawerth--Milman \cite{JM}, these works established a key bridge between
classical interpolation theory and endpoint behavior, opening the way for
further developments. Here we mention only a representative sample of
contributions in different directions: \cite{DLTYY}, \cite{PS}, 
\cite{Hurri}, \cite{DR}, \cite{Maa}, \cite{GAT} and the references therein. We hope that the new connections and tools developed in this paper will have
further impact in both the Calculus of Variations and Extrapolation theory.

\section{Interpolation of normed pairs and Gagliardo completions}

\label{section:prelim}

In this preliminary  section, we review some basic elements of  interpolation theory and the
properties of Gagliardo relative completions as an endpoint extrapolation
methodology.

Let $(A_{0},A_{1})$ be a compatible pair of normed spaces. One of the most popular approaches to construct intermediate spaces relative to $(A_0, A_1)$ relies on $K$-functionals (cf. \eqref{eq:classical_K_intro}). 
For $\theta\in(0,1)$ and $p\in\lbrack1,\infty)$, the \emph{(real) interpolation
space} $(A_{0},A_{1})_{\theta,p}$ consists of all elements $f\in A_{0}+A_{1}$
such that 
\begin{equation}\label{5.10.1} 
\Vert f\Vert_{(A_{0},A_{1})_{\theta,p}}:=\bigg(\int_{0}^{\infty
}[t^{-\theta}K(t,f;A_{0},A_{1})]^{p}\,\frac{dt}{t}\bigg)^{1/p}<\infty.
\end{equation}
When $p=\infty$, the space $(A_{0},A_{1})_{\theta,\infty}$ is equipped with
the norm 
\begin{equation}  \label{Deflim}
\Vert f\Vert_{(A_{0},A_{1})_{\theta,\infty}}:=\sup
_{t>0}\{t^{-\theta}K(t,f;A_{0},A_{1})\}.
\end{equation}

In extrapolation theory, it is standard to normalize these norms (cf. \cite%
{JM}). Let 
\begin{equation*}
c_{\theta,p}:=\left\{ 
\begin{array}{cc}
(\theta(1-\theta)p)^{\frac{1}{p}} &\qquad  \text{if} \qquad p<\infty, \\ 
1 &\qquad  \text{if} \qquad p=\infty,%
\end{array}
\right.   
\end{equation*}
and define the normalized norm by 
\begin{equation}\label{2.9.5}
\Vert f\Vert_{(A_{0},A_{1})_{\theta,p}^{\blacktriangleleft}}:=c_{\theta
,p}\Vert f\Vert_{(A_{0},A_{1})_{\theta,p}}. 
\end{equation}
By $(A_0, A_1)_{\theta, p}^{\blacktriangleleft}$ we mean  $(A_0, A_1)_{\theta, p}$ equipped with the norm \eqref{2.9.5}. Under this normalization, the following standard estimates hold (cf. \cite{JM, M}): 
\begin{equation}
\Vert f\Vert_{(A_{0},A_{1})_{\theta,\infty}}\leq\Vert
f\Vert_{(A_{0},A_{1})_{\theta,p}^{\blacktriangleleft}}\qquad\text{for} \qquad %
f\in(A_{0},A_{1})_{\theta,p},   \label{sec2:1}
\end{equation}%
\begin{equation}
\Vert f\Vert_{(A_{0},A_{1})_{\theta,p}^{\blacktriangleleft}}\leq\Vert
f\Vert_{A_{0}}^{1-\theta}\Vert f\Vert_{A_{1}}^{\theta}\qquad\text{for} \qquad f\in
A_{0}\cap A_{1}.   \label{sec2:2}
\end{equation}
Furthermore, the basic identity 
\begin{equation}\label{7.9.4}
K(t,f;A_{0},A_{0})=\min\{1,t\}\Vert f\Vert_{A_{0}}
\end{equation}
yields 
\begin{equation}
\Vert f\Vert_{(A_{0},A_{0})_{\theta,p}^{\blacktriangleleft}}=\Vert
f\Vert_{A_{0}}.   \label{sec2:3}
\end{equation}

The spaces corresponding to the limiting parameter values $\theta\in\{0,1\}$
are the Gagliardo relative completions $G_\theta(A_0, A_1)$, cf. \eqref{sec:ante}. 
It is easy to see that, for $i=0, 1$, 
\begin{equation}
A_{i}\hookrightarrow  G_{i}(A_{0},A_{1}) \qquad \text{with} \qquad \|f\|_{G_{i}(A_{0},A_{1})} \leq \|f\|_{A_i}. \label{sec:sera}
\end{equation}
Moreover, the Gagliardo pair preserves the underlying $K$-functional structure, namely, 
\begin{equation}
K(t,f;G_{0}(A_{0},A_{1}),G_{1}(A_{0},A_{1}))=K(t,f;A_{0},A_{1}), 
\label{sec2:5}
\end{equation}
and consequently, 
\begin{equation}
\|f\|_{\left( G_{0}(A_{0},A_{1}),G_{1}(A_{0},A_{1})\right)
_{\theta,p}^{\blacktriangleleft}}=\|f\|_{(A_{0},A_{1})_{\theta,p}^{%
\blacktriangleleft}}.   \label{sec2:6}
\end{equation}

The asymptotic role of Gagliardo relative completions within real
interpolation scales is characterized by the following limiting behavior:

\begin{proposition}
\label{prop:lim}Let $p \in [1, \infty]$ and $i\in\{0,1\}$. Setting $i^{\ast
}=1-i$, we have
\begin{equation}
\lim_{\theta\rightarrow i}\left\Vert f\right\Vert _{(A_{0},A_{1})_{\theta
,p}^{\blacktriangleleft}}=\left\Vert f\right\Vert _{G_{i}(A_{0},A_{1})}\qquad%
\text{for} \qquad f\in A_{i^{\ast}}.   \label{sec2:7}
\end{equation}
\end{proposition}

\begin{proof}
Consider the case $i=1$. For $f\in A_{0}\cap A_{1}$ the formula  %
\eqref{sec2:7} was already established in \cite[Theorem 1]{M-Lim} (if $p < \infty$) and \cite[Lemma 1]{KMX} (if $p \leq \infty$), see also \cite{DM}. Applying this to the Gagliardo  pair $%
\left( G_{0}(A_{0},A_{1}),G_{1}(A_{0},A_{1})\right) $ and invoking %
\eqref{sec2:6}, we obtain 
\begin{equation}
\lim_{\theta\rightarrow1^{-}}\left\Vert f\right\Vert
_{(A_{0},A_{1})_{\theta,p}^{\blacktriangleleft}}=\left\Vert f\right\Vert
_{G_{1}(A_{0},A_{1})}, \qquad f\in G_{0}(A_{0},A_{1})\cap
G_{1}(A_{0},A_{1}).   \label{sec2:8}
\end{equation}
By the inclusion properties (\ref{sec:sera}), \eqref{sec2:8} also holds for $f\in
A_{0}\cap G_{1}(A_{0},A_{1})$. It remains to check the validity of \eqref{sec2:8} for  $f\in A_{0}\setminus G_{1}(A_{0},A_{1})
$. For this purpose, we note that \eqref{sec2:1} yields 
\begin{equation*}
\left\Vert f\right\Vert
_{(A_{0},A_{1})_{\theta,p}^{\blacktriangleleft}}\geq\left\Vert f\right\Vert
_{(A_{0},A_{1})_{\theta,\infty}} \geq t^{-\theta} K(t, f; A_0, A_1) 
\end{equation*}
for every $t > 0$ and taking limits as $\theta\rightarrow1^{-}$:  
\begin{equation*}
\liminf_{\theta\rightarrow1^{-}}\left\Vert f\right\Vert
_{(A_{0},A_{1})_{\theta,p}^{\blacktriangleleft}}\geq\frac{K(t,f;A_{0},A_{1})%
}{t},\qquad\forall t>0. 
\end{equation*}
Taking now the supremum over $t>0$, we find that  
\begin{equation*}
\liminf_{\theta\rightarrow1^{-}}\left\Vert f\right\Vert
_{(A_{0},A_{1})_{\theta,p}^{\blacktriangleleft}}\geq\left\Vert f\right\Vert
_{G_{1}(A_{0},A_{1})} = + \infty,
\end{equation*}
which gives \eqref{sec2:7} for $i=1$ and $f  \not \in G_1(A_0, A_1)$.  

The case $i=0$ in  \eqref{sec2:7}  follows from the case $i=1$ when applied to the reversed pair $(A_{1},A_{0})$.
Indeed, since $K(t,f;A_{1},A_{0})=tK(\frac{1}{t},f;A_{0},A_{1})$, we have $\|f\|_{
G_{0}(A_{1},A_{0})}=\|f\|_{G_{1}(A_{0},A_{1})}$ and $\|f\|_{(A_{1},A_{0})_{\theta,p}^{%
\blacktriangleleft}}=\|f\|_{(A_{0},A_{1})_{1-\theta,p}^{\blacktriangleleft}}$. 
\end{proof}

\section{Variational $K$-functionals and  $\Gamma $%
-convergence}

\label{section:GagliardoGamma}

We now introduce an extension of $\Gamma$-convergence that tracks a sequence
of functionals establishing a direct bridge between De Giorgi and Gagliardo
methodologies.

\begin{definition}
\label{GdG} Let $\{F_{n}\}_{n\in\mathbb{N}}$ and $\{G_{n}\}_{n\in\mathbb{N}}$
be sequences of functionals mapping a normed  space $X$ to $[0, + \infty]$. For each $f\in X$, the upper and lower generalized $%
\Gamma $-limits are defined by 
\begin{equation*}
\Gamma-\limsup_{n\rightarrow\infty}(F_{n},G_{n})(f):=\inf\Big\{\limsup
_{n\rightarrow\infty}F_{n}(f_{n}):\{f_{n}\}_{n\in\mathbb{N}}\subset X \text{
with } \lim_{n\rightarrow\infty}G_{n}(f-f_n)=0\Big\}, 
\end{equation*}
\begin{equation*}
\Gamma-\liminf_{n\rightarrow\infty}(F_{n},G_{n})(f):=\inf\Big\{\liminf
_{n\rightarrow\infty}F_{n}(f_{n}):\{f_{n}\}_{n\in\mathbb{N}}\subset X \text{
with } \lim_{n\rightarrow\infty}G_{n}(f-f_n)=0\Big\}. 
\end{equation*}
The pair $(F_{n},G_{n})$ is said to $\Gamma$-converge if these limits
coincide for all $f\in X$, denoted by the common functional $%
\Gamma-\lim_{n\rightarrow \infty}(F_{n},G_{n})$. 
\end{definition}
\begin{remark}
	This definition extends
analogously to families of functionals $\{F_{\varepsilon}\}_{\varepsilon>0}$ and $%
\{G_{\varepsilon }\}_{\varepsilon>0}$.
\end{remark}

\begin{remark}
\label{Remark3.2} If $G_{n}(f)=\Vert f\Vert_{X}$ for every $n\in\mathbb{N}$
and  $f\in X$, Definition \ref{GdG} reduces to classical $\Gamma $%
-convergence  (cf. \eqref{intro1}-\eqref{intro2}) on $X$: 
\begin{equation*}
\Gamma-\lim_{n\rightarrow\infty}(F_{n},G_{n})=\Gamma_{X}-\lim_{n\rightarrow
\infty}F_{n}. 
\end{equation*}
\end{remark}

The structural analogue to Definition \ref{GdG} in the setting of
interpolation theory is the following notion:

\begin{definition}
\label{DefGKFunct} Given two functionals $F,G:X\rightarrow [0, + \infty]$, the variational  $K$-functional associated with the
pair $(G,F)$ is defined for each $f\in X$ and $t>0$ by 
\begin{equation*}
K(t,f;G,F):=\inf_{f=g+h}\{G(g)+tF(h)\}. 
\end{equation*}
\end{definition}

\begin{remark}
Let $\text{\textrm{dom}}\, F = \{f\in X : F(f)<+\infty\}$. If $F$ and $G$
are norms, Definition \ref{DefGKFunct} coincides with the classical
$K$-functional (cf. \eqref{eq:classical_K_intro}), meaning $K(t,f;G,F) = K(t,f;\text{%
\textrm{dom}}\,G, \text{\textrm{dom}}\,F)$ for $f \in\text{\textrm{dom}}\,G
+ \text{\textrm{dom}}\,F$. Then Definition \ref{DefGKFunct} yields the
natural non-linear (in the sense that norms are replaced by general functionals) extension of $K$-functionals.
\end{remark}

The following theorem provides the core equivalence, showing that
generalized $\Gamma$-limits are comparable to their Gagliardo completions.

\begin{theorem}
\label{Thm3} Let $\{F_{n}\}_{n\in\mathbb{N}}$ and $\{G_{n}\}_{n\in\mathbb{N}}
$ be sequences of functionals on $X$. Then, for any $f\in X$, 
\begin{align*}
\sup_{t>0}\liminf_{n\rightarrow\infty}\frac{K(t,f;G_{n},F_{n})}{t} &
\leq\Gamma-\liminf_{n\rightarrow\infty}(F_{n},G_{n})(f) \\
& \leq\Gamma-\limsup_{n\rightarrow\infty}(F_{n},G_{n})(f)\leq \inf_{\{t_n\}_{n \in \mathbb{N}}}\limsup
_{n\rightarrow\infty}\frac{K(t_n,f;G_{n},F_{n})}{t_n},
\end{align*}
where the infimum runs over all sequences $\{t_n\}_{n \in \mathbb{N}} \subset (0, \infty)$ such that $t_n \to 0$ as $n \to \infty$.  
\end{theorem}

\begin{proof}
We first prove the upper bound, more precisely, given any sequence $\{t_n\}_{n \in \mathbb{N}}$ formed by positive numbers with  $t_n \to 0$ as $n \to \infty$ then 
\begin{equation}
\Gamma-\limsup_{n\rightarrow\infty}(F_{n},G_{n})(f)\leq  \limsup
_{n\rightarrow\infty}\frac{K(t_n,f;G_{n},F_{n})}{t_n}.   \label{P3.122}
\end{equation}
Suppose that $L:=\limsup_{n\rightarrow\infty}K\big(t_n,f;G_{n},F_{n}%
\big)/t_n<  \infty$, otherwise \eqref{P3.122} holds trivially. In particular,  $K(t_n, f; G_n, F_n) < +\infty$ for all sufficiently large $n \geq n_0$. Given  $n\geq n_0$, choose a decomposition $f=g_{n}+h_{n}$ satisfying 
\begin{equation}
G_{n}(g_{n})+t_n F_{n}(h_{n})\leq K(t_n,f;G_{n},F_{n})%
+t_n^2.   \label{P2.1_thm}
\end{equation}
This implies that 
\begin{equation*}
G_{n}(f-h_{n})=G_{n}(g_{n})\leq K(t_n,f;G_{n},F_{n})+t^2_n. 
\end{equation*}
Taking limits yields 
\begin{equation*}
\limsup_{n\rightarrow\infty}G_{n}(f-h_{n})\leq\limsup_{n\rightarrow\infty }
 \frac{K (t_n,f;G_{n},F_{n})}{t_n} t_n =L 
\lim_{n\rightarrow\infty} t_n =0, 
\end{equation*}
which confirms that $G_{n}(f-h_{n})\rightarrow0$ as $n\rightarrow\infty$.
Moreover, from \eqref{P2.1_thm} we have 
\begin{equation*}
F_{n}(h_{n})\leq \frac{K(t_n,f;G_{n},F_{n})}{t_n}+t_n. 
\end{equation*}
Taking now limits on both sides of the previous estimate  we obtain $\limsup_{n\rightarrow%
\infty}F_{n}(h_{n})\leq L$. By Definition \ref{GdG}, this implies $\Gamma
-\limsup_{n\rightarrow\infty}(F_{n},G_{n})(f)\leq L$, and \eqref{P3.122}
follows.

Next we establish the lower bound: 
\begin{equation}
\sup_{t>0}\liminf_{n\rightarrow\infty}\frac{K(t,f;G_{n},F_{n})}{t}\leq
\Gamma-\liminf_{n\rightarrow\infty}(F_{n},G_{n})(f).   \label{P3.322}
\end{equation}
Let $M:=\Gamma-\liminf_{n\rightarrow\infty}(F_{n},G_{n})(f)$. We can assume $%
M<\infty$. Given $\varepsilon>0$, choose a sequence $\{f_{n}\}_{n\in%
\mathbb{N}} \subset X$ such that $G_{n}(f-f_n)\rightarrow0$ as $n\rightarrow \infty$
and 
\begin{equation}
\liminf_{n\rightarrow\infty}F_{n}(f_{n})\leq M+\varepsilon. 
\label{P2.2_thm}
\end{equation}
Fix $t>0$. There exists $m_0 = m_{0}(t,\varepsilon)\in\mathbb{N}$ such
that 
\begin{equation}
G_{n}(f-f_{n})\leq\varepsilon t\qquad\text{for all} \qquad n\geq
m_{0}.   \label{intermedia}
\end{equation}
Let $n\geq m_{0}$. To compute $K(t,f;G_{n},F_{n})$ consider
the decomposition $f=(f-f_{n})+f_{n}$, and use (\ref{intermedia}) to obtain 
\begin{equation}
\frac{K(t,f;G_{n},F_{n})}{t}\leq\frac{G_{n}(f-f_{n})}{t}+F_{n}(f_{n})\leq%
\varepsilon+F_{n}(f_{n}).   \label{P2.3_thm}
\end{equation}
Taking limits as $n\rightarrow\infty$ on both sides of %
\eqref{P2.3_thm} and applying \eqref{P2.2_thm}, 
\begin{equation*}
\liminf_{n\rightarrow\infty}\frac{K(t,f;G_{n},F_{n})}{t}\leq\varepsilon
+\liminf_{n\rightarrow\infty}F_{n}(f_{n})\leq2\varepsilon+M. 
\end{equation*}
Since this holds for any arbitrarily chosen $t>0$, we take the supremum over 
$t>0$ to arrive at: 
\begin{equation*}
\sup_{t>0}\liminf_{n\rightarrow\infty}\frac{K(t,f;G_{n},F_{n})}{t}%
\leq2\varepsilon+M. 
\end{equation*}
Since $\varepsilon>0$ was arbitrary, passing to the limit as $\varepsilon
\rightarrow0^{+}$ establishes the desired estimate \eqref{P3.322}.
\end{proof}

In the context of this section, to simplify the exposition,  we adopt the following 
conventional shorthand notation for the variational $K$-functional evaluated
against the ambient normed space: 
\begin{equation*}
K(t,f;X,F_{n}):=K(t,f;\Vert\cdot\Vert_{X},F_{n}).   
\end{equation*}
Specializing Theorem \ref{Thm3} to the framework detailed in Remark \ref%
{Remark3.2} yields Theorem \ref{teoarriba}. 


\section{{$\Gamma$-extrapolation\label{SectionGammaExtrapol}}}

In this section we show the $\Gamma$-convergence of the family of
interpolation norms $\Vert\cdot\Vert_{(A_{0},A_{1})_{\theta,p}^{%
\blacktriangleleft}}$ to the Gagliardo completions of the original pair $%
(A_{0},A_{1})$ as the interpolation parameter $\theta\in(0,1)$ approaches
the endpoints $0$ or $1$.

\begin{theorem}
\label{Thm1} Let $p\in\lbrack1,\infty], i\in\{0,1\}$ and set $i^{\ast}=1-i$. 
Then for $f\in A_{i}$ we have 
\begin{equation}\label{9.9.2}
\Gamma_{A_{i}}-\lim_{\theta\rightarrow i^{\ast}}\Vert
f\Vert_{(A_{0},A_{1})_{\theta,p}^{\blacktriangleleft}}=\Vert
f\Vert_{G_{i^{\ast}}(A_{0},A_{1})}. 
\end{equation}
\end{theorem}

\begin{proof}
The proof for  the case $i=0$ was already given in Theorem \ref{ThmExtrapolIntro}, modulo the validity of the estimates \eqref{elmal} and \eqref{elbien}. Let us first check \eqref{elmal}: For every $t>0$ and  any decomposition $f = f_0 + f_1$ with $f_0 \in A_0$ and $f_1 \in(A_{0},A_{1})_{\theta,p}$,
we have 
\begin{align}
K(t,f;A_{0},A_{1}) & \leq K(t,f_0;A_{0},A_{1})+K(t,f_1; A_{0},A_{1})   \leq\Vert f_0\Vert_{A_{0}}+K(t,f_1;A_{0},A_{1}).   \label{Al1}
\end{align}
Moreover, by virtue of \eqref{sec2:1} it follows that 
\begin{equation*}
K(t,f_1;A_{0},A_{1})\leq t^{\theta}\Vert f_1\Vert_{(A_{0},A_{1})_{\theta
,p}^{\blacktriangleleft}},\qquad\forall t>0. 
\end{equation*}
Then updating the estimate \eqref{Al1}, we obtain 
\begin{equation}\label{9.9.3}
K(t^{1/\theta},f;A_{0},A_{1})\leq\Vert f_0\Vert_{A_{0}}+t\Vert f_1\Vert
_{(A_{0},A_{1})_{\theta,p}^{\blacktriangleleft}} = \|f_0\|_{A_0} + t F_\theta(f_1),
\end{equation}
where  $F_{\theta}$ is the family of functionals given by \eqref{9.9.1}. 
Moreover, this estimate holds trivially if $f_1 \in A_0 \backslash (A_0, A_1)_{\theta, p}$ (since $F_\theta(f_1) = +\infty$). Then taking the infimum in \eqref{9.9.3} over all $f = f_0 + f_1$ with $f_0, f_1 \in A_0$, we arrive
at 
\begin{equation*}
K(t^{1/\theta},f;A_{0},A_{1})\leq K(t,f;A_{0},F_{\theta}).  
\end{equation*}
This proves \eqref{elmal}. 

Next we turn our attention to  \eqref{elbien}. By the definition of $K$-functional and applying Proposition \ref{prop:lim}, we get
\begin{align*}
	\limsup_{n\rightarrow\infty}nK\Big(\frac{1}{n},f,A_{0},F_{\theta_{n}} \Big) \leq \limsup_{n \to \infty} F_{\theta_n}(f) = \|f\|_{G_1(A_0, A_1)}. 
\end{align*}
This completes the proof of \eqref{9.9.2} with $i=0$. 

The case $i=1$ can be obtained from \eqref{9.9.2} with $i=0$ when applied to the reverse pair $(A_1, A_0)$. Indeed, in this case we have
\begin{equation}\label{14.9.1}
\Gamma_{A_{1}}-\lim_{\theta\rightarrow1^{-}}\Vert
f\Vert_{(A_{1},A_{0})_{\theta,p}^{\blacktriangleleft}}=\Vert
f\Vert_{G_{1}(A_{1},A_{0})},\qquad f\in A_{1}. 
\end{equation}
Now since $K(t,f;A_{1},A_{0})=tK(\frac{1}{t},f;A_{0},A_{1}),$ we get
\begin{align}\label{14.9.2}
\Vert f\Vert_{G_{1}(A_{1},A_{0})} & =\sup_{t > 0}\frac{K(t,f;A_{1},A_{0})}{t}  =\sup_{t > 0}K(t,f;A_{0},A_{1})  =\Vert f\Vert_{G_{0}(A_{0},A_{1})}
\end{align}
and, taking into account that $c_{\theta, p} = c_{1-\theta, p}$, 
\begin{align}
	\Vert
f\Vert_{(A_{1},A_{0})_{\theta,p}^{\blacktriangleleft}} &= c_{\theta, p} \bigg(\int_{0}^{\infty
}[t^{-\theta}K(t,f;A_{1},A_{0})]^{p}\,\frac{dt}{t}\bigg)^{1/p} \nonumber \\
& =  c_{\theta, p} \bigg(\int_{0}^{\infty
}[t^{1-\theta}K(t^{-1},f;A_{0},A_{1})]^{p}\,\frac{dt}{t}\bigg)^{1/p} \nonumber  \\
& = c_{1-\theta, p} \|f\|_{(A_0, A_1)_{1-\theta, p}} = \Vert
f\Vert_{(A_{0},A_{1})_{1-\theta,p}^{\blacktriangleleft}}.  \label{14.9.3} 
\end{align}
Here, we assume that $p < \infty$, while the case $p=\infty$ is even easier. 
As a combination of \eqref{14.9.1}, \eqref{14.9.2} and \eqref{14.9.3}, 
$$
	\Gamma_{A_1}-\lim_{\theta \to 0^{+}} \Vert
f\Vert_{(A_{0},A_{1})_{\theta,p}^{\blacktriangleleft}} = \|f\|_{G_0(A_0, A_1)}, \qquad f \in A_1. 
$$
\end{proof}

Another manifestation of the strong connection between
interpolation/extrapolation and $\Gamma$-convergence uncovered in this paper is the following: 
one can consider the $K$-functional of a given normed  pair as a family of
functionals depending on the penalty parameter $t$, for which the $\Gamma$-limits
are computed. To be more precise, we have:

\begin{theorem}
\label{ThmGKFunct} We have
\begin{equation}
\Gamma_{A_{0}}-\lim_{t\rightarrow0}\frac{K(t,f;A_{0},A_{1})}{t}=\Vert
f\Vert_{G_{1}(A_{0},A_{1})}  \qquad \text{if} \qquad f \in A_0  \label{PJ22.2}
\end{equation}
and 
\begin{equation}
\Gamma_{A_{1}}-\lim_{t\rightarrow\infty}K(t,f;A_{0},A_{1})=\Vert f\Vert
_{G_{0}(A_{0},A_{1})} \qquad \text{if} \qquad f \in A_1.   \label{PJ22.3}
\end{equation}
\end{theorem}

\begin{proof}
We show \eqref{PJ22.2}. For each $n\in{\mathbb{N}}$, consider 
\begin{equation}\label{1.10.6}
F_{n}(f):=nK\Big(\frac{1}{n},f;A_{0},A_{1}\Big), \qquad f \in A_0. 
\end{equation}
On the one hand, we have the trivial estimate 
\begin{equation}
nK\Big(\frac{1}{n},f;A_{0},F_{n}\Big)\leq F_{n}(f).   \label{PJ4.2}
\end{equation}
On the other hand, given any decomposition $f = f_0 + f_1$ with $f_0, f_1 \in A_0$  and $t>0$, we have 
\begin{align}
K(t,f;A_{0},A_{1}) & \leq K(t,f_0;A_{0},A_{1})+K(t,f_1;A_{0},A_{1})  \leq\Vert f_0\Vert_{A_{0}}+K(t,f_1;A_{0},A_{1}).   \label{PJ4.1}
\end{align}
Without loss of generality, we may assume that $n\in{\mathbb{N}}$ is
sufficiently large so that $\frac{1}{n}\leq t$ and therefore (by basic monotonicity properties of $K$-functionals) 
\begin{equation*}
K(t,f_1;A_{0},A_{1})\leq tnK\Big(\frac{1}{n},f_1;A_{0},A_{1}\Big). 
\end{equation*}
Inserting this into \eqref{PJ4.1}, we arrive at 
\begin{align*}
K(t,f;A_{0},A_{1}) & \leq\Vert f_0\Vert_{A_{0}}+tnK\Big(\frac{1}{n}%
,f_1;A_{0},A_{1}\Big)  =\Vert f_0\Vert_{A_{0}}+tF_{n}(f_1).
\end{align*}
Taking the infimum over all possible decompositions $f = f_0 + f_1$, we obtain 
\begin{equation*}
K(t,f;A_{0},A_{1})\leq K(t,f;A_{0},F_{n}) 
\end{equation*}
if $\frac{1}{n}\leq t$. In particular 
\begin{equation}
K(t,f;A_{0},A_{1})\leq\liminf_{n\rightarrow\infty}K(t,f;A_{0},F_{n}). 
\label{PJ4.3}
\end{equation}

Applying  now Theorem \ref{teoarriba} to the functionals $F_n$ given by \eqref{1.10.6} defined on $X=A_{0}$, combined with \eqref{PJ4.2}
and \eqref{PJ4.3}, gives 
\begin{align*}
\Vert f\Vert_{G_{1}(A_{0},A_{1})} & =\sup_{t>0}\frac{K(t,f;A_{0},A_{1})}{t}
 \leq \sup_{t > 0}\liminf_{n\rightarrow\infty}\frac{K(t,f;A_{0},F_{n})}{t} \\
& \leq\Gamma_{A_{0}}-\liminf_{n\rightarrow\infty}F_{n}(f)  \leq\Gamma_{A_{0}}-\limsup_{n\rightarrow\infty}F_{n}(f) \\
& \leq\limsup_{n\rightarrow\infty}nK\Big(\frac{1}{n},f;A_{0},F_{n}\Big)  \leq\limsup_{n\rightarrow\infty}F_{n}(f) \leq \Vert f\Vert_{G_{1}(A_{0},A_{1})}. 
\end{align*}
As a byproduct, we have shown that 
\begin{equation*}
\Gamma_{A_{0}}-\lim_{n\rightarrow\infty}F_{n}(f)=\Vert
f\Vert_{G_{1}(A_{0},A_{1})} \qquad \text{if} \qquad f \in A_0. 
\end{equation*}
This completes the proof of \eqref{PJ22.2}.

The assertion \eqref{PJ22.3} is an immediate consequence of \eqref{PJ22.2}
applied to the reverse pair $(A_{1},A_{0})$.
\end{proof}

\section{Extrapolation of $\Gamma$-convergence: applications and examples
\label{sec:app&exples}}

In this section, we showcase the versatility of our extrapolation
methodology in handling both $\Gamma$-convergence and pointwise limits
across a wide variety of functionals. Our approach offers a unified
perspective; regardless of the application, the starting point is always a
suitable pair of compatible spaces. Although the execution details differ,
the overall plan of action remains the same and relies on the analysis of
the associated  $K$-functional.

\subsection{$\Gamma$-convergence to Sobolev norms}

\label{SubSecSobolevGamma}

Recall that, for $s\in(0,1)$ and $p\in\lbrack1,\infty)$, the \emph{fractional Sobolev
space} $\dot{W}^{s,p}(\mathbb{R}^{N})$ is equipped with the Gagliardo
seminorm 
\begin{equation}
\Vert f\Vert_{\dot{W}^{s,p}(\mathbb{R}^{N})}:=\bigg(\iint_{\mathbb{R}%
^{N}\times\mathbb{R}^{N}}\frac{|f(x)-f(y)|^{p}}{|x-y|^{sp+N}}\,dx \, dy\bigg)%
^{1/p}.   \label{DefGag}
\end{equation}
An important result that follows from the work of Ponce \cite[Theorem 8]{P}
asserts that the $\Gamma$-limit as $s\rightarrow1^{-}$ in $L_{\text{loc}}^{1}(%
\mathbb{R}^{N})$ of the family of Gagliardo seminorms \eqref{DefGag},
adequately normalized by the prefactor $(1-s)^{1/p}$, is a multiple of the
standard Sobolev seminorm 
\begin{equation*}
\Vert f\Vert_{\dot{W}^{1,p}(\mathbb{R}^{N})}:=\Vert\nabla f\Vert _{L^{p}(%
\mathbb{R}^{N})}. 
\end{equation*}
Specifically, if $p>1$ then 
\begin{equation}
\Gamma_{L_{\text{loc}}^{1}(\mathbb{R}^{N})}-\lim_{s\rightarrow1^{-}}(1-s)^{1/p}%
\Vert f\Vert_{\dot{W}^{s,p}(\mathbb{R}^{N})}=K_{p,N}\,\Vert f\Vert_{\dot{W}%
^{1,p}(\mathbb{R}^{N})},   \label{S5.1}
\end{equation}
where $K_{p,N}$ is an explicit constant that depends only on $p$ and $N$.
See also \cite{BPS}, \cite{CDLKNP} and \cite{GS}.

Next we show that \eqref{S5.1} can be considered as a special case of
Theorem \ref{Thm1} for the pair $(L^{p}(\mathbb{R}^{N}),\dot{W}^{1,p}(%
\mathbb{R}^{N}))$. More generally,  we will show that \eqref{S5.1} also holds in
setting of higher order Sobolev spaces based on rearrangement invariant
spaces. Given a rearrangement invariant space $X$ (in short, r.i. space) on $\mathbb{R}^{N}$
endowed with the Lebesgue measure and $k\in{\mathbb{N}}$, the \emph{Sobolev
space} $\dot{W}^{k}X$ is defined by 
\begin{equation*}
\Vert f\Vert_{\dot{W}^{k}X}:=\sum_{|\alpha|=k}\Vert D^{\alpha}f\Vert_{X}, 
\end{equation*}
where $\alpha=(\alpha_{1},\ldots,\alpha_N)\in{\mathbb{N}}_{0}^{N}$ is a
multi-index with order $|\alpha|=\alpha_{1}+\cdots+\alpha_N$. Let 
\begin{equation*}
W^{k}X:=X\cap\dot{W}^{k}X 
\end{equation*}
equipped  with the norm 
\begin{equation*}
\Vert f\Vert_{W^{k}X}:=\Vert f\Vert_{X}+\Vert f\Vert_{\dot{W}^{k}X}. 
\end{equation*}
On the other hand, the \emph{Besov space} $\dot{B}_{q}^{sk}X,s\in
(0,1),q\in\lbrack1,\infty]$, is the set of all $f\in L_{\text{loc}}^{1}(%
\mathbb{R}^{N})$ such that
\begin{equation}
\Vert f\Vert_{\dot{B}_{q}^{sk}X}:=\bigg(\int_{\mathbb{R}^{N}}\frac{\Vert
\Delta_{h}^{k}f\Vert_{X}^{q}}{|h|^{skq+N}}\,dh\bigg)^{1/q}<\infty 
\label{DefBesovD}
\end{equation}
(with the usual modification if $q=\infty$), where $\Delta_{h}^{k}$ denotes the \emph{$k$-th order difference of step $h$}
given by 
\begin{equation}\label{14.9.5}
\Delta_{h}^{1}f(x)=\Delta_{h}f(x):=f(x+h)-f(x),\qquad\Delta_{h}^{k}:=%
\Delta_{h} \Delta_{h}^{k-1} \quad\text{if}\quad k>1. 
\end{equation}
Let 
\begin{equation*}
\Vert f\Vert_{B_{q}^{sk}X}:=\Vert f\Vert_{X}+\Vert f\Vert_{\dot{B}%
_{q}^{sk}X}.  
\end{equation*}
In the special case $X=L^{p}(\mathbb{R}^{N})$, we shall adopt the standard
notation $\dot{B}_{p,q}^{sk}(\mathbb{R}^{N})$ (resp. $B_{p,q}^{sk}(\mathbb{R}%
^{N})$) rather than $\dot{B}_{q}^{sk}L^{p}(\mathbb{R}^{N})$ (resp. $%
B_{q}^{sk}L^{p}(\mathbb{R}^{N})$). Note that if $s\in(0,1),$ then $\dot{W}^{s, p}(\R^N) = \dot{B}^s_{p, p}(\R^N)$ with 
\begin{equation}
\Vert f\Vert_{\dot{W}^{s,p}(\mathbb{R}^{N})}=\Vert f\Vert_{\dot{B}_{p,p}^{s}(%
\mathbb{R}^{N})}.   \label{CoincGFS}
\end{equation}

\begin{theorem}
\label{Thm54} Let $q\in\lbrack1,\infty]$ and $k\in{\mathbb{N}}$. Then for $%
f\in W^k X$ we have 
\begin{equation*}
\Vert f\Vert_{\dot{W}^{k}X} \lesssim \Gamma_{X}-\liminf_{s\rightarrow1^{-}}(1-s)^{1/q}\Vert f\Vert_{\dot{B}%
_{q}^{sk}X}\leq \Gamma_{X}-\limsup_{s\rightarrow1^{-}}(1-s)^{1/q}\Vert
f\Vert_{\dot {B}_{q}^{sk}X}\lesssim \Vert f\Vert_{\dot{W}^{k}X}. 
\end{equation*}
\end{theorem}

\begin{remark}
In particular, applying Theorem \ref{Thm54} with $X=L^{p}(\mathbb{R}^{N}),k=1
$ and $q=p,$ we obtain (cf. \eqref{CoincGFS}) 
\begin{align*}
C_1 \Vert f\Vert_{\dot{W}^{1,p}(\mathbb{R}^{N})}& \leq \Gamma_{L^{p}(\mathbb{R}^{N})}-\liminf_{s\rightarrow1^{-}}(1-s)^{1/p}\Vert
f\Vert_{\dot{W}^{s,p}(\mathbb{R}^{N})} \\
& \leq \Gamma_{L^{p}(\mathbb{R}%
^{N})}-\limsup_{s\rightarrow1^{-}}(1-s)^{1/p}\Vert f\Vert_{\dot{W}^{s,p}(%
\mathbb{R}^{N})}\leq C_2 \Vert f\Vert_{\dot{W}^{1,p}(\mathbb{R}^{N})}, 
\end{align*}
which gives the $L^p$-topology analog of  \eqref{S5.1}, up to equivalence constants (cf. Section \ref{SectionExact} below for a detailed study on exact equivalence constants). More generally, if $k \in \mathbb{N}$ then
\begin{align}
C_1 \Vert f\Vert_{\dot{W}^{k,p}(\mathbb{R}^{N})}& \leq \Gamma_{L^{p}(\mathbb{R}^{N})}-\liminf_{s\rightarrow1^{-}}(1-s)^{1/p}\Vert
f\Vert_{\dot{B}^{s k}_{p, p}(\mathbb{R}^{N})} \nonumber\\
& \leq \Gamma_{L^{p}(\mathbb{R}%
^{N})}-\limsup_{s\rightarrow1^{-}}(1-s)^{1/p}\Vert f\Vert_{\dot{B}^{s k}_{p, p}(%
\mathbb{R}^{N})}\leq C_2 \Vert f\Vert_{\dot{W}^{k,p}(\mathbb{R}^{N})}.  \label{14.9.8}
\end{align}
\end{remark}

\begin{proof}[Proof of Theorem \protect\ref{Thm54}]
Apply Theorem \ref{Thm1} to the pair $(A_{0},A_{1})=(X,\dot{W}^{k}X)$, then 
\begin{equation}
\Gamma_{X}-\lim_{s\rightarrow1^{-}}\Vert f\Vert_{(X,\dot{W}%
^{k}X)^{\blacktriangleleft}_{s,q}}=\Vert f\Vert_{G_{1}(X,\dot{W}^{k}X)}   \label{PD28.2}
\end{equation}
for every $f\in X$. The method of proof now consists in estimating
appropriately both sides of \eqref{PD28.2}.

Now, for each $t>0$ and $f \in X$,  
\begin{equation}
K(t^{k},f;X,\dot{W}^{k}X)\approx  \begin{cases}
\big(\frac{1}{t^{N}}\int_{|h|<t}\Vert\Delta_{h}^{k}f\Vert_{X}^{q}\,dh\big)%
^{1/q} & \qquad \text{if} \qquad q \in (0, \infty), \\
\omega_k(f, t)_X & \qquad \text{if} \qquad q = \infty,   \label{PD28.1}
\end{cases}
\end{equation}
where
\begin{equation}\label{15.9.1}
\omega_k(f, t)_X:=	\sup_{|h| \leq t} \|\Delta^k_h f\|_X
\end{equation}
is the classical \emph{$k$-th order modulus of smoothness} of $f$. For $X=L^{p}(\mathbb{R}^{N}), \, p \in [1, \infty]$, the formula \eqref{PD28.1} with $q=\infty$ is well known (see e.g. \cite[Theorem 4.12, p. 339, (4.42), p. 341]{BS}), while the case $q < \infty$ may be found
in \cite[Appendix A]{KMX}; in this case, we shall use the simplified notation $\omega_k(f, t)_p$ to denote $\omega_k(f, t)_{L^p(\R^N)}$ and, when there is no danger of confusion, $\omega(f, t)_p := \omega_1(f, t)_p$. For general r.i. spaces,  the proof of \eqref{PD28.1} follows along similar
lines as in the setting of Lebesgue spaces, cf. \cite[Lemma B.1]{DLTYY} for further details.

Let $s\in(0,1)$ and $q \in [1, \infty)$. Using successively the definition of the
interpolation norm, a change of variables,  the characterization of the $K$-functional \eqref{PD28.1} 
and Fubini's theorem, we get 
\begin{align}
\Vert f\Vert_{(X,\dot{W}^{k}X)_{s,q}} & \approx\bigg(\int_{0}^{\infty
}t^{-skq}K(t^{k},f;X,\dot{W}^{k}X)^{q}\,\frac{dt}{t}\bigg)^{1/q}  \notag \\
& \approx\bigg(\int_{0}^{\infty}t^{-skq-N}\int_{|h|<t}\Vert\Delta_{h}^{k}f%
\Vert_{X}^{q}\,dh\,\frac{dt}{t}\bigg)^{1/q}  \notag \\
& \approx\bigg(\int_{\mathbb{R}^{N}}\frac{\Vert\Delta_{h}^{k}f\Vert_{X}^{q}}{%
|h|^{skq+N}}\,dh\bigg)^{1/q}=\Vert f\Vert_{\dot{B}_{q}^{sk}X}. 
\label{PD28.3}
\end{align}
Here, we stress that all underlying equivalence constants are uniform with respect to $s \in (0, 1)$.  The
case $q=\infty$ in \eqref{PD28.3} is even easier and then we shall omit the
details.

We estimate the left-hand side of \eqref{PD28.2} via \eqref{PD28.3}, 
\begin{align*}
\Gamma_{X}-\liminf_{s\rightarrow1^{-}}\Vert f\Vert_{(X,\dot{W}^{k}X)_{s,q}^\blacktriangleleft}
& \lesssim\Gamma_{X}-\liminf_{s\rightarrow1^{-}}(1-s)^{1/q}\Vert f\Vert_{\dot {B%
}_{q}^{sk}X}   \\
& \leq\Gamma_{X}-\limsup_{s\rightarrow1^{-}}(1-s)^{1/q}\Vert f\Vert_{\dot{B}%
_{q}^{sk}X}   \\
& \lesssim\Gamma_{X}-\limsup_{s\rightarrow1^{-}}\Vert f\Vert_{(X,\dot{W}%
^{k}X)_{s,q}^\blacktriangleleft}.   
\end{align*}
Using these estimates in \eqref{PD28.2}, we achieve that, for $f \in X$, 
\begin{align}
	\Vert f\Vert_{G_{1}(X,\dot{W}^{k}X)}  &\lesssim \Gamma_{X}-\liminf_{s\rightarrow1^{-}}(1-s)^{1/q}\Vert f\Vert_{\dot {B%
}_{q}^{sk}X}  \notag \\
& \leq\Gamma_{X}-\limsup_{s\rightarrow1^{-}}(1-s)^{1/q}\Vert f\Vert_{\dot{B}%
_{q}^{sk}X}  \notag \\
& \lesssim \Vert f\Vert_{G_{1}(X,\dot{W}^{k}X)}. \label{13.9.1}
\end{align}

To estimate $\|f\|_{G_1(X, \dot{W}^k X)}$, we will make use of a result obtained in \cite[(6.8)]{DLTYY}: If $f\in\dot{W}^{k}X$ and $q \in (0, \infty)$ then the following limit exists  
\begin{equation*}
	\lim_{t\rightarrow0}\frac{\big(\frac{1}{t^{N/k}}\int_{|h|<t^{1/k}}%
\Vert\Delta_{h}^{k}f\Vert_{X}^{q}\,dh\big)^{1/q}}{t} < \infty
\end{equation*}
and, in addition,  its value (which can be explicitly computed) satisfies
\begin{equation}\label{10.9.2}
	\lim_{t\rightarrow0}\frac{\big(\frac{1}{t^{N/k}}\int_{|h|<t^{1/k}}%
\Vert\Delta_{h}^{k}f\Vert_{X}^{q}\,dh\big)^{1/q}}{t}  \approx \Vert f\Vert_{\dot{W}^{k}X}. 
\end{equation}
Concerning the case $q= \infty$, we make use of the fact that (see \cite[(1.8)]{T} for the case $X = L^p(\R^N)$ but the arguments can be easily adapted to work with any r.i. space)
\begin{equation}\label{12.9.1}
	\lim_{t \to 0} \frac{\omega_k(f, t)_X}{t^k} = \sup_{h \neq 0} \frac{\|\Delta^k_h f\|_X}{|h|^k}
\end{equation}
provided that $f \in \dot{W}^k X$. Note that the right-hand side of \eqref{12.9.1} is finite due to the  fact that  (cf. \cite[(6.6)]{DLTYY}) 
$$
	 \sup_{h \neq 0} \frac{\|\Delta^k_h f\|_X}{|h|^k} \approx \|f\|_{\dot{W}^k X}, \qquad f \in \dot{W}^k X. 
$$
Hence the counterpart of \eqref{10.9.2} related to $q=\infty$ reads 
\begin{equation}\label{12.9.2}
	\lim_{t \to 0} \frac{\omega_k(f, t)_X}{t^k}  \approx \|f\|_{\dot{W}^k X}. 
\end{equation}

Assume that $f\in\dot{W}^{k}X$. Using \eqref{3.9.1}, \eqref{PD28.1} and \eqref{10.9.2}, we derive
\begin{align}
\Vert f\Vert_{G_{1}(X,\dot{W}^{k}X)} & =\lim_{t\rightarrow0}\frac {K(t,f;X,%
\dot{W}^{k}X)}{t}   \approx \lim_{t\rightarrow0}\frac{\big(\frac{1}{t^{N/k}}\int_{|h|<t^{1/k}}%
\Vert\Delta_{h}^{k}f\Vert_{X}^{q}\,dh\big)^{1/q}}{t}  \approx\Vert f\Vert_{\dot{W}^{k}X}   \label{PD28.5i}
\end{align}
if $q < \infty$. The case $q=\infty$ follows similarly but using instead \eqref{12.9.2}.

Updating \eqref{13.9.1} with \eqref{PD28.5i}, we find for $f \in W^k X$
\begin{align*}
\Vert f\Vert_{\dot{W}^{k}X} & \lesssim\Gamma_{X}-\liminf_{s\rightarrow
1^{-}}(1-s)^{1/q}\Vert f\Vert_{\dot{B}_{q}^{sk}X} \\
& \leq\Gamma_{X}-\limsup_{s\rightarrow1^{-}}(1-s)^{1/q}\Vert f\Vert_{\dot{B}%
_{q}^{sk}X}\lesssim\Vert f\Vert_{\dot{W}^{k}X}.
\end{align*}
\end{proof}

\subsection{$\Gamma$-convergence to fractional Sobolev norms via Butzer seminorms}
As pointed out in  \cite{BSY},  it is an interesting question to what extent BBM-type formulae related to the gradient operator (or more generally, higher order derivatives) can be extended to deal with non-local differential operators  such as the fractional Laplacian $(-\Delta)^{t/2}, t > 0$. Note that standard approaches to the classical  BBM formula \eqref{14.9.4} (as well as its $\Gamma$-counterpart \eqref{S5.1}) strongly rely on the local properties of the gradient operator. Next we show how the new machinery provided by Theorem \ref{Thm1} can overcome these issues with the help of Butzer seminorms introduced in \cite{DM}. 

Recall that, for $t > 0$ and $p \in (1, \infty)$, the \emph{Riesz potential space} $\dot{H}^{t, p}(\R^N)$ is endowed with
\begin{equation}\label{15.9.2}
	\|f\|_{\dot{H}^{t, p}(\R^N)} := \|(-\Delta)^{\frac{t}{2}} f\|_{L^p(\R^N)}
\end{equation}
and the \emph{Bessel potential space} $H^{t, p}(\R^N)$ is defined by
$$
	\|f\|_{H^{t, p}(\R^N)}:= \|f\|_{L^p(\R^N)} + \|f\|_{\dot{H}^{t, p}(\R^N)}.  
$$
For $t > 0$ and $h \in \R^N$, the \emph{Butzer operator} is defined by (cf. (1.15) and especially  footnote 8 in \cite{DM})
$$
	\Delta^t_h f (x) := \sum_{j=0}^\infty (-1)^j \binom{t}{j} f(x+ (t-j) h), \qquad x \in \R^N,
$$
where
$$
	\binom{t}{j} := \frac{t(t-1) \cdots (t-j+1)}{j!} \qquad \text{if} \qquad j \in \mathbb{N},  \qquad \binom{t}{0}:= 1. 
$$
Clearly, if $t=k \in \mathbb{N}$ then $\Delta^t_h$ coincides with the standard differences \eqref{14.9.5}. Mimicking the definition of the classical Gagliardo seminorm (cf. \eqref{DefGag})
$$
	\Vert f\Vert_{\dot{W}^{s,p}(\mathbb{R}^{N})} = \bigg(\iint_{\mathbb{R}%
^{N}\times\mathbb{R}^{N}}\frac{|\Delta_h f(x)|^{p}}{|h|^{sp+N}}\,dx \, dh\bigg)%
^{1/p},
$$
one may introduce the \emph{Butzer seminorm} (cf. \cite[Definition 1.1]{DM})
\begin{equation}\label{14.9.6}
	\|f\|_{\dot{W}^{s, p}(\R^N), t} := \bigg(\iint_{\mathbb{R}%
^{N}\times\mathbb{R}^{N}}\frac{|\Delta^t_h f(x)|^{p}}{|h|^{s t p+N}}\,dx \, dh\bigg)%
^{1/p}
\end{equation}
for $p \in (1, \infty)$ and $s \in (0, 1)$. Observe that
$$
	\|f\|_{\dot{W}^{s, p}(\R^N), 1} =  \Vert f\Vert_{\dot{W}^{s,p}(\mathbb{R}^{N})} 
$$
and, more generally,  
\begin{equation}\label{14.9.7}
	 \|f\|_{\dot{W}^{s, p}(\R^N), k} =  \|f\|_{\dot{B}^{s k}_{p, p}(\R^N)} \qquad \text{if} \qquad k \in \mathbb{N}.  
\end{equation}
Accordingly, \eqref{14.9.6} consists of a natural (fractional) extension of the notions of Gagliardo and Besov seminorms.  

\begin{theorem}\label{Thm14.9.26}
	Let $p \in (1, \infty)$ and $t > 0$. Then for $f \in H^{t, p}(\R^N)$ we have
	\begin{align}
	 \|f\|_{\dot{H}^{t, p}(\R^N)} &\lesssim \Gamma_{L^p(\R^N)}-\liminf_{s \to 1^{-}} (1-s)^{1/p} \|f\|_{\dot{W}^{s, p}(\R^N), t} \nonumber \\
		&  \leq \Gamma_{L^p(\R^N)}-\limsup_{s \to 1^{-}} (1-s)^{1/p} \|f\|_{\dot{W}^{s, p}(\R^N), t}  \lesssim  \|f\|_{\dot{H}^{t, p}(\R^N)}.  \label{14.9.9}
	\end{align}
\end{theorem}

\begin{remark}
\begin{enumerate}
\item[(i)]	Let $t = k \in \mathbb{N}$ in Theorem \ref{Thm14.9.26}. Using the well-known fact that $$\dot{H}^{k, p}(\R^N) = \dot{W}^{k, p}(\R^N), \qquad p \in (1, \infty),$$
	and \eqref{14.9.7}, then \eqref{14.9.9} gives back \eqref{14.9.8}. 
\item[(ii)] 
	The pointwise counterpart of Theorem \ref{Thm14.9.26} was already obtained in \cite[Theorem 3.4]{DM}. 
	\end{enumerate}
\end{remark}

\begin{proof}[Proof of Theorem \ref{Thm14.9.26}]
	Consider the pair $(A_0, A_1) = (L^p(\R^N), \dot{H}^{t, p}(\R^N))$. Applying Theorem \ref{Thm1} we infer that 
	\begin{equation}\label{14.9.10}
		\Gamma_{L^p(\R^N)}-\lim_{s \to 1^{-}} \|f\|_{(L^p(\R^N), \dot{H}^{t, p}(\R^N))_{s, p}^{\blacktriangleleft}} = \|f\|_{G_1(L^p(\R^N), \dot{H}^{t, p}(\R^N))} 
	\end{equation}
	if $f \in L^p(\R^N)$. 
	
	Next we identify the corresponding interpolation norms and Gagliardo completions that arise in \eqref{14.9.10}.  To do this, we shall employ the following characterization of the $K$-functional relative to $(L^p(\R^N), \dot{H}^{t, p}(\R^N))$ in terms of Butzer operators: For $u > 0$ and $f \in L^p(\R^N)$, 
	\begin{equation}\label{14.9.11}
		K(u^t, f; L^p(\R^N), \dot{H}^{t, p}(\R^N)) \approx \frac{1}{u^N} \int_{|h| < u} \|\Delta^t_h f\|_{L^p(\R^N)} \, dh \approx \sup_{|h| < u} \|\Delta^t_h f\|_{L^p(\R^N)};
	\end{equation}
	cf. \cite[(1.14) and (1.35)]{KT}. 
	
	Assume that $s \in (0, 1)$. By a simple change of variables 
	\begin{align}\label{14.9.12}
		\|f\|_{(L^p(\R^N), \dot{H}^{t, p}(\R^N))_{s, p}} & \approx  \bigg(\int_0^\infty [u^{-s t} K(u^t, f; L^p(\R^N), \dot{H}^{t, p}(\R^N) )]^p \, \frac{du}{u} \bigg)^{1/p}
	\end{align} 
	and using the first estimate in \eqref{14.9.11} together with H\"older's inequality with exponent $p$ and Fubini's theorem,  we obtain that
	\begin{align}
		\|f\|_{(L^p(\R^N), \dot{H}^{t, p}(\R^N))_{s, p}} & \approx   \bigg(\int_0^\infty \bigg[u^{-s t-N}  \int_{|h| < u} \|\Delta^t_h f\|_{L^p(\R^N)} \, dh \bigg]^p \, \frac{du}{u} \bigg)^{1/p} \nonumber \\
		& \lesssim  \bigg(\int_0^\infty u^{-s t p-N}  \int_{|h| < u} \|\Delta^t_h f\|^p_{L^p(\R^N)} \, dh \, \frac{du}{u} \bigg)^{1/p} \nonumber \\
		& \approx \bigg(\int_{\R^N}  |h|^{-s t p -N} \|\Delta^t_h f\|^p_{L^p(\R^N)}  \, dh \bigg)^{1/p} = \|f\|_{\dot{W}^{s, p}(\R^N), t}, \label{14.9.13}
	\end{align}
	where all equivalence constants may be taken to be independent of $s \in (0, 1)$. Conversely,  the second estimate in \eqref{14.9.11} gives 
	\begin{align*}
		  \|f\|_{\dot{W}^{s, p}(\R^N), t} &  \lesssim \bigg(\int_{\R^N}  |h|^{-s t p -N} K(|h|^t, f; L^p(\R^N), \dot{H}^{t, p}(\R^N))^p  \, dh \bigg)^{1/p}
	\end{align*}
	and then applying polar coordinates
	\begin{align}
		 \|f\|_{\dot{W}^{s, p}(\R^N), t} & \lesssim  \bigg(\int_0^\infty  u^{-s t p} K(u^t, f; L^p(\R^N), \dot{H}^{t, p}(\R^N))^p  \, \frac{du}{u} \bigg)^{1/p} \nonumber \\
		 & \approx \|f\|_{(L^p(\R^N), \dot{H}^{t, p}(\R^N))_{s, p}}, \label{14.9.14}
	\end{align}
	where we have used \eqref{14.9.12} in the last step. Once again, we observe that all underlying equivalence constants in \eqref{14.9.14} are independent of $s$. Putting now together \eqref{14.9.13} and \eqref{14.9.14}, we have shown that 
	\begin{equation}\label{14.9.15}
		 \|f\|_{(L^p(\R^N), \dot{H}^{t, p}(\R^N))_{s, p}} \approx  \|f\|_{\dot{W}^{s, p}(\R^N), t} 
	\end{equation}
	uniformly with respect to $s$. 
	
	Next we turn our attention to $G_1(L^p(\R^N), \dot{H}^{t, p}(\R^N))$. It follows from \eqref{14.9.11} that 
	\begin{align}
		\|f\|_{G_1(L^p(\R^N), \dot{H}^{t, p}(\R^N))} &= \sup_{u > 0} \frac{K(u^t, f; L^p(\R^N), \dot{H}^{t, p}(\R^N)) }{u^t} \nonumber \\
		& \approx \sup_{h \neq 0}  \frac{\|\Delta^t_h f\|_{L^p(\R^N)}}{|h|^t}. \label{14.9.16}
	\end{align}
	Moreover, according to \cite[Lemma 3.6]{DT2}, if $f \in H^{t, p}(\R^N)$ then 
	$$
	\|f\|_{\dot{H}^{t, p}(\R^N)}	\approx \sup_{h \neq 0}  \frac{\|\Delta^t_h f\|_{L^p(\R^N)}}{|h|^t}. 
	$$
	Hence \eqref{14.9.16} can be updated to say that 
	\begin{equation}\label{14.9.17}
		\|f\|_{G_1(L^p(\R^N), \dot{H}^{t, p}(\R^N))}  \approx \|f\|_{\dot{H}^{t, p}(\R^N)}. 
	\end{equation}
	
	From \eqref{14.9.10}, \eqref{14.9.15} and \eqref{14.9.17}, we infer that \eqref{14.9.9} holds. 
\end{proof}

\subsection{$\Gamma$-convergence to the total variation}

\label{SusectionBV} Recall  that $f\in L^{1}(\mathbb{R}^{N})$ has bounded
variation if and only if its distributional gradient $\nabla f$ is a finite
measure, i.e. its total variation 
\begin{equation*}
\Vert f\Vert_{\dot{\text{$BV$}}(\mathbb{R}^{N})}:=\sup\bigg\{\int_{\mathbb{R}%
^{N}}f(x)\text{div}\phi(x) \, dx:\phi\in C_{c}^{1}(\mathbb{R}^{N},\mathbb{R}%
^{N}),\Vert\phi\Vert_{L^{\infty}(\mathbb{R}^{N})}\leq 1\bigg\}
\end{equation*}
is finite. 
The higher order counterparts $\dot{BV}^{k}(\mathbb{R}^{N}),k\in\mathbb{N}$,
are defined by 
\begin{equation*}
\Vert f\Vert_{\dot{\text{$BV$}}^{k}(\mathbb{R}^{N})}:=\sum_{|\alpha|=k-1}%
\Vert D^{\alpha}f\Vert_{\dot{\text{$BV$}}(\mathbb{R}^{N})}. 
\end{equation*}
Let $BV^k (\R^N) := L^1(\R^N) \cap \dot{BV}^k(\R^N)$. 

$\Gamma$-convergence results  of non-local functionals to the total variation have 
been widely investigated. In particular, convergence of Gagliardo seminorms \eqref{DefGag} with $%
p=1$ was already addressed by Ponce \cite{P} (see also \cite[Theorem 6.9, p. 240]{BM} and \cite{N}). Next we show how this result follows in a simple way from the general extrapolation principle of Theorem \ref{Thm1}, even more,  the result can be considerably extended  to work with  Besov norms and  higher order total
variation.

\begin{theorem}
\label{ThmTV} Let $q\in\lbrack1,\infty]$ and $k\in{\mathbb{N}}$. Then, for $f \in BV^k (\R^N)$, we have 
\begin{align*}
\Vert f\Vert_{\dot{BV}^{k}(\mathbb{R}^{N})} &\lesssim \Gamma_{L^{1}(\mathbb{R}^{N})}-\liminf_{s\rightarrow1^{-}}(1-s)^{1/q}\Vert
f\Vert_{\dot{B}_{1,q}^{sk}(\mathbb{R}^{N})} \\
& \leq\Gamma_{L^{1}(\mathbb{R}%
^{N})}-\limsup_{s\rightarrow1^{-}}(1-s)^{1/q}\Vert f\Vert_{\dot{B}_{1,q}^{sk}(%
\mathbb{R}^{N})}\lesssim \Vert f\Vert_{\dot{BV}^{k}(\mathbb{R}^{N})}. 
\end{align*}
\end{theorem}

\begin{proof} From \eqref{13.9.1} applied with $X=L^{1}(\mathbb{R}^{N})$, it follows that for $f \in L^1(\R^N)$, 
\begin{align*}
\Vert f\Vert_{G_{1}(L^{1}(\mathbb{R}^{N}),\dot{W}^{k,1}(\mathbb{R}%
^{N}))} & \lesssim \Gamma_{L^{1}(\mathbb{R}^{N})}-\liminf_{s\rightarrow1^{-}}(1-s)^{1/q}\Vert
f\Vert_{\dot{B}_{1,q}^{sk}(\mathbb{R}^{N})} \\
& \leq \Gamma_{L^{1}(\mathbb{R}%
^{N})}-\limsup_{s\rightarrow1^{-}}(1-s)^{1/q}\Vert f\Vert_{\dot {B}_{1,q}^{sk}(%
\mathbb{R}^{N})} \\
& \lesssim \Vert f\Vert_{G_{1}(L^{1}(\mathbb{R}^{N}),\dot{W}^{k,1}(\mathbb{R}%
^{N}))}.
\end{align*}
To conclude it remains to prove that 
\begin{equation}
\Vert f\Vert_{G_{1}(L^{1}(\mathbb{R}^{N}),\dot{W}^{k,1}(\mathbb{R}%
^{N}))}\approx\Vert f\Vert_{\dot{\text{$BV$}}^{k}(\mathbb{R}^{N})}, \qquad f \in \dot{BV}^k(\R^N). 
\label{PJ1.1}
\end{equation}

Let us compute $G_{1}(L^{1}(\mathbb{R}^{N}),\dot{W}^{k,1}(\mathbb{R}^{N}))$.
Using \eqref{PD28.1}, 
\begin{align}
\Vert f\Vert_{G_{1}(L^{1}(\mathbb{R}^{N}),\dot{W}^{k,1}(\mathbb{R}^{N}))} &
=\sup_{t > 0}\frac{K(t^{k},f;L^{1}(\mathbb{R}^{N}),\dot{W}^{k,1}(%
\mathbb{R}^{N}))}{t^{k}}     \approx\sup_{t > 0}\frac{\omega_{k}(f,t)_{1}}{t^{k}}.   \label{PJ1.2}
\end{align}
Here,  we use the abbreviated notation $\omega_k(f, t)_1 := \omega_k(f, t)_{L^1(\R^N)}$.  To prove \eqref{PJ1.1}
when $k=1$ we use \eqref{PJ1.2} and the well-known characterization (see e.g. \cite{CDDD}) of the
total variation in terms of the moduli of smoothness: 
\begin{equation}\label{13.9.2}
\Vert f\Vert_{\dot{\text{$BV$}}(\mathbb{R}^{N})}\approx \sup_{t > 0} \frac{\omega(f, t)_1}{t}.
\end{equation}
 It follows from \eqref{PJ1.2} that 
$$
	\Vert f\Vert_{G_{1}(L^{1}(\mathbb{R}^{N}),\dot{W}^{1,1}(\mathbb{R}^{N}))}  \approx \Vert f\Vert_{\dot{\text{$BV$}}(\mathbb{R}^{N})}
$$
and then \eqref{PJ1.1} with $k=1$ holds. For the case $k>1$ we use  the classical
estimates for the moduli of smoothness of derivatives  (cf. \cite[(12.2) and (12.3), p. 117]%
{DT}) 
\begin{equation}
t^{-(k-1)}\omega_{k}(f,t)_{1}\lesssim\sup_{|\alpha|=k-1}\omega_{1}(D^{\alpha
}f,t)_{1}\lesssim\int_{0}^{t}u^{-(k-1)}\omega_{k}(f,u)_{1}\,\frac{du}{u}. 
\label{PJ1.4}
\end{equation}
Note that \eqref{PJ1.4} leads to 
\begin{equation*}
\sup_{t > 0}\frac{\omega_{k}(f,t)_{1}}{t^{k}}\approx\sup_{|\alpha
|=k-1}\sup_{t > 0}\frac{\omega_{1}(D^{\alpha}f,t)_{1}}{t}
\end{equation*}
and therefore, by \eqref{PJ1.2} and \eqref{13.9.2},  
\begin{align*}
\Vert f\Vert_{G_{1}(L^{1}(\mathbb{R}^{N}),\dot{W}^{k,1}(\mathbb{R}%
^{N}))} & \approx\sup_{|\alpha|=k-1}\sup_{t > 0}\frac{\omega_{1}(D^{\alpha
}f,t)_{1}}{t} \\
& \approx\sup_{|\alpha|=k-1}\Vert D^{\alpha}f\Vert_{\dot {\text{$%
BV$}}(\mathbb{R}^{N})} \approx \|f\|_{\dot{BV}^k(\R^N)}. 
\end{align*}
This proves the desired claim \eqref{PJ1.1} with $k > 1$ and completes the proof. 
\end{proof}

\subsection{$\Gamma$-convergence of BMO-type functionals to the total
variation\label{sec:bmotype}}

For $f\in L_{\text{loc}}^{1}(\mathbb{R}^{N})$ and $\varepsilon>0$, consider
the BMO-type functional given by 
\begin{equation}
\kappa_{\varepsilon}(f):=\varepsilon^{N-1}\sup_{\mathcal{F}%
_{\varepsilon}}\sum_{Q\in\mathcal{F}_{\varepsilon}} \frac{1}{|Q|}\int_{Q}|f-f_{Q}|, 
\label{BMO-type}
\end{equation}
where $\mathcal{F}_{\varepsilon}$ is a family of pairwise disjoint cubes
with side length $\varepsilon$ and arbitrary orientation. As usual, $%
f_{Q}:=\frac{1}{|Q|} \int_{Q}f$, the mean value of $f$ on $Q$. These functionals were
first introduced by Ambrosio, Bourgain, Brezis and Figalli \cite{ABBF} to
provide a characterization of sets with finite perimeter. This result was later extended in  \cite{DePFS, FMS, PS2} to  $SBV$-functions (i.e.,
special functions of bounded variation). However,
as shown in \cite{DePFS} for general $f\in\dot{B V }(\mathbb{R}^{N})$
the limit $\lim_{\varepsilon\rightarrow0}\kappa_{\varepsilon}(f)$ may not
exist. This failure can be overcome if one works with the additional
flexibility provided by the $\Gamma$-convergence. Indeed, it was shown by
Arroyo-Rabasa, Bonicatto and Del Nin \cite{ARBDN} that 
\begin{equation}
\Gamma_{L_{\text{loc}}^{1}(\mathbb{R}^{N})}-\lim_{\varepsilon\rightarrow
0}\kappa_{\varepsilon}(f)=\frac{1}{4}\,\Vert f\Vert_{\dot{BV}(\mathbb{R}%
^{N})}   \label{ARBDN}
\end{equation}
for $f \in L^1_{\text{loc}}(\R^N)$. 
We note that the appearance of the prefactor $1/4$ is related to the optimal
Poincar\'{e} constant for affine maps on the unit cube.

We now prove that \eqref{ARBDN} (up to the optimal constant, cf. Section \ref{SectionExact} for a remedy) is also part of
the general extrapolation phenomenon described in Section \ref%
{SectionGammaExtrapol}.

\begin{theorem}
We have, for $f \in L^1(\R^N)$,  
\begin{equation}
\Vert f\Vert_{\dot {%
BV}(\mathbb{R}^{N})} \lesssim \Gamma_{L^{1}(\mathbb{R}^{N})}-\liminf_{\varepsilon\rightarrow0}\kappa_{%
\varepsilon}(f)\leq \Gamma_{L^{1}(\mathbb{R}^{N})}-\limsup
_{\varepsilon\rightarrow0}\kappa_{\varepsilon}(f)\lesssim \Vert f\Vert_{\dot {%
BV}(\mathbb{R}^{N})}.   \label{ARBDN2}
\end{equation}
\end{theorem}

\begin{proof}
To apply the extrapolation method we shall use the following characterization
of the $K$-functional for the pair $(L^{1},\dot{W}^{1,1})$, which has been recently obtained in \cite%
{DS}: for all $\varepsilon>0$ and $f \in L^1(\R^N)$, 
\begin{equation}
\kappa_{c \varepsilon}(f) \lesssim \frac{K(\varepsilon,f;L^{1}(\mathbb{R}^{N}),\dot{W}^{1,1}(\mathbb{R}%
^{N}))}{\varepsilon} \lesssim \kappa_{c' \varepsilon}(f)   \label{K-Funct-DS}
\end{equation}
for some constants $c, c' > 0$. 
Accordingly, the validity of \eqref{ARBDN2} follows immediately from Theorem %
\ref{ThmGKFunct} specialised to the pair $(A_0, A_1) = (L^1(\R^N), \dot{W}^{1, 1}(\R^N))$ and the formula \eqref{PJ1.1} with $k=1$ (which remains valid adequately interpreted for $f \in L^1(\R^N)$).
\end{proof}

\subsection{$\Gamma$-convergence to Sobolev norms in metric measure spaces}

\label{GConvMetric} Let $(\mathcal{X},d,\mu)$ be a complete and \emph{%
doubling} metric measure space in the sense that $\mu$ is a positive Borel
measure for which there exists a constant $C_{d}\geq1$ such that for every
ball $B(x,r):=\{y\in\mathcal{X}:d(x,y)<r\},\,x\in\mathcal{X}$ and $r>0$ it
holds 
\begin{equation}
0<\mu(B(x,2r))\leq C_{d}\,\mu(B(x,r))<\infty.   \label{DefDoubling}
\end{equation}

Recall that 
$\mathcal{X}$ is said to support a $p$-\emph{Poincar\'{e} inequality}, $%
p\in\lbrack1,\infty)$, if there exist constants $C>0$ and $\tau\geq1$ such
that for every ball $B(x,r)$ and every Lipschitz function $f:\mathcal{X}%
\rightarrow\mathbb{R}$, 
\begin{equation}
\fint_{B(x,r)}|f-f_{B(x,r)}|\,d\mu\leq Cr\bigg(\fint_{B(x,\tau r)}(\text{Lip }%
f)^{p}\,d\mu\bigg)^{1/p}.   \label{DefPoin}
\end{equation}
Here, we use the standard notation $f_{B(x,r)}:= \fint_{B(x, r)} f \, d \mu:= \frac{1}{\mu(B(x, r))}\int_{B(x,r)}f\,d\mu$ and $%
\text{Lip }f$ is the \emph{pointwise upper Lipschitz constant} of $f$: 
\begin{equation*}
\text{Lip }f(y):=\limsup_{r\rightarrow0}\sup_{x\in\mathcal{X}:d(x,y)\leq r}%
\frac{|f(x)-f(y)|}{r}. 
\end{equation*}

In this subsection, we assume that $(\mathcal{X},d,\mu)$ is a \emph{Cheeger
space} or \emph{PI space}, i.e., it satisfies \eqref{DefDoubling} and %
\eqref{DefPoin} for some $p \geq 1$. The \emph{Korevaar--Schoen type energy functionals} are
defined for $f\in L_{\text{loc}}^{p}(\mathcal{X})$ and $t>0$ by 
\begin{equation}
E_{p,t}(f):=\bigg(\int_{\mathcal{X}}\fint_{B(x,t)}|f(x)-f(y)|^{p}\,d\mu(y)%
\,d\mu(x)\bigg)^{1/p}.   \label{DefKS}
\end{equation}
The associated \emph{Korevaar--Schoen space} is given by (cf. \cite{KS}) 
\begin{equation*}
\Vert f\Vert_{KS^{1,p}(\mathcal{X})}:=\sup_{t>0}\frac{E_{p,t}(f)}{t}<\infty. 
\end{equation*}
In the setting of Cheeger spaces, it is well known (cf. \cite[Chapter 10]%
{HKST}) that $KS^{1,p}(\mathcal{X}),p\in(1,\infty)$ coincides with Newtonian
Sobolev spaces, Hajlasz-Sobolev spaces and furthermore 
\begin{equation*}
\Vert f\Vert_{KS^{1,p}(\mathcal{X})}\approx\limsup_{t\rightarrow0}\frac{%
E_{p,t}(f)}{t}. 
\end{equation*}
In particular, when $p=1$ then $KS^{1,1}(\mathcal{X})$ coincides with the space
of functions of bounded variation $\dot{BV}(\mathcal{X})$.

Let $s\in(0,1)$ and $q\in\lbrack1,\infty]$. Following \cite{GKS}, the Besov
space $\dot{\mathcal{B}}_{p,q}^{s}(\mathcal{X})$ is formed by all $f\in L_{%
\text{loc}}^{p}(\mathcal{X})$ that satisfy 
\begin{equation*}
\Vert f\Vert_{\dot{\mathcal{B}}_{p,q}^{s}(\mathcal{X})}:=\bigg(%
\int_{0}^{\infty}[t^{-s}E_{p,t}(f)]^{q}\,\frac{dt}{t}\bigg)^{1/q}<\infty 
\end{equation*}
(with the usual convention if $q=\infty$). Note that with equivalence
constants that are independent of $s$ (cf. \cite[Theorem 5.2]{GKS}) 
\begin{equation}
\Vert f\Vert_{\dot{\mathcal{B}}_{p,p}^{s}(\mathcal{X})}\approx\bigg(\int_{%
\mathcal{X}}\int_{\mathcal{X}}\frac{|f(x)-f(y)|^{p}}{d(x,y)^{sp}%
\mu(B(x,d(x,y)))}\,d\mu(y)\,d\mu(x)\bigg)^{1/p}.   \label{EquivBGMMS}
\end{equation}
In the Euclidean setting $\mathcal{X}=\mathbb{R}^{N}$ equipped with the
Lebesgue measure, it is known that (recall  
\eqref{DefBesovD}) 
\begin{equation*}
\Vert f\Vert_{\dot{\mathcal{B}}_{p,q}^{s}(\mathbb{R}^{N})}\approx\Vert
f\Vert_{\dot{B}_{p,q}^{s}(\mathbb{R}^{N})}
\end{equation*}
and, in particular, with equivalence constants independent of $s$ (cf. %
\eqref{EquivBGMMS}) 
\begin{equation*}
\Vert f\Vert_{\dot{\mathcal{B}}_{p,p}^{s}(\mathbb{R}^{N})}\approx\Vert
f\Vert_{\dot{W}^{s,p}(\mathbb{R}^{N})}. 
\end{equation*}

\begin{theorem}
\label{ThmGCSMMS} Let $p \in [1, \infty)$ and $q \in[1, \infty]$. Then, for $f \in L^p(\mathcal{X}) \cap KS^{1, p}(\mathcal{X})$,  
\begin{align} 
\|f\|_{KS^{1, p}(\mathcal{X})}& \lesssim \Gamma_{L^{p}(\mathcal{X})}-\liminf_{s \to1^{-}} (1-s)^{1/q} \|f\|_{\dot{%
\mathcal{B}}^{s}_{p, q}(\mathcal{X})} \nonumber \\
& \leq \Gamma_{L^{p}(\mathcal{X}%
)}-\limsup_{s \to1^{-}} (1-s)^{1/q} \|f\|_{\dot{\mathcal{B}}^{s}_{p, q}(%
\mathcal{X})} \lesssim \|f\|_{KS^{1, p}(\mathcal{X})}.  \label{PJ4.7}
\end{align}
In particular, if $p=q$ then (cf. \eqref{EquivBGMMS}) 
\begin{align}
 \|f\|_{KS^{1, p}(\mathcal{%
X})}^{p} &\lesssim \Gamma_{L^{p}(\mathcal{X})}-\liminf_{s \to1^{-}} (1-s) \int_{\mathcal{X}} \int_{%
\mathcal{X}} \frac{|f(x)-f(y)|^{p}}{d(x, y)^{s p} \mu(B(x, d(x, y)))} \, d
\mu(y) \, d \mu(x)   \notag \\
&  \leq \Gamma_{L^{p}(\mathcal{X})}-\limsup_{s \to1^{-}} (1-s)
\int_{\mathcal{X}} \int_{\mathcal{X}} \frac{|f(x)-f(y)|^{p}}{d(x, y)^{s p}
\mu(B(x, d(x, y)))} \, d \mu(y) \, d \mu(x) \lesssim \|f\|_{KS^{1, p}(\mathcal{%
X})}^{p}.   \label{PJ4.8}
\end{align}
\end{theorem}

\begin{remark}
The counterpart of \eqref{PJ4.8} associated to 
pointwise convergence may be found in \cite{Mu} and \cite{DiMS}.
\end{remark}

\begin{proof}[Proof of Theorem \protect\ref{ThmGCSMMS}]
We apply Theorem \ref{Thm1} to the pair $(A_{0},A_{1})=(L^{p}(\mathcal{X}%
),KS^{1,p}(\mathcal{X}))$. It follows from \cite[Theorem 4.1]{GKS} that 
\begin{equation}
K(t,f;L^{p}(\mathcal{X}),KS^{1,p}(\mathcal{X}))\approx E_{p,t}(f). 
\label{PJ4.4}
\end{equation}
As a consequence, we have
\begin{equation}
\Vert f\Vert_{(L^{p}(\mathcal{X}),KS^{1,p}(\mathcal{X}))_{s,q}}\approx\Vert
f\Vert_{\dot{\mathcal{B}}_{p,q}^{s}(\mathcal{X})},   \label{PJ4.5}
\end{equation}
with underlying equivalence constants independent of $s\in(0,1)$.
Furthermore, from \eqref{PJ4.4}, we also derive that 
\begin{equation}
\Vert f\Vert_{G_{1}(L^{p}(\mathcal{X}),KS^{1,p}(\mathcal{X}))}\approx\Vert
f\Vert_{KS^{1,p}(\mathcal{X})}.   \label{PJ4.6}
\end{equation}
Combining \eqref{PJ4.5} and \eqref{PJ4.6} with Theorem \ref{Thm1}, we arrive
at the desired estimates \eqref{PJ4.7}.
\end{proof}

A different approach to Sobolev spaces on metric measure spaces has been
recently proposed by Alonso Ruiz and Baudoin in \cite{ARB}. At the heart of
their approach is a new characterization of $KS^{1,p}(\mathcal{X})$ in terms
of $\Gamma$-limits of the functionals $E_{p,t}(f)$
given by \eqref{DefKS}. Namely, Theorem 3.1 in \cite{ARB} guarantees the
existence of a positive sequence of numbers $(t_{n})_{n\in{\mathbb{N}}}$
with $\lim_{n\rightarrow\infty}t_{n}=0$ such that\footnote{%
Our notation is slightly different from \cite{ARB}, more precisely, the
definition of $E_{p,t}(f)$ as given in \cite[(2)]{ARB} corresponds to $\big(%
\frac{E_{p,t}(f)}{t}\big)^{p}$ according to the notation used in %
\eqref{DefKS}.} 
\begin{equation}\label{4.10.2}
\mathcal{E}_{p}(f):=\Gamma_{L^p(\mathcal{X})}-\lim_{n\rightarrow\infty}\bigg(\frac{%
E_{p,t_{n}}(f)}{t_{n}}\bigg)^{p}
\end{equation}
exists and, for $f \in L^p(\mathcal{X})$,  
\begin{equation}\label{4.10.3}
f \in KS^{1,p}(\mathcal{X}) \iff \mathcal{E}_{p}(f)<\infty. 
\end{equation}
We show that this characterization is a simple consequence of the $\Gamma $%
-extrapolation method.

\begin{theorem}\label{Thm4.10.1}
Let $1 \leq p < \infty$. Then, for $f \in L^p(\mathcal{X})$,  
\begin{equation}  \label{PJ5.1}
\|f\|_{KS^{1, p}(\mathcal{X})} \lesssim \Gamma_{L^{p}(\mathcal{X})}-\liminf_{t \to0} \frac{E_{p, t}(f)}{t}
\leq \Gamma_{L^{p}(\mathcal{X})}-\limsup_{t \to0} \frac{E_{p, t}(f)}{t}
\lesssim \|f\|_{KS^{1, p}(\mathcal{X})}.
\end{equation}
\end{theorem}

\begin{remark}
	Note that existence of a $\Gamma$-convergent subfamily in \eqref{PJ5.1} is always guaranteed in light of \cite[Theorem 8.5]{DALM}. In particular, the functional $\mathcal{E}_p$ in \eqref{4.10.2} is well-defined and \eqref{4.10.3} holds.  
\end{remark}

\begin{proof}[Proof of Theorem \ref{Thm4.10.1}]
Consider the pair $(A_{0}, A_{1}) = (L^{p}(\mathcal{X}), KS^{1, p}(\mathcal{X%
}))$ and apply Theorem \ref{ThmGKFunct}. Then \eqref{PJ5.1} follows
immediately from the characterization of the corresponding $K$-functional
given in \eqref{PJ4.4} and the Gagliardo completion computed in \eqref{PJ4.6}%
.
\end{proof}

\subsection{$\Gamma$-convergence to rearrangement invariant norms}

The Maz'ya--Shaposhnikova formula \cite{MS} claims that, for $f\in
C_{0}^{\infty}(\mathbb{R}^{N})$, 
\begin{equation*}
\lim_{s\rightarrow0^{+}}s^{1/p}\Vert f\Vert_{\dot{W}^{s,p}(\mathbb{R}%
^{N})}=c_{N, p}\,\Vert f\Vert_{L^{p}(\mathbb{R}^{N})},  
\end{equation*}
where $c_{N, p}$ is a known explicit constant. The asymptotic analysis of the
Gagliardo seminorms $\Vert f\Vert_{\dot{W}^{s,p}(\mathbb{R}^{N})}$ as $%
s\rightarrow0^{+}$ in the setting of $\Gamma$-convergence was first
investigated by De Luca, Novaga, and Ponsiglione \cite{DLNP} for fractional
perimeters, i.e., restricted to $f=\mathbf{1}_{E}$, the characteristic
function of the set $E$. This restriction has been removed recently by
Crismale, De Luca, Kubin, Ninno and Ponsiglione \cite{CDLKNP}. More precisely, these authors
work with functions $f\in L^{2}$ whose support is contained in a Lipschitz
bounded domain and applied their results to derive stability of fractional
heat flows.

Next we apply the $\Gamma$-extrapolation method to extend the result of \cite%
{CDLKNP} to the context of r.i.  spaces $X$ on $\mathbb{R}%
^N$. This may be considered as the counterpart of Theorem \ref{Thm54} when
the smoothness parameter $s \to 0^{+}$. We shall follow
the notation introduced in Subsection \ref{SubSecSobolevGamma}.

\begin{theorem}\label{THMGammaMS}
Let $q\in\lbrack1,\infty]$ and $k\in{\mathbb{N}}$. Then for $f\in W^k X$ we have
\begin{equation}
 \Vert f\Vert_{X} \lesssim \Gamma_{\dot{W}^{k}X}-\liminf_{s\rightarrow0^{+}}s^{1/q}\Vert f\Vert_{\dot{B}%
_{q}^{sk}X}\leq \Gamma_{\dot{W}^{k}X}-\limsup_{s\rightarrow0^{+}}s^{1/q}\Vert
f\Vert_{\dot{B}_{q}^{sk}X}\lesssim \Vert f\Vert_{X}.   \label{PJ6.3}
\end{equation}
In particular, for $k=1,X=L^{p}(\mathbb{R}^{N})$ and $q=p\in\lbrack1,\infty)
$, 
\begin{equation}
\Vert f\Vert_{L^{p}(\mathbb{R}^{N})} \lesssim \Gamma_{\dot{W}^{1,p}(\mathbb{R}^{N})}-\liminf_{s\rightarrow0^{+}}s^{1/p}\Vert
f\Vert_{\dot{W}^{s,p}(\mathbb{R}^{N})}\leq \Gamma_{\dot{W}^{1,p}(\mathbb{R}%
^{N})}-\limsup_{s\rightarrow0^{+}}s^{1/p}\Vert f\Vert_{\dot{W}^{s,p}(\mathbb{R}%
^{N})}\lesssim \Vert f\Vert_{L^{p}(\mathbb{R}^{N})}.   \label{PJ6.4}
\end{equation}
\end{theorem}

\begin{proof}
We apply Theorem \ref{Thm1} to the pair $(A_{0},A_{1})=(X,\dot{W}^{k}X)$.
Recall that (cf. \eqref{PD28.3}) 
\begin{equation}
\Vert f\Vert_{(X,\dot{W}^{k}X)_{s,q}}\approx\Vert f\Vert_{\dot{B}_{q}^{sk}X} 
\label{PJ6.1}
\end{equation}
uniformly in $s\in(0,1)$. Moreover, according to  \cite[%
pp. 230-231]{DLTYY} 
\begin{equation*}
\lim_{t\rightarrow\infty}\frac{1}{t^{N}}\int_{|h|<t}\Vert\Delta_{h}^{k}f%
\Vert_{X}^{q}\,dh=C_{N}\lim_{|h|\rightarrow\infty}\Vert\Delta_{h}^{k}f%
\Vert_{X}^{q}\approx\Vert f\Vert_{X}^{q}
\end{equation*}
if $q < \infty$. In the case $q=\infty$, it is clear that (cf. \eqref{15.9.1})
$$
	\lim_{t \to \infty} \omega_k(f, t)_X = \sup_{h \in \R^N} \|\Delta^k_h f\|_X \approx \|f\|_X. 
$$
Consequently, combining with \eqref{PD28.1} we obtain 
\begin{equation}
\Vert f\Vert_{G_{0}(X,\dot{W}^{k}X)}=\lim_{t\rightarrow\infty}K(t^{k},f;X,%
\dot{W}^{k}X)\approx\Vert f\Vert_{X}.   \label{PJ6.2}
\end{equation}
An application of Theorem \ref{Thm1}, taking into account \eqref{PJ6.1} and %
\eqref{PJ6.2}, gives the desired assertion \eqref{PJ6.3}.
\end{proof}

\begin{remark}
The method of proof of Theorem \ref{THMGammaMS} admits some improvements. First, it is possible to update the mode of $\Gamma$-convergence  
from $\dot{W}^{k} X$ to $X$. Second, the result is valid for $f \in X$. These issues, together with exact $\Gamma$-formulae,  will be discussed in Section \ref{SectionExact}  below. 
\end{remark}

\subsection{$\Gamma$-convergence of heat-type energies\label{sec:garo}}

The \emph{heat semigroup} $(H_{t})_{t>0}$ is defined for $f\in L^{1}(\mathbb{%
R}^{N})$ by 
\begin{equation*}
H_{t}f(x):=h_{t}\ast f, 
\end{equation*}
where for $x\in\mathbb{R}^{N}$ and $t>0$, $h_{t}(x)$ is the \emph{heat kernel%
} 
\begin{equation}
h_{t}(x):=\frac{e^{-\frac{|x|^{2}}{4t}}}{(4\pi t)^{\frac{N}{2}}}. 
\label{DefHeatKernel}
\end{equation}

The $\Gamma$-convergence of heat-type energies to the total variation has
been recently studied by Gennaioli and Stefani. In particular, Theorem 1.6
in \cite{GS} asserts that 
\begin{equation}
\Gamma_{L^{1}(\mathbb{R}^{N})}-\lim_{t\rightarrow0}\frac{1}{\sqrt{t}}\int_{%
\mathbb{R}^{N}}H_{t}(|f-f(x)|)(x)\,dx=\frac{2}{\sqrt{\pi}}\,\Vert f\Vert_{%
\dot{BV}(\mathbb{R}^{N})}.   \label{FormulaGS}
\end{equation}
As observed in \cite[(1.12)]{GS}, this formula is of special relevance when
applied to $f=\mathbf{1}_{E}$ based on its reformulation in terms of the amount of heat that
remains inside the set $E$ asymptotically as $t\rightarrow0$. The method of 
\cite{GS} is a refinement of Ponce's approach \cite{P} and involves families of
non-local functionals that arise in BBM formulae combined with a useful
observation from \cite[Theorem B]{GT} (where the pointwise counterpart of %
\eqref{FormulaGS} was obtained even in the setting of Carnot groups)
connecting heat-type energies with the BBM formulae. Next we shall provide a
completely different approach to \eqref{FormulaGS} based on the general $\Gamma$-extrapolation phenomenon. In particular, this approach is self-contained and independent of BBM
formulae.

\begin{theorem}
\label{ThmJ22.1} Let $f\in BV(\mathbb{R}^{N}),$ then 
\begin{align*}
\Vert
f\Vert_{\dot{BV}(\mathbb{R}^{N})}& \lesssim \Gamma_{L^{1}(\mathbb{R}^{N})}-\liminf_{t\rightarrow0}\frac{1}{\sqrt{t}}%
\int_{\mathbb{R}^{N}}H_{t}(|f-f(x)|)(x)\,dx & \\
& \leq \Gamma_{L^{1}(\mathbb{R}^{N})}-\limsup_{t\rightarrow 0}%
\frac{1}{\sqrt{t}}\int_{\mathbb{R}^{N}}H_{t}(|f-f(x)|)(x)\,dx\lesssim \Vert
f\Vert_{\dot{BV}(\mathbb{R}^{N})}.
\end{align*}
\end{theorem}

\begin{proof}
Assume first that $f \geq 0$. By definition of $H_t$, we have  
\begin{equation}
\int_{\mathbb{R}^{N}}H_{t}(|f-f(x)|)(x)\,dx=\int_{\mathbb{R}^{N}}\int_{%
\mathbb{R}^{N}}h_{t}(x-y)|f(y)-f(x)|\,dy\,dx.    \label{uba}
\end{equation}
It is convenient to rewrite this formula using superlevel sets. Indeed, let 
\begin{equation*}
E_{\tau}^{f}:=\{x\in\mathbb{R}^{N}:f(x)>\tau\},  \qquad \tau \in \R, 
\end{equation*}
and observe that 
\begin{equation}
|f(x)-f(y)|=\int_0^\infty |\mathbf{1}_{E_{\tau}^{f}}(x)-\mathbf{1}%
_{E_{\tau}^{f}}(y)|\,d\tau.   \label{uba1}
\end{equation}
Update (\ref{uba}) using (\ref{uba1}) and then apply Fubini's theorem to obtain 
\begin{align}
\int_{\mathbb{R}^{N}}H_{t}(|f-f(x)|)(x)\,dx &  =\int_0^\infty \int_{\mathbb{R}^{N}}\bigg(\int_{\mathbb{R}%
^{N}}h_{t}(x-y)|\mathbf{1}_{E_{\tau}^{f}}(x)-\mathbf{1}_{E_{\tau}^{f}}(y)|%
\,dy\bigg)dx\,d\tau  \notag \\
& =\int_0^\infty \int_{\mathbb{R}^{N}}H_{t}(|\mathbf{1}_{E_{\tau}^{f}}-%
\mathbf{1}_{E_{\tau}^{f}}(x)|)(x)\,dx\,d\tau.   \label{PJ6.5}
\end{align}
Since $\int_{\mathbb{R}^{N}}h_{t}(x)\,dx=1,t>0,$ one can easily verify that 
\begin{equation*}
\int_{\mathbb{R}^{N}}H_{t}(|\mathbf{1}_{E_{\tau}^{f}}-\mathbf{1}%
_{E_{\tau}^{f}}(x)|)(x)\,dx=\Vert H_{t}\mathbf{1}_{E_{\tau}^{f}}-\mathbf{1}%
_{E_{\tau }^{f}}\Vert_{L^{1}(\mathbb{R}^{N})}
\end{equation*}
and inserting this into \eqref{PJ6.5}, we find
that 
\begin{equation*}
\int_{\mathbb{R}^{N}}H_{t}(|f-f(x)|)(x)\,dx=\int_0^\infty \Vert H_{t}%
\mathbf{1}_{E_{\tau}^{f}}-\mathbf{1}_{E_{\tau}^{f}}\Vert_{L^{1}(\mathbb{R}%
^{N})}\,d\tau. 
\end{equation*}
Hence, the desired $\Gamma$-limit can be reduced to 
\begin{equation}
\Gamma_{L^{1}(\mathbb{R}^{N})}-\lim_{t\rightarrow0}\frac{1}{\sqrt{t}}\int_{%
\mathbb{R}^{N}}H_{t}(|f-f(x)|)(x)\,dx=\Gamma_{L^{1}(\mathbb{R}%
^{N})}-\lim_{t\rightarrow0}\int_0^\infty \frac{\Vert H_{t}\mathbf{1}%
_{E_{\tau }^{f}}-\mathbf{1}_{E_{\tau}^{f}}\Vert_{L^{1}(\mathbb{R}^{N})}}{%
\sqrt{t}}\,d\tau.   \label{PJ6.6}
\end{equation}

Consider the Riesz potential  space given by 
\begin{equation*}
\dot{H}^{2,1}(\mathbb{R}^{N}):=\{f\in L^{1}(\mathbb{R}^{N}):\Vert
f\Vert_{\dot{H}^{2,1}(\mathbb{R}^{N})}:=\Vert\Delta f\Vert _{L^{1}(%
\mathbb{R}^{N})}<\infty\};
\end{equation*}
cf. \eqref{15.9.2}. 
It is clear that 
\begin{equation}
\Vert f\Vert_{\dot{H}^{2,1}(\mathbb{R}^{N})}\lesssim\Vert f\Vert_{%
\dot{W}^{2,1}(\mathbb{R}^{N})}   \label{PJ6.7}
\end{equation}
although the converse is not  true. We will make use of the known
fact that (see e.g. \cite[p. 99]{DT}) 
\begin{equation}
K(t,f;L^{1}(\mathbb{R}^{N}),\dot{H}^{2,1}(\mathbb{R}%
^{N}))\approx\Vert H_{t}f-f\Vert_{L^{1}(\mathbb{R}^{N})}    \label{PJ6.8}
\end{equation}
if $f \in L^1(\R^N)$. 
Specifically, we use \eqref{PJ6.8}
for the function $\mathbf{1}_{E_{\tau}^{f}} \in L^1(\R^N)$ in order to write 
\begin{equation}
\int_0^\infty \frac{\Vert H_{t}\mathbf{1}_{E_{\tau}^{f}}-\mathbf{1}%
_{E_{\tau}^{f}}\Vert_{L^{1}(\mathbb{R}^{N})}}{\sqrt{t}}\,d\tau\approx \int_0^\infty \frac{K(t,\mathbf{1}_{E_{\tau}^{f}};L^{1}(\mathbb{R}^{N}),\dot{H}^{2,1}(\mathbb{R}^{N}))}{\sqrt{t}}\,d\tau.   \label{PJ6.9}
\end{equation}
To estimate the right-hand side of \eqref{PJ6.6} we can use \eqref{PJ6.7}
and \eqref{PD28.1},  which lead to 
\begin{align}\label{PJ.14}
K(t,f;L^{1}(\mathbb{R}^{N}),\dot{H}^{2,1}(\mathbb{R}^{N})) &
\lesssim K(t,f;L^{1}(\mathbb{R}^{N}),\dot{W}^{2,1}(\mathbb{R}^{N})) \approx \omega_{2}(f,\sqrt{t})_{1}
\end{align}
and combined with the trivial estimate 
\begin{equation*}
\omega_{2}(f,t)_{1}\lesssim \omega_{1}(f,t)_{1}
\end{equation*}
we derive 
\begin{equation*}
K(t,f;L^{1}(\mathbb{R}^{N}),\dot{H}^{2,1}(\mathbb{R}%
^{N}))\lesssim\omega_{1}(f,\sqrt{t})_{1}
\end{equation*}
or equivalently (apply again \eqref{PD28.1}) 
\begin{equation}
K(t,f;L^{1}(\mathbb{R}^{N}),\dot{H}^{2,1}(\mathbb{R}^{N}))\lesssim
K(\sqrt{t},f;L^{1}(\mathbb{R}^{N}),\dot{W}^{1,1}(\mathbb{R}^{N})). 
\label{PJ6.17}
\end{equation}
Hence, we can estimate \eqref{PJ6.9} as follows 
\begin{equation}
\int_0^\infty \frac{\Vert H_{t}\mathbf{1}_{E_{\tau}^{f}}-\mathbf{1}%
_{E_{\tau}^{f}}\Vert_{L^{1}(\mathbb{R}^{N})}}{\sqrt{t}}\,d\tau\lesssim \int_0^\infty \frac{K(\sqrt{t},\mathbf{1}_{E_{\tau}^{f}};L^{1}(\mathbb{R}^{N}),%
\dot{W}^{1,1}(\mathbb{R}^{N}))}{\sqrt{t}}\,d\tau.   \label{PJ6.18}
\end{equation}
As a consequence, using the trivial fact that $\Gamma-\limsup$ of any family of functionals is dominated by the corresponding pointwise $\limsup$, we have 
\begin{align}
\Gamma_{L^{1}(\mathbb{R}^{N})}-\limsup_{t\rightarrow0}\int_0^\infty \frac{%
\Vert H_{t}\mathbf{1}_{E_{\tau}^{f}}-\mathbf{1}_{E_{\tau}^{f}}\Vert_{L^{1}(%
\mathbb{R}^{N})}}{\sqrt{t}}\,d\tau & \nonumber \\ 
& \hspace{-6cm}\lesssim\limsup_{t\rightarrow 0}\int_0^\infty \frac{K(\sqrt{t},\mathbf{1}_{E_{\tau}^{f}};L^{1}(\mathbb{R}^{N}),%
\dot{W}^{1,1}(\mathbb{R}^{N}))}{\sqrt{t}}\,d\tau.   \label{PJ6.10}
\end{align}
Note that, by \eqref{PJ1.1}, 
\begin{equation*}
\sup_{t>0}\frac{K(t,f;L^{1}(\mathbb{R}^{N}),\dot{W}^{1,1}(\mathbb{R}^{N}))}{t%
}=\Vert f\Vert_{G_{1}(L^{1}(\mathbb{R}^{N}),\dot{W}^{1,1}(\mathbb{R}%
^{N}))}\approx\Vert f\Vert_{\dot{BV}(\mathbb{R}^{N})}. 
\end{equation*}
Applying this to $\mathbf{1}_{E_{\tau}^{f}}$, we have 
\begin{equation}
\sup_{t>0}\frac{K(t,\mathbf{1}_{E_{\tau}^{f}};L^{1}(\mathbb{R}^{N}),\dot {W}%
^{1,1}(\mathbb{R}^{N}))}{t}\approx\Vert\mathbf{1}_{E_{\tau}^{f}}\Vert _{\dot{%
BV}(\mathbb{R}^{N})}   \label{PJ6.11}
\end{equation}
and, by the classical co-area formula for $BV$ (see e.g. 
\cite[Theorem 2.2]{CDDD}): 
\begin{equation}
\int_{0}^\infty \Vert\mathbf{1}_{E_{\tau}^{f}}\Vert_{\dot{BV}(%
\mathbb{R}^{N})}\,d\tau=\Vert f\Vert_{\dot{BV}(\mathbb{R}^{N})}<\infty. 
\label{PJ6.12}
\end{equation}
Accordingly, one can apply the reverse Fatou's lemma together with a simple change of
variables to get that 
\begin{align*}
\limsup_{t\rightarrow0}\int_0^\infty \frac{K(\sqrt{t},\mathbf{1}_{E_{\tau
}^{f}};L^{1}(\mathbb{R}^{N}),\dot{W}^{1,1}(\mathbb{R}^{N}))}{\sqrt{t}}%
\,d\tau & \leq\int_0^\infty \limsup_{t\rightarrow0}\frac{K(t,\mathbf{1}%
_{E_{\tau }^{f}};L^{1}(\mathbb{R}^{N}),\dot{W}^{1,1}(\mathbb{R}^{N}))}{t}%
\,d\tau \\
& \hspace{-5cm}=\int_0^\infty \sup_{t>0}\frac{K(t,\mathbf{1}%
_{E_{\tau}^{f}};L^{1}(\mathbb{R}^{N}),\dot{W}^{1,1}(\mathbb{R}^{N}))}{t}%
\,d\tau \approx\Vert f\Vert_{\dot{BV}(\mathbb{R}^{N})},
\end{align*}
where the last step follows from \eqref{PJ6.11}-\eqref{PJ6.12}. Plugging
this into \eqref{PJ6.10}, we arrive at 
\begin{equation}
\Gamma_{L^{1}(\mathbb{R}^{N})}-\limsup_{t\rightarrow0}\int_0^\infty \frac{%
\Vert H_{t}\mathbf{1}_{E_{\tau}^{f}}-\mathbf{1}_{E_{\tau}^{f}}\Vert_{L^{1}(%
\mathbb{R}^{N})}}{\sqrt{t}}\,d\tau\lesssim\Vert f\Vert_{\dot {BV}(\mathbb{R}%
^{N})}.   \label{PJ6.13}
\end{equation}

Next we focus on the lower estimate of \eqref{PJ6.6}. To achieve this, we
will apply the following converse estimate to \eqref{PJ.14} that may be
found in \cite[Theorem 7, (4.3)]{W} 
\begin{equation}
\omega_{4}(f,\sqrt{t})_{1}\lesssim K(t,f;L^{1}(\mathbb{R}^{N}),\dot {H}^{2,1}(\mathbb{R}^{N})).   \label{PJ6.15}
\end{equation}
Moreover, in the special case that $f=\mathbf{1}_{E}$ for a set $E\subset 
\mathbb{R}^{N}$ with $|E| < \infty$, it is routine to check that 
\begin{equation*}
|\Delta_{h}\mathbf{1}_{E}(x)|\leq|\Delta_{h}^{4}\mathbf{1}_{E}(x)| 
\end{equation*}
for every $x,h\in\mathbb{R}^{N}$ and then 
\begin{equation*}
\omega_{1}(\mathbf{1}_{E},t)_{1}\leq\omega_{4}(\mathbf{1}_{E},t)_{1},\qquad%
\forall t>0. 
\end{equation*}
Thanks to \eqref{PD28.1}, we obtain 
\begin{equation}
K(t,\mathbf{1}_{E};L^{1}(\mathbb{R}^{N}),\dot{W}^{1,1}(\mathbb{R}%
^{N}))\lesssim\omega_{4}(\mathbf{1}_{E},t)_{1},\qquad\forall t>0. 
\label{PJ6.16}
\end{equation}
Putting together \eqref{PJ6.15} and \eqref{PJ6.16}, we have 
\begin{equation*}
K(\sqrt{t},\mathbf{1}_{E};L^{1}(\mathbb{R}^{N}),\dot{W}^{1,1}(\mathbb{R}%
^{N}))\lesssim K(t,\mathbf{1}_{E};L^{1}(\mathbb{R}^{N}),\dot{H}%
^{2,1}(\mathbb{R}^{N})), 
\end{equation*}
which shows that the converse estimate of \eqref{PJ6.17} also holds for
characteristic functions associated to  sets of finite measure (but not for general functions in $L^1(\R^N)$). Consequently, from 
\eqref{PJ6.9}, we get 
\begin{equation*}
\int_0^\infty \frac{\Vert H_{t}\mathbf{1}_{E_{\tau}^{f}}-\mathbf{1}%
_{E_{\tau}^{f}}\Vert_{L^{1}(\mathbb{R}^{N})}}{\sqrt{t}}\,d\tau\gtrsim  \int_0^\infty \frac{K(\sqrt{t},\mathbf{1}_{E_{\tau}^{f}};L^{1}(\mathbb{R}^{N}),%
\dot{W}^{1,1}(\mathbb{R}^{N}))}{\sqrt{t}}\,d\tau. 
\end{equation*}
Then a simple application of the Fatou's lemma in the setting of $\Gamma $%
-convergence leads to 
\begin{align}
\Gamma_{L^{1}(\mathbb{R}^{N})}-\liminf_{t\rightarrow0}\int_0^\infty \frac{%
\Vert H_{t}\mathbf{1}_{E_{\tau}^{f}}-\mathbf{1}_{E_{\tau}^{f}}\Vert_{L^{1}(%
\mathbb{R}^{N})}}{\sqrt{t}}\,d\tau & \nonumber \\
& \hspace{-6cm} \gtrsim\Gamma _{L^{1}(\mathbb{R}%
^{N})}-\liminf_{t\rightarrow0}\int_0^\infty \frac {K(\sqrt{t},\mathbf{1}%
_{E_{\tau}^{f}};L^{1}(\mathbb{R}^{N}),\dot{W}^{1,1}(\mathbb{R}^{N}))}{\sqrt{t%
}}\,d\tau  \notag \\
&\hspace{-6cm}  \geq\int_0^\infty \Gamma_{L^{1}(\mathbb{R}^{N})}-\liminf_{t\rightarrow
0}\frac{K(t,\mathbf{1}_{E_{\tau}^{f}};L^{1}(\mathbb{R}^{N}),\dot{W}^{1,1}(%
\mathbb{R}^{N}))}{t}\,d\tau,   \label{PJ6.19}
\end{align}
where we have also applied a change of variables ($\sqrt{t}\leftrightarrow t$%
) in the last step. Invoke now Theorem \ref{ThmGKFunct} with $%
(A_{0},A_{1})=(L^{1}(\mathbb{R}^{N}),\dot{W}^{1,1}(\mathbb{R}^{N}))$ and %
\eqref{PJ1.1} to find that 
\begin{equation*}
\Gamma_{L^{1}(\mathbb{R}^{N})}-\lim_{t\rightarrow0}\frac{K(t,\mathbf{1}%
_{E_{\tau}^{f}};L^{1}(\mathbb{R}^{N}),\dot{W}^{1,1}(\mathbb{R}^{N}))}{t}%
=\Vert\mathbf{1}_{E_{\tau}^{f}}\Vert_{G_{1}(L^{1}(\mathbb{R}^{N}),\dot {W}%
^{1,1}(\mathbb{R}^{N}))}\approx\Vert\mathbf{1}_{E_{\tau}^{f}}\Vert _{\dot{BV}%
(\mathbb{R}^{N})}
\end{equation*}
with equivalence constants that are independent of $f$ and $\tau$. Using
this estimate into \eqref{PJ6.19} and applying \eqref{PJ6.12}, we get 
\begin{equation}
\Gamma_{L^{1}(\mathbb{R}^{N})}-\liminf_{t\rightarrow0}\int_0^\infty \frac{%
\Vert H_{t}\mathbf{1}_{E_{\tau}^{f}}-\mathbf{1}_{E_{\tau}^{f}}\Vert_{L^{1}(%
\mathbb{R}^{N})}}{\sqrt{t}}\,d\tau\gtrsim\int_0^\infty \Vert\mathbf{1}%
_{E_{\tau}^{f}}\Vert_{\dot{BV}(\mathbb{R}^{N})}\,d\tau=\Vert f\Vert_{\dot{BV}%
(\mathbb{R}^{N})}.   \label{PJ6.20}
\end{equation}

Finally, by \eqref{PJ6.6}, \eqref{PJ6.13}, and \eqref{PJ6.20}, we  
conclude that 
\begin{align*}
\Gamma_{L^{1}(\mathbb{R}^{N})}-\liminf_{t\rightarrow0}\frac{1}{\sqrt{t}}%
\int_{\mathbb{R}^{N}}H_{t}(|f-f(x)|)(x)\,dx & \\
& \hspace{-5cm}\approx\Gamma_{L^{1}(\mathbb{R}^{N})}-\limsup_{t\rightarrow 0}%
\frac{1}{\sqrt{t}}\int_{\mathbb{R}^{N}}H_{t}(|f-f(x)|)(x)\,dx\approx\Vert
f\Vert_{\dot{BV}(\mathbb{R}^{N})}
\end{align*}
for $0 \leq f \in BV(\R^N)$. 

The general case $f \in BV(\R^N)$ follows from the previous case via standard arguments based on the  decomposition $f = f^{+}-f^{-}$ with $f^{+}, f^{-} \geq 0$. Note that $$\|f\|_{\dot{BV}(\R^N)} = \|f^{+}\|_{\dot{BV}(\R^N)}+ \|f^{-}\|_{\dot{BV}(\R^N)}$$ and (cf. \eqref{uba}) 
$$
	\int_{\mathbb{R}^{N}}H_{t}(|f-f(x)|)(x)\,dx= \int_{\mathbb{R}^{N}}H_{t}(|f^{+}-f^{+}(x)|)(x)\,dx  +  \int_{\mathbb{R}^{N}}H_{t}(|f^{-}-f^{-}(x)|)(x)\,dx.  
$$
\end{proof}

The extension of Theorem \ref{ThmJ22.1} to Sobolev and Lebesgue spaces is given in the
following result.

\begin{theorem}
\label{TheoremJ22.5} Let $p\in\lbrack1,\infty)$. If $f\in W^{1,p}(%
\mathbb{R}^{N})$ then 
\begin{align}
 \Vert
f\Vert_{\dot{W}^{1,p}(\mathbb{R}^{N})}^{p} &\lesssim \Gamma_{L^{p}(\mathbb{R}^{N})}-\liminf_{t\rightarrow0}\frac{1}{t^{p/2}}\int_{%
\mathbb{R}^{N}}H_{t}(|f-f(x)|^{p})(x)\,dx   \notag \\
& \leq \Gamma_{L^{p}(\mathbb{R}^{N})}-\limsup_{t\rightarrow 0}%
\frac{1}{t^{p/2}}\int_{\mathbb{R}^{N}}H_{t}(|f-f(x)|^{p})(x)\,dx\lesssim \Vert
f\Vert_{\dot{W}^{1,p}(\mathbb{R}^{N})}^{p}   \label{PJ21.2}
\end{align}
and 
\begin{align}
\Vert f\Vert_{L^{p}(\mathbb{R}^{N})}^{p} &\lesssim \Gamma_{\dot{W}^{1,p}(\mathbb{R}^{N})}-\liminf_{t\rightarrow\infty}\int_{%
\mathbb{R}^{N}}H_{t}(|f-f(x)|^{p})(x)\,dx   \notag \\
& \leq \Gamma_{\dot{W}^{1,p}(\mathbb{R}^{N})}-\limsup
_{t\rightarrow\infty}\int_{\mathbb{R}^{N}}H_{t}(|f-f(x)|^{p})(x)\,dx\lesssim 
\Vert f\Vert_{L^{p}(\mathbb{R}^{N})}^{p}.   \label{PJ22.4}
\end{align}
\end{theorem}

\begin{remark}
The $\Gamma$-limit in \eqref{PJ21.2} was already obtained in \cite[Theorem
1.6]{GS} by a different approach. On the other hand, \eqref{PJ22.4} is apparently new (in particular, in view of the method presented in Section \ref{SectionExact} below,  it seems plausible  to update its mode of $\Gamma$-convergence to $L^p$ and compute exact formulae). 
\end{remark}

\begin{proof}[Proof of Theorem \protect\ref{TheoremJ22.5}]
Given $t>0$, we can apply simple changes of variables and Fubini's theorem
to get 
\begin{align*}
\int_{\mathbb{R}^{N}}H_{t}(|f-f(x)|^{p})(x)\,dx & =\int_{\mathbb{R}%
^{N}}\int_{\mathbb{R}^{N}}h_{t}(\xi)|f(x+\xi)-f(x)|^{p}\,d\xi\,dx \\
& =\int_{\mathbb{R}^{N}}h_{t}(\xi)\Vert\Delta_{\xi}f\Vert_{L^{p}(\mathbb{R}%
^{N})}^{p}\,d\xi \\
& =\int_{\mathbb{R}^{N}}h_{1}(\xi)\Vert\Delta_{\sqrt{t}\xi}f\Vert _{L^{p}(%
\mathbb{R}^{N})}^{p}\,d\xi.
\end{align*}
As a consequence, we derive 
\begin{equation}
\frac{1}{t^{p/2}}\int_{\mathbb{R}^{N}}H_{t}(|f-f(x)|^{p})(x)\,dx=\int_{%
\mathbb{R}^{N}}h_{1}(\xi)\bigg(\frac{\Vert\Delta_{\sqrt{t}\xi}f\Vert _{L^{p}(%
\mathbb{R}^{N})}}{\sqrt{t}}\bigg)^{p}\,d\xi.   \label{PJ11.3}
\end{equation}

For each fixed $\xi\in\mathbb{R}^{N}\backslash\{0\}$, let us denote by $\dot{%
W}_{\xi}^{1,p}(\mathbb{R}^{N})$ the directional Sobolev space equipped with 
\begin{equation*}
\Vert f\Vert_{\dot{W}_{\xi}^{1,p}(\mathbb{R}^{N})}:=\Vert D_{\xi}f\Vert
_{L^{p}(\mathbb{R}^{N})}, 
\end{equation*}
where $D_{\xi}f(x) := \langle \nabla f(x), \xi \rangle$. We also let 
\begin{equation*}
T_{t,\xi}f(x):=f(x+t\xi) 
\end{equation*}
for $t>0$ and $x\in\mathbb{R}^{N}$. It is plain to see that $\{T_{t,\xi
}\}_{t>0}$ defines a semigroup of linear operators acting on $L^{p}(\mathbb{R%
}^{N})$ with corresponding generator $\Lambda_{\xi}f=D_{\xi}f.$ In
particular, the domain of $\Lambda_{\xi}$ coincides with $\dot{W}%
_{\xi}^{1,p}(\mathbb{R}^{N})$, more precisely, 
\begin{equation*}
\Vert f\Vert_{D(\Lambda_{\xi})}=\Vert f\Vert_{\dot{W}_{\xi}^{1,p}(\mathbb{R}%
^{N})}. 
\end{equation*}
By virtue of \cite[Theorem 6.7.3]{BL}, we can characterize the $K$%
-functional relative to $(L^{p}(\mathbb{R}^{N}),D(\Lambda_{\xi}))=(L^{p}(%
\mathbb{R}^{N}),\dot{W}_{\xi}^{1,p}(\mathbb{R}^{N}))$ as 
\begin{equation}
K(t,f;L^{p}(\mathbb{R}^{N}),\dot{W}_{\xi}^{1,p}(\mathbb{R}^{N}))\approx
\sup_{0<s\leq t}\Vert T_{s,\xi}f-f\Vert_{L^{p}(\mathbb{R}^{N})}=\sup_{0<s\leq t}%
\Vert\Delta_{s\xi}f\Vert_{L^{p}(\mathbb{R}^{N})},   \label{PJ11.1}
\end{equation}
where the hidden equivalence constant may be taken to be independent of $\xi$.


Using \eqref{PJ11.1} in the formula \eqref{PJ11.3}, we obtain the following
estimate 
\begin{equation}
\frac{1}{t^{p/2}}\int_{\mathbb{R}^{N}}H_{t}(|f-f(x)|^{p})(x)\,dx\lesssim
\int_{\mathbb{R}^{N}}h_{1}(\xi)\bigg(\frac{K(\sqrt{t},f;L^{p}(\mathbb{R}%
^{N}),\dot{W}_{\xi}^{1,p}(\mathbb{R}^{N}))}{\sqrt{t}}\bigg)^{p}\,d\xi. 
\label{PJ11.4}
\end{equation}
Moreover, by \eqref{sec:sera} and Cauchy--Schwarz inequality, 
\begin{align}
\frac{1}{t^{p/2}}\int_{\mathbb{R}^{N}}H_{t}(|f-f(x)|^{p})(x)\,dx &  \lesssim \int_{\mathbb{R}^{N}}h_{1}(\xi)\sup_{t > 0}\bigg(\frac{K(\sqrt{t},f;L^{p}(\mathbb{R}%
^{N}),\dot{W}_{\xi}^{1,p}(\mathbb{R}^{N}))}{\sqrt{t}}\bigg)^{p}\,d\xi \nonumber \\
& = \int_{\R^N} h_1(\xi) \|f\|_{G_1(L^p(\R^N), \dot{W}^{1, p}_\xi(\R^N))}^p  \, d \xi\ \nonumber\\
&\leq\int_{\mathbb{R}^{N}}h_{1}(\xi)\Vert f\Vert_{\dot {W}%
_{\xi}^{1,p}(\mathbb{R}^{N})}^{p}\,d\xi \nonumber \\ 
&\leq\bigg(\int_{\mathbb{R}^{N}}h_{1}(\xi)|\xi|^{p}\,d\xi%
\bigg)\Vert f\Vert_{\dot{W}^{1,p}(\mathbb{R}^{N})}^{p}.   \label{PJ11.4a}
\end{align}
Since $h_{1}(\xi)$ is a radial function (recall \eqref{DefHeatKernel}), we
can apply polar coordinates together with a simple change of variables to show the finiteness of the last integral,
more precisely, 
\begin{equation}
\int_{\mathbb{R}^{N}}h_{1}(\xi)|\xi|^{p}\,d\xi=c_{N}\int_{0}^{%
\infty}h_{1}(u)u^{p+N-1}\,du=c_{N}^{\prime}\int_{0}^{%
\infty}e^{-u^{2}}u^{p+N-1}\,du<\infty.   \label{PJ11.5}
\end{equation}
In particular, by \eqref{PJ11.4a} and \eqref{PJ11.5}, 
\begin{align}
\Gamma_{L^{p}(\mathbb{R}^{N})}-\limsup_{t\rightarrow0}\frac{1}{t^{p/2}}\int_{%
\mathbb{R}^{N}}H_{t}(|f-f(x)|^{p})(x)\,dx &  \notag \\
& \hspace{-6cm}\lesssim\limsup_{t\rightarrow0} \frac{1}{t^{p/2}}\int_{%
\mathbb{R}^{N}}H_{t}(|f-f(x)|^{p})(x)\,dx  \lesssim \Vert f\Vert_{\dot{W}^{1,p}(\mathbb{R}^{N})}^{p}.
\label{PJ11.6}
\end{align}

On the other hand, by \eqref{PJ11.3}, the monotonicity of heat kernels (cf. %
\eqref{DefHeatKernel}), and a simple change of variables, we have 
\begin{align}
\frac{1}{t^{p/2}}\int_{\mathbb{R}^{N}}H_{t}(|f-f(x)|^{p})(x)\,dx & \geq
\int_{|\xi|<1}h_{1}(\xi)\bigg(\frac{\Vert\Delta_{\sqrt{t}\xi}f\Vert _{L^{p}(%
\mathbb{R}^{N})}}{\sqrt{t}}\bigg)^{p}\,d\xi \nonumber \\
& \gtrsim\frac{1}{t^{\frac{p+N}{2}}}\int_{|\xi|<\sqrt{t}}\Vert\Delta_{\xi
}f\Vert_{L^{p}(\mathbb{R}^{N})}^{p}\,d\xi. \label{20.9.3}
\end{align}
Then, by virtue of \eqref{PD28.1} (with $k=1,q=p$ and $X=L^{p}(\mathbb{R}%
^{N})$), the previous estimate reads as 
\begin{equation*}
\frac{1}{t^{p/2}}\int_{\mathbb{R}^{N}}H_{t}(|f-f(x)|^{p})(x)\,dx\gtrsim %
\bigg(\frac{K(\sqrt{t},f;L^{p}(\mathbb{R}^{N}),\dot{W}^{1,p}(\mathbb{R}^{N}))%
}{\sqrt{t}}\bigg)^{p}.  
\end{equation*}
We are now in a position to invoke \eqref{PJ22.2} with $(A_{0},A_{1})=(L^{p}(%
\mathbb{R}^{N}),\dot{W}^{1,p}(\mathbb{R}^{N}))$ so that 
\begin{align}
\Gamma_{L^{p}(\mathbb{R}^{N})}-\liminf_{t\rightarrow0}\frac{1}{t^{p/2}}\int_{%
\mathbb{R}^{N}}H_{t}(|f-f(x)|^{p})(x)\,dx &  \notag \\
& \hspace{-5.5cm}\gtrsim\Gamma_{L^{p}(\mathbb{R}^{N})}-\liminf_{t\rightarrow
0}\bigg(\frac{K(t,f;L^{p}(\mathbb{R}^{N}),\dot{W}^{1,p}(\mathbb{R}^{N}))}{t}%
\bigg)^{p}  \notag \\
& \hspace{-5.5cm}=\Vert f\Vert_{G_{1}(L^{p}(\mathbb{R}^{N}),\dot{W}^{1,p}(%
\mathbb{R}^{N}))}^{p}\approx\Vert f\Vert_{\dot{W}^{1,p}(\mathbb{R}%
^{N})}^{p},   \label{PJ21.1}
\end{align}
where the last step follows from \eqref{PD28.5i}.

Putting together \eqref{PJ11.6} and \eqref{PJ21.1}, we arrive at the desired
estimates  \eqref{PJ21.2}. On the other hand, the proof of %
\eqref{PJ22.4} can be obtained in a similar fashion but now relying instead
on \eqref{PJ22.3}.
\end{proof}

Our method also works with fractional heat kernels. More precisely,  $%
(H_{t}^{s})_{t>0}$ denotes the fractional heat semigroup with $s\in(0,1)$
given by 
\begin{equation*}
H_{t}^{s}f=h_{t}^{s}\ast f. 
\end{equation*}
It can be shown that 
\begin{equation*}
h_{t}^{s}(x)\approx\frac{t}{\big(t^{\frac{1}{s}}+|x|^{2}\big)^{\frac{N+2s}{2}%
}}   
\end{equation*}
for all $x\in\mathbb{R}^{N}$ and $t>0$. However, apart from the
distinguished case $s=1/2$, there is not an explicit formula for $h_{t}^{s}$%
, cf. \cite[Section 6.2]{GS} for a concise overview on fractional heat
kernels.

The extension of \eqref{PJ21.2} to fractional heat kernels may be found in
the following result, which was already obtained in \cite[Theorem 1.7]{GS}.
Next we give a new proof of this result based on the $\Gamma$-extrapolation
approach.

\begin{theorem}
Let $s\in(0,1)$ and $p\in[1,\infty)$ with\footnote{We restrict our attention to the most interesting case $2 s \geq p$. We mention that, for the remaining range $2 s < p$, somewhat expected limit formulae hold, cf. \cite[Theorem 1.7]{GS}.} $2s\geq p$. If $f\in W^{1,p}(\mathbb{R}^{N})$ then%
\begin{align*}
\Vert f\Vert_{\dot{W}^{1,p}(\mathbb{R}^{N})}^{p} & \lesssim \Gamma_{L^{p}(\mathbb{R}^{N})}-\liminf_{t\rightarrow0}\frac{1}{\psi_{s,p}(t)}%
\int_{\mathbb{R}^{N}}H_{t}^{s}(|f-f(x)|^{p})(x)\,dx   \\
& \leq  \Gamma_{L^{p}(\mathbb{R}^{N})}-\limsup_{t\rightarrow0}\frac{1}{\psi
_{s,p}(t)}\int_{\mathbb{R}^{N}}H_{t}^{s}(|f-f(x)|^{p})(x)\,dx \lesssim \Vert f\Vert_{\dot{W}^{1,p}(\mathbb{R}^{N})}^{p},
\end{align*}
where 
\begin{equation*}
\psi_{s,p}(t)=\left\{ 
\begin{array}{cl}
t^{\frac{p}{2s}} &\qquad  \text{if}\qquad2s>p, \\ 
t|\log t| &\qquad  \text{if}\qquad2s=p.%
\end{array}
\right. 
\end{equation*}
\end{theorem}

\begin{proof}
The proof in the case $2s>p$ follows line by line the argument given in the
proof of Theorem \ref{TheoremJ22.5} and then it is omitted. Therefore, we
now focus on the limiting case $2s=p$, where some changes, albeit still
minor, are needed.

Assume that $2s=p$. Applying a simple change of variables, we can write 
\begin{equation}
\frac{1}{\psi_{s,p}(t)}\int_{\mathbb{R}^{N}}H_{t}^{s}(|f-f(x)|^{p})(x)%
\,dx=t^{\frac{N}{2s}}\int_{\mathbb{R}^{N}}\frac{h_{t}^{s}(t^{\frac{1}{2s}%
}\xi)}{\psi_{s,p}(t)}\Vert\Delta_{t^{\frac{1}{2s}}\xi}f\Vert_{L^{p}(\mathbb{R%
}^{N})}^{p}\,d\xi.   \label{PJ25.2}
\end{equation}
Accordingly, by \eqref{PJ11.1}, we have 
\begin{equation}
\frac{1}{\psi_{s,p}(t)}\int_{\mathbb{R}^{N}}H_{t}^{s}(|f-f(x)|^{p})(x)\,dx%
\lesssim I+II+III,   \label{PJ25.6}
\end{equation}
where 
\begin{equation*}
I:=\frac{1}{|\log t|}\int_{|\xi|<1}\frac{1}{(1+|\xi|^{2})^{\frac{N+2s}{2}}}%
\bigg(\frac{K(t^{\frac{1}{p}},f;L^{p}(\mathbb{R}^{N}),\dot{W}_{\xi}^{1,p}(%
\mathbb{R}^{N}))}{t^{\frac{1}{p}}}\bigg)^{p}\,d\xi, 
\end{equation*}%
\begin{equation*}
II:=\frac{1}{|\log t|}\int_{1<|\xi|<\frac{1}{t}}\frac{1}{(1+|\xi|^{2})^{%
\frac{N+2s}{2}}}\bigg(\frac{K(t^{\frac{1}{p}},f;L^{p}(\mathbb{R}^{N}),\dot{W}%
_{\xi}^{1,p}(\mathbb{R}^{N}))}{t^{\frac{1}{p}}}\bigg)^{p}\,d\xi, 
\end{equation*}
and 
\begin{equation*}
III:=\frac{1}{|\log t|}\int_{|\xi|>\frac{1}{t}}\frac{1}{(1+|\xi|^{2})^{\frac{%
N+2s}{2}}}\bigg(\frac{K(t^{\frac{1}{p}},f;L^{p}(\mathbb{R}^{N}),\dot{W}%
_{\xi}^{1,p}(\mathbb{R}^{N}))}{t^{\frac{1}{p}}}\bigg)^{p}\,d\xi. 
\end{equation*}
Here, without loss of generality, we may assume that $t<1$. We now
estimate each of these three terms. Firstly,  $I$ can be controlled using ideas
similar to those that led to \eqref{PJ11.4}: 
\begin{equation}
I\leq\frac{1}{|\log t|}\bigg(\int_{|\xi|<1}\frac{|\xi|^{p}}{(1+|\xi |^{2})^{%
\frac{N+2s}{2}}}\,d\xi\bigg)\Vert f\Vert_{\dot{W}^{1,p}(\mathbb{R}%
^{N})}^{p}\approx\frac{1}{|\log t|}\Vert f\Vert_{\dot{W}^{1,p}(\mathbb{R}%
^{N})}^{p}.   \label{PJ25.7}
\end{equation}
To estimate $III$, we use \eqref{sec:sera} to get 
\begin{align}
	III &\leq \frac{1}{t |\log t|} \int_{|\xi| > \frac{1}{t}} \frac{1}{(1+|\xi|^2)^{\frac{N+2 s}{2}}} \sup_{t > 0} K(t^{\frac{1}{p}}, f; L^p(\R^N), \dot{W}^{1, p}_\xi(\R^N))^p \, d \xi \nonumber \\
	& =   \frac{1}{t |\log t|} \int_{|\xi| > \frac{1}{t}} \frac{1}{(1+|\xi|^2)^{\frac{N+2 s}{2}}} \|f\|_{G_0(L^p(\R^N), \dot{W}^{1, p}_\xi(\R^N))}^p \, d \xi \nonumber \\
	& \leq \frac{1}{t |\log t|}  \bigg( \int_{|\xi| > \frac{1}{t}} \frac{1}{(1+|\xi|^2)^{\frac{N+2 s}{2}}}  \, d \xi  \bigg) \|f\|_{L^p(\R^N)}^p \nonumber \\
	& \approx  \frac{1}{t |\log t|}  \bigg( \int_{|\xi| > \frac{1}{t}} \frac{1}{|\xi|^{N+ 2 s}}  \, d \xi  \bigg) \|f\|_{L^p(\R^N)}^p  \approx 
\frac{t^{p-1}}{|\log t|}\Vert f\Vert_{L^{p}(\mathbb{R}^{N})}^{p}. \label{PJ25.8}
\end{align}
Finally, we can estimate $II$ as follows: 
\begin{equation}
II\lesssim\frac{1}{|\log t|}\bigg(\int_{1<|\xi|<\frac{1}{t}}\frac{1}{|\xi
|^{N}}\,d\xi\bigg)\Vert f\Vert_{\dot{W}^{1,p}(\mathbb{R}^{N})}^{p}\approx%
\Vert f\Vert_{\dot{W}^{1,p}(\mathbb{R}^{N})}^{p}.   \label{PJ25.9}
\end{equation}
Inserting \eqref{PJ25.7}--\eqref{PJ25.9} into \eqref{PJ25.6}, we arrive at 
\begin{equation*}
\frac{1}{\psi_{s,p}(t)}\int_{\mathbb{R}^{N}}H_{t}^{s}(|f-f(x)|^{p})(x)\,dx%
\lesssim\bigg(1+\frac{1}{|\log t|}\bigg)\Vert f\Vert_{\dot{W}^{1,p}(\mathbb{R%
}^{N})}^{p}+\frac{t^{p-1}}{|\log t|}\Vert f\Vert _{L^{p}(\mathbb{R}%
^{N})}^{p}. 
\end{equation*}
Consequently, 
\begin{align}
\Gamma_{L^{p}(\mathbb{R}^{N})}-\limsup_{t\rightarrow0}\frac{1}{\psi_{s,p}(t)}%
\int_{\mathbb{R}^{N}}H_{t}^{s}(|f-f(x)|^{p})(x)\,dx &  \notag \\
& \hspace{-6cm}\lesssim\limsup_{t\rightarrow0}\bigg[\bigg(1+\frac{1}{|\log t|%
}\bigg)\Vert f\Vert_{\dot{W}^{1,p}(\mathbb{R}^{N})}^{p}+\frac{t^{p-1}}{|\log
t|}\Vert f\Vert_{L^{p}(\mathbb{R}^{N})}^{p}\bigg]=\Vert f\Vert_{\dot{W}%
^{1,p}(\mathbb{R}^{N})}^{p}.   \label{PJ26.3}
\end{align}

Conversely, by \eqref{PJ25.2} and a change of variables we see that 
\begin{equation}\label{PJ26.1}
\frac{1}{\psi_{s,p}(t)}\int_{\mathbb{R}^{N}}H_{t}^{s}(|f-f(x)|^{p})(x)\,dx%
\gtrsim\frac{1}{|\log t|}\int_{t^{\frac{1}{p}}<|\xi|<1}\frac {%
\Vert\Delta_{\xi}f\Vert_{L^{p}(\mathbb{R}^{N})}^{p}}{|\xi|^{N+p}}\,d\xi. 
\end{equation}
Moreover, the $\Gamma$-limit of the right-hand side of \eqref{PJ26.1} is a
prototypical example of the classical BBM formula, namely, 
\begin{equation}
\Gamma_{L^{p}(\mathbb{R}^{N})}-\lim_{t\rightarrow0}\frac{1}{|\log t|}%
\int_{t<|\xi|<1}\frac{\Vert\Delta_{\xi}f\Vert_{L^{p}(\mathbb{R}^{N})}^{p}}{%
|\xi|^{N+p}}\,d\xi=c\Vert f\Vert_{\dot{W}^{1,p}(\mathbb{R}^{N})}^{p}; 
\label{PJ26.2}
\end{equation}
cf. \cite[Section 6.7.5]{BM}. Then it follows from \eqref{PJ26.1} and %
\eqref{PJ26.2} that 
\begin{equation*}
\Gamma_{L^{p}(\mathbb{R}^{N})}-\liminf_{t\rightarrow0}\frac{1}{\psi_{s,p}(t)}%
\int_{\mathbb{R}^{N}}H_{t}^{s}(|f-f(x)|^{p})(x)\,dx\gtrsim\Vert f\Vert_{\dot{%
W}^{1,p}(\mathbb{R}^{N})}^{p}, 
\end{equation*}
which combined with \eqref{PJ26.3} gives the desired result in the limiting
case $2s=p$.
\end{proof}

\begin{remark}
	The formula  \eqref{PJ26.2} is a special case of the general BBM formulation: 
	\begin{equation}\label{16.9.1}
		\Gamma_{L^p(\R^N)} - \lim_{t \to 0} \int_{\R^N} \|\Delta_\xi f\|_{L^p(\R^N)}^p \phi_t(|\xi|) \, d \xi = c \Vert f\Vert_{\dot{W}^{1,p}(\mathbb{R}^{N})}^{p}
	\end{equation}
	for a suitable family $\{\phi_t\}$ of nonnegative kernels. One can show that  \eqref{16.9.1} is also part of a general extrapolation phenomenon in the spirit of Theorem \ref{Thm1}, but now formulated in terms of  non-standard interpolation weights, cf. \cite{DLTYY}. However, this extension requires some further technicalities and we believe it deserves a separate study. 
\end{remark}

\section{Exact formulae for $\Gamma$-limits}

\label{SectionExact} In this section we confront an apparent practical
limitation of the  interpolation/extrapolation method: explicit formulae
for the $K$-functional are only available in a handful of classical cases.
Consequently, although the limiting procedures we employ are exact (cf. Theorem \ref{Thm1}), they
typically characterize the target norms only up to multiplicative constants.
Indeed, beyond some classical pairs such as $(L^{1},L^{\infty})$, $(C,C^{1})$, or $%
(L^{\infty},L^{\infty}(1/t))$, one generally possesses merely approximate
characterizations of the form 
\begin{equation}
K(t,f;A_{0},A_{1})\approx G(t,f),\qquad f\in A_{0}+A_{1},\ t>0, 
\label{PJ22.6}
\end{equation}
where $G(t,f)$ is an explicit, computable expression.

The goal of this section is to show that this apparent drawback can be overcome by a straightforward
refinement of the method. A moment's reflection shows that one can
effectively turn \eqref{PJ22.6} on its head: the quantities actually
computed in our arguments carried out in Section \ref{sec:app&exples} are, for all practical purposes, limits of the
expressions 
\begin{equation}
\Vert f\Vert_{\theta,p,G}^{\blacktriangleleft}:=c_{\theta,p}\left(
\int_{0}^{\infty}t^{-\theta p}G(t,f)^{p}\,\frac{dt}{t}\right) ^{1/p}. 
\label{PJ22.7}
\end{equation}
It therefore suffices to repeat the arguments of Proposition \ref{prop:lim}
and Theorem \ref{Thm1} with $G(t,f)$ in place of the true $K$-functional,
while carefully tracking the resulting constants (which turn out to be
explicitly computable in the situations of interest). In this way one can
normalize and recover the exact interpolation norms and their limiting
behavior.

To make these comments more rigorous, we now discuss in some detail the
model case studied in Section \ref{SubSecSobolevGamma} corresponding to the
pair $(A_{0},A_{1})=(X,\dot{W}^{k}X)$, where $X$ is a r.i. space on $\mathbb{%
R}^{N}$ equipped with an absolutely continuous norm\footnote{%
Recall the well-known fact that $X$ has absolutely continuous norm if and
only if $\Vert f\mathbf{1}_{E_{n}}\Vert_{X}\rightarrow0$ for every sequence $%
E_{n}\downarrow\emptyset$; cf. \cite[Proposition 3.2, Chapter 1]{BS}.} and $p \in [1, \infty)$. In
light of the characterization of the $K$-functional given in \eqref{PD28.1},
a natural candidate for $G$ in \eqref{PJ22.6} and \eqref{PJ22.7} is 
\begin{equation}
G(t,f):=\bigg(\frac{1}{t^{N/k}}\int_{|h|<t^{1/k}}\Vert\Delta_{h}^{k}f\Vert
_{X}^{p}\,dh\bigg)^{1/p}.   \label{PJ23.10}
\end{equation}
In particular, the following convolution-type estimate is satisfied: If $%
f\in X$ and $g\in L^{1}(\mathbb{R}^{N})$ then 
\begin{equation}
G(t,f\ast g)\leq G(t,f)\Vert g\Vert_{L^{1}(\mathbb{R}^{N})}.   \label{PJ23.1}
\end{equation}
To see this note that since every r.i. space $X$ is an exact interpolation
space between $L^{1}(\mathbb{R}^{N})$ and $L^{\infty}(\mathbb{R}^{N})$ (cf. 
\cite[Theorem 2.2, p. 106]{BS}), interpolating the endpoint cases we have 
\begin{equation}
\Vert f\ast g\Vert_{X}\leq\Vert f\Vert_{X}\Vert g\Vert_{L^{1}(\mathbb{R}%
^{N})}.   \label{PJ23.1new}
\end{equation}
Therefore (\ref{PJ23.1}) follows from the elementary identity $%
\Delta_{h}^{k}(f\ast g)=\Delta_{h}^{k}f\ast g$ and \eqref{PJ23.1new}. For $%
\delta>0$, we define 
\begin{equation*}
\eta_{\delta}(x):=\delta^{-N}\eta\Big(\frac{x}{\delta}\Big),\qquad x\in%
\mathbb{R}^{N}, 
\end{equation*}
where $\eta \geq 0$ is a fixed smooth function with compact support, $\Vert\eta
\Vert_{L^{1}(\mathbb{R}^{N})}=1$, and let $f_{\delta}:=f\ast\eta_{\delta}.$
Using \eqref{PJ22.7} and \eqref{PJ23.1}, we see that 
\begin{equation}
\Vert f_{\delta}\Vert_{\theta,p,G}^{\blacktriangleleft}\leq\Vert
f\Vert_{\theta,p,G}^{\blacktriangleleft}\Vert\eta_{\delta}\Vert_{L^{1}(%
\mathbb{R}^{N})}=\Vert f\Vert_{\theta,p,G}^{\blacktriangleleft}. 
\label{1.7}
\end{equation}


On the other hand, from the sub-additivity of $G$, i.e., 
\begin{equation*}
G(t,f+g)\leq G(t,f)+G(t,g) 
\end{equation*}
we immediately derive the triangle inequality for the functional $\Vert
\cdot\Vert_{\theta,p, G}^{\blacktriangleleft}$ introduced in \eqref{PJ22.7}: 
\begin{equation}
\Vert f+g\Vert_{\theta,p,G}^{\blacktriangleleft}\leq\Vert f\Vert_{\theta
,p,G}^{\blacktriangleleft}+\Vert g\Vert_{\theta,p,G}^{\blacktriangleleft}. 
\label{7.2}
\end{equation}

Assume for a moment that the pointwise limit as $\theta\rightarrow1^{-}$ of
the functional $\Vert f\Vert_{\theta,p,G}^{\blacktriangleleft}$ exists. Let  
\begin{equation}
L(f):=\lim_{\theta\rightarrow1^{-}}\Vert
f\Vert_{\theta,p,G}^{\blacktriangleleft }.   \label{PJ23.4}
\end{equation}
Let $\{f_{\theta}\}_{\theta\in(0,1)}$ be such that $f_{\theta}\rightarrow f$
in $X$ as $\theta\rightarrow1^{-}$, and let $f_{\theta,\delta}:=f_{\theta}%
\ast\eta_{\delta}.\,$By \eqref{1.7} 
\begin{equation}
\Vert f_{\theta,\delta}\Vert_{\theta,p,G}^{\blacktriangleleft}\leq\Vert
f_{\theta}\Vert_{\theta,p,G}^{\blacktriangleleft},   \label{7.1}
\end{equation}
which combined with \eqref{7.2}, yields 
\begin{align}
|\Vert f_{\theta,\delta}\Vert_{\theta,p,G}^{\blacktriangleleft}-L(f_{\delta
})| & \leq|\Vert
f_{\theta,\delta}\Vert_{\theta,p,G}^{\blacktriangleleft}-\Vert
f_{\delta}\Vert_{\theta,p,G}^{\blacktriangleleft }|+|\Vert
f_{\delta}\Vert_{\theta,p,G}^{\blacktriangleleft}-L(f_{\delta })|  \notag \\
& \leq\Vert
f_{\theta,\delta}-f_{\delta}\Vert_{\theta,p,G}^{\blacktriangleleft}+|\Vert
f_{\delta}\Vert_{\theta,p,G}^{\blacktriangleleft }-L(f_{\delta})|. 
\label{7.3n}
\end{align}
It follows from \eqref{PJ22.6} (cf. also \eqref{PD28.1}) and \eqref{sec2:2}
that, for every $\theta\in(0,1)$ and $f\in W^{k}X$, 
\begin{equation*}
\Vert f\Vert_{\theta,p,G}^{\blacktriangleleft}\approx\Vert f\Vert_{(X,\dot {W%
}^{k}X)_{\theta,p}^{\blacktriangleleft}}\leq\Vert f\Vert_{W^{k}X}
\end{equation*}
and, in particular, 
\begin{equation}
\Vert f_{\theta,\delta}-f_{\delta}\Vert_{\theta,p,G}^{\blacktriangleleft
}\lesssim\Vert f_{\theta,\delta}-f_{\delta}\Vert_{W^{k}X},   \label{7.4}
\end{equation}
where the hidden equivalence constant is independent of $\theta$ and $%
\delta$. The fact that $f_{\theta}\rightarrow f$ in $X$ as $\theta \to 1^{-}$ guarantees that $f_{\theta,\delta}\rightarrow f_{\delta}$ in $%
W^{k}X$ as $\theta \to 1^{-}$ requires a proof. On the one hand, the convergence in $X$ follows
trivially from the fact that (cf. \eqref{PJ23.1new}) 
\begin{equation}
\Vert f_{\theta,\delta}-f_{\delta}\Vert_{X}\leq\Vert f_{\theta}-f\Vert_{X}, 
\label{PJ23.2}
\end{equation}
while the convergence in $\dot{W}^{k}X$ is a consequence of 
\begin{equation}
\Vert D^{\alpha}f_{\theta,\delta}-D^{\alpha}f_{\delta}\Vert_{X}\leq\Vert
D^{\alpha}\eta_{\delta}\Vert_{L^{1}(\mathbb{R}^{N})}\Vert
f_{\theta}-f\Vert_{X}   \label{PJ23.3}
\end{equation}
which is valid for every $\alpha\in{\mathbb{N}}_{0}^{N}$ with $|\alpha|=k$.
By virtue of \eqref{7.4}--\eqref{PJ23.3}, we have 
\begin{equation}
\Vert f_{\theta,\delta}-f_{\delta}\Vert_{\theta,p,G}^{\blacktriangleleft
}\lesssim\bigg(1+\sum_{|\alpha|=k}\Vert D^{\alpha}\eta_{\delta}\Vert _{L^{1}(%
\mathbb{R}^{N})}\bigg)\Vert f_{\theta}-f\Vert_{X}   \label{7.5}
\end{equation}
uniformly with respect to $\theta$ and $\delta$. Inserting \eqref{7.5} into %
\eqref{7.3n}, we arrive at 
\begin{equation*}
|\Vert f_{\theta,\delta}\Vert_{\theta,p,G}^{\blacktriangleleft}-L(f_{\delta})|\lesssim \bigg(%
1+\sum_{|\alpha|=k}\Vert D^{\alpha}\eta_{\delta}\Vert_{L^{1}(\mathbb{R}^{N})}%
\bigg)\Vert f_{\theta}-f\Vert_{X}+|\Vert
f_{\delta}\Vert_{\theta,p,G}^{\blacktriangleleft}-L(f_{\delta})| 
\end{equation*}
and from \eqref{PJ23.4} we derive 
\begin{equation}
\lim_{\theta\rightarrow1^{-}}\Vert f_{\theta,\delta}\Vert_{\theta
,p,G}^{\blacktriangleleft}=L(f_{\delta}),\qquad\forall\delta>0.   \label{7.6}
\end{equation}
Combining \eqref{7.1} and \eqref{7.6} yields 
\begin{equation}
\liminf_{\theta\rightarrow1^{-}}\Vert
f_{\theta}\Vert_{\theta,p,G}^{\blacktriangleleft}\geq\lim_{\theta%
\rightarrow1^{-}}\Vert f_{\theta,\delta
}\Vert_{\theta,p,G}^{\blacktriangleleft}=L(f_{\delta}).   \label{Pdjjsjas2}
\end{equation}
Assume further that $L$ (cf. \eqref{PJ23.4}) fulfils 
\begin{equation}
\liminf_{\delta\rightarrow0}L(f_{\delta})\geq L(f).   \label{Pdjjsjas}
\end{equation}
Under this assumption, one can take limits as $\delta\rightarrow0$ in %
\eqref{Pdjjsjas2} to arrive at 
\begin{equation*}
\liminf_{\theta\rightarrow1^{-}}\Vert
f_{\theta}\Vert_{\theta,p,G}^{\blacktriangleleft}\geq L(f) 
\end{equation*}
and taking now the infimum over all possible $\{f_{\theta}\}_{\theta \in (0, 1)}$ with $f_{\theta
}\rightarrow f$ in $X$ as $\theta\rightarrow1^{-}$, we conclude that 
\begin{equation}
\Gamma_{X}-\liminf_{\theta\rightarrow1^{-}}\Vert
f\Vert_{\theta,p,G}^{\blacktriangleleft}\geq L(f).   \label{14.4}
\end{equation}
Since $L(f)<\infty$ the converse estimate of \eqref{14.4} also
holds and we obtain 
\begin{equation}
\Gamma_{X}-\lim_{\theta\rightarrow1^{-}}\Vert
f\Vert_{\theta,p,G}^{\blacktriangleleft}=L(f).   \label{PJ23.5}
\end{equation}
We mention that the mollifier based technique employed above  is inspired by some ideas developed earlier in \cite[Theorem
6.9]{BM} in the setting of the classical Gagliardo seminorms \eqref{S5.1}.

The following step consists of computing $L(f)$ in \eqref{PJ23.5}. In this
regard, we will make use of \cite[Lemma 1]{KMX}, which is valid for any
function $G$, not necessarily $K$-functionals (and then it is enough to work
with characterizations of $K$-functionals in the sense of \eqref{PJ22.6}).
Namely, if\footnote{%
The stronger assumption $\Vert f\Vert_{\theta_{0},p, G}^{\blacktriangleleft}<\infty$ for some $%
\theta_{0}\in(0,1)$ was made in \cite{KMX}, but one can check that the same
proof works with the more general assumption \eqref{PJ23.7}.} 
\begin{equation}
G(t,f)\lesssim1,\qquad\forall t>0   \label{PJ23.7}
\end{equation}
and $\lim_{t \to 0} G(t, f)/t$ exists  then 
\begin{equation}
L(f)=\lim_{t\rightarrow0}\frac{G(t,f)}{t}.   \label{PJ23.6}
\end{equation}

Combining now \eqref{PJ23.5} with \eqref{PJ23.6}, we establish 
\begin{equation}
\Gamma_{X}-\lim_{\theta\rightarrow1^{-}}\Vert
f\Vert_{\theta,p,G}^{\blacktriangleleft}=\lim_{t\rightarrow0}\frac{G(t,f)}{t}
\label{PJ23.5*}
\end{equation}
for every $f\in X$ such that the right-hand side is finite (note that the required condition \eqref{PJ23.7} is
automatically satisfied because, by \eqref{PD28.1}, $G(t,f)\approx K(t,f;X,%
\dot{W}^{k}X)\leq\Vert f\Vert_{X}$). The existence and exact value of the
limit given in the right-hand side of \eqref{PJ23.5*} was already studied in 
\cite{DLTYY}. In particular, it is shown in \cite[(6.8)]{DLTYY} that  if $f\in\dot{W}^{k}X$ then 
\begin{equation}
\lim_{t\rightarrow0}\frac{G(t,f)}{t}=(kp+N)^{-1/p}\bigg(\int_{\mathbb{S}%
^{N-1}}\bigg\|\sum_{|\alpha|=k} \frac{k!}{\alpha!} \,  \omega^{\alpha}D^{\alpha}f\bigg\|%
_{X}^{p}\,d\sigma^{N-1}(\omega)\bigg)^{1/p},   \label{PJ23.8}
\end{equation}
where 
$\sigma^{N-1}$ denotes the surface measure on the $(N-1)$-dimensional sphere 
$\mathbb{S}^{N-1}$. Note that, as a consequence of \eqref{PJ23.6} and %
\eqref{PJ23.8} and an application of the Fatou's lemma, the functional $L(f)$
introduced in \eqref{PJ23.4} satisfies indeed the needed assumption %
\eqref{Pdjjsjas} for $f \in W^k X$. Therefore, in light of \eqref{PJ23.5*}, we have 
\begin{equation}
\Gamma_{X}-\lim_{\theta\rightarrow1^{-}}\Vert
f\Vert_{\theta,p,G}^{\blacktriangleleft}=(kp+N)^{-1/p}\bigg(\int_{\mathbb{S}%
^{N-1}}\bigg\|\sum_{|\alpha|=k} \frac{k!}{\alpha!} \,  \omega^{\alpha}D^{\alpha}f\bigg\|%
_{X}^{p}\,d\sigma^{N-1}(\omega)\bigg)^{1/p}   \label{PJ23.9}
\end{equation}
if $f \in W^k X$. 

Finally, in a similar spirit as done in \eqref{PD28.3} but now tracking the
value of the equivalence constants on each step, it is possible to compute
the functionals $\Vert f\Vert_{\theta,p,G}^{\blacktriangleleft}$ defined by \eqref{PJ22.7} with %
\eqref{PJ23.10}. Specifically, applying a simple change of variables
together with Fubini's theorem, we have 
\begin{align*}
(c_{\theta, p}^{-1}\Vert f\Vert_{\theta,p, G}^{\blacktriangleleft})^p & =k\int_{0}^{\infty}t^{-\theta
kp-N}\int_{|h|<t}\Vert\Delta_{h}^{k}f\Vert_{X}^{p}\,dh\,\frac{dt}{t} \\
& =k\int_{\mathbb{R}^{N}}\Vert\Delta_{h}^{k}f\Vert_{X}^{p}\int_{|h|}^{\infty
}t^{-\theta kp-N}\,\frac{dt}{t}\,dh \\
& =\frac{k}{\theta kp+N}\int_{\mathbb{R}^{N}}|h|^{-\theta kp-N}\Vert
\Delta_{h}^{k}f\Vert_{X}^{p}\,dh \\
& =\frac{k}{\theta kp+N}\Vert f\Vert_{\dot{B}_{p}^{\theta k}X}^{p}.
\end{align*}
Then \eqref{PJ23.9} reads as 
\begin{equation}
\Gamma_{X}-\lim_{\theta\rightarrow1^{-}}(1-\theta)^{1/p}\Vert f\Vert_{\dot{B}%
_{p}^{\theta k}X}=\bigg(\frac{1}{kp}\int_{\mathbb{S}^{N-1}}\bigg\|\sum
_{|\alpha|=k} \frac{k!}{\alpha!}  \, \omega^{\alpha}D^{\alpha}f\bigg\|_{X}^{p}\,d\sigma^{N-1}(\omega)%
\bigg)^{1/p}.   \label{Pdjjsjas3}
\end{equation}
Moreover, it is known that (cf. \cite[(6.3)]{DLTYY}) 
\begin{equation}
\bigg(\int_{\mathbb{S}^{N-1}}\bigg\|\sum_{|\alpha|=k} \frac{k!}{\alpha!} \, \omega^{\alpha}D^{%
\alpha }f\bigg\|_{X}^{p}\,d\sigma^{N-1}(\omega)\bigg)^{1/p}\approx\Vert
f\Vert _{\dot{W}^{k}X}.   \label{PF12.1}
\end{equation}
Specializing \eqref{Pdjjsjas3} with $X=L^{p}(\mathbb{R}^{N}),p\in(1,\infty),$ and $k=1$, we recover the classical formula \eqref{S5.1}: 
\begin{equation}
\Gamma_{L^{p}(\mathbb{R}^{N})}-\lim_{\theta\rightarrow1^{-}}(1-\theta)^{1/p}%
\Vert f\Vert_{\dot{W}^{\theta,p}(\mathbb{R}^{N})}=\bigg(\frac{C_{p,N}}{p}%
\bigg)^{1/p}\Vert\nabla f\Vert_{L^{p}(\mathbb{R}^{N})},   \label{Pdjjsjas4}
\end{equation}
where 
\begin{equation}
C_{p,N}:=\int_{\mathbb{S}^{N-1}}|e\cdot\omega|^{p}\,d\sigma^{N-1}(\omega) 
\label{PF10.1}
\end{equation}
for any fixed $e\in\mathbb{S}^{N-1}$. We mention that the above method still
applies if $p=1$ in \eqref{Pdjjsjas4}, even in the more general setting of $%
\dot{BV}(\mathbb{R}^{N})$, then sharpening the results from
Section \ref{SusectionBV}.

This methodology is robust enough to deal with the limit $\theta\rightarrow
0^{+}$. For this purpose, one only has to replace $L(f)$ in \eqref{PJ23.4} by 
\begin{equation*}
L(f):=\lim_{\theta\rightarrow0^{+}}\Vert
f\Vert_{\theta,p,G}^{\blacktriangleleft}
\end{equation*}
and apply that (cf. Corollary 6.35, the discussion on p. 35 and the formula
(6.34) in \cite{DLTYY}) 
\begin{equation*}
L(f)=\bigg(\frac{|\mathbb{S}^{N-1}|}{N}\bigg)^{1/p}\lim_{|h|\rightarrow%
\infty }\Vert\Delta_{h}^{k}f\Vert_{X}, 
\end{equation*}
leading to the following counterpart for \eqref{Pdjjsjas3}: If $f\in X$ then 
\begin{equation*}
\Gamma_{X}-\lim_{\theta\rightarrow0^{+}}\theta^{1/p}\Vert f\Vert_{\dot{B}%
_{p}^{\theta k}X}=\bigg(\frac{|\mathbb{S}^{N-1}|}{kp}\bigg)%
^{1/p}\lim_{|h|\rightarrow\infty}\Vert\Delta_{h}^{k}f\Vert_{X}. 
\end{equation*}
Note that $\lim_{|h|\rightarrow\infty}\Vert\Delta_{h}^{k}f\Vert_{X}\approx%
\Vert f\Vert_{X}$. In particular, if $X=L^{p}(\mathbb{R}^{N}),p\in\lbrack1,%
\infty)$, and $k=1$ the previous formula yields (cf. \cite[(6.34)]{DLTYY}) 
\begin{equation*}
\Gamma_{L^{p}(\mathbb{R}^{N})}-\lim_{\theta\rightarrow0^{+}}\theta^{1/p}\Vert
f\Vert_{\dot{W}^{\theta,p}(\mathbb{R}^{N})}=\bigg(\frac{2|\mathbb{S}^{N-1}|}{%
p}\bigg)^{1/p}\Vert f\Vert_{L^{p}(\mathbb{R}^{N})}
\end{equation*}
for $f \in L^p(\R^N)$, which gives a refinement of \eqref{PJ6.4}. We leave further details to the
interested reader.

The road map we have followed in the above discussion can be applied to deal
with other settings where we have existence of corresponding pointwise
limits for smooth functions (although not necessarily for general functions). We leave the discussion for another occasion since a full
development will take us outside the scope of the present paper.

\section{Extrapolation of compactness\label{SectionCompactness}}

The theorems we prove in this section combine extrapolation with the
powerful compactness theorem for real interpolation spaces of Cwikel \cite%
{C2}. In an informal manner, Cwikel's result says that if an operator is
compact on at least one of the endpoints then compactness is preserved on 
\emph{all} intermediate spaces generated by real interpolation. The precise
statement reads as follows:

\begin{theorem}
\label{Cwikelsthm} Let $(A_{0},A_{1})$ and $(B_{0},B_{1})$ be interpolation
pairs of Banach spaces and let $T:A_{0}+A_{1}\longrightarrow B_{0}+B_{1}$ be
a linear operator such that $T:A_{0}\longrightarrow B_{0}$ is bounded and $%
T:A_{1}\longrightarrow B_{1}$ is compact. Then for any $\theta \in (0, 1)$ and $p \in [1, \infty]$,
\begin{equation*}
T:(A_{0},A_{1})_{\theta,p}\longrightarrow(B_{0},B_{1})_{\theta,p}  
\end{equation*}
is compact. 
\end{theorem}

Our first result on extrapolation of compactness is the following:

\begin{theorem}
\label{Thm5 copy(1)} Let $(A_{0},A_{1})$ and $(B_{0},B_{1})$ be
interpolation pairs of Banach spaces. Let $T:A_{0}+A_{1}\longrightarrow
B_{0}+B_{1}$ be a linear operator such that $T:A_{0}\longrightarrow B_{0}$
is bounded and $T:A_{1}\longrightarrow B_{1}$ is compact. Let $%
p\in\lbrack1,\infty ],\{\theta_{n}\}_{n \in \mathbb{N}}\subset(0,1)$ with $\lim_{n\rightarrow%
\infty}\theta_{n}=1$, and let $\{u_{n}\}_{n \in \mathbb{N}}$ be a sequence such that 
\begin{equation}
\liminf_{n\rightarrow\infty}\Vert
u_{n}\Vert_{(A_{0},A_{1})_{\theta_{n},p}^{\blacktriangleleft}}<\infty. 
\label{vuelta1}
\end{equation}
Then given any $\eta\in(0,1)$ there exists $u\in B_{1}+(B_{0},B_{1})_{\eta,p}^{\blacktriangleleft}
$ such that $Tu_{n}\rightarrow u$ in $B_{1}+(B_{0},B_{1})_{\eta,p}^{\blacktriangleleft}$ as $%
n\rightarrow\infty$ (up to a subsequence).
\end{theorem}

Specialising Theorem \ref{Thm5 copy(1)} with the degenerate pair $(B_0, B_1) = (A_0, A_0)$ (and hence, by \eqref{sec2:3},  $(B_0, B_1)_{\eta, p} = A_0$) and $T = \text{id}$ (the identity operator) we establish Theorem \ref{ThmCwiGen}.   

\begin{proof}[Proof of Theorem \ref{Thm5 copy(1)}]
We will make use of the following extrapolation estimate (cf. \cite[%
Corollary 2.1]{KM}): Let $0<\eta<\theta<1$ then for all $f\in(A_{0},A_{1})_{%
\theta,p}^{\blacktriangleleft}$ we have 
\begin{equation}
\Vert f\Vert_{A_{1}+(A_{0},A_{1})_{\eta,p}^{\blacktriangleleft}}\leq M\Vert
f\Vert_{(A_{0},A_{1})_{\theta,p}^{\blacktriangleleft}},   \label{PF3.1}
\end{equation}
where $M$ is a positive constant independent of $\theta$ and $\eta$. Now, on
account that $\theta_{n}\rightarrow1$ as $n\rightarrow\infty$, it is
possible to choose $\eta\in(0,1)$ such that $\eta<\theta_{n}$ for every $%
n\in {\mathbb{N}}$. Consequently,  by \eqref{PF3.1}, 
\begin{equation*}
\Vert f\Vert_{A_{1}+(A_{0},A_{1})_{\eta,p}^{\blacktriangleleft}}\leq M\Vert
f\Vert_{(A_{0},A_{1})_{\theta_{n},p}^{\blacktriangleleft}}
\end{equation*}
for all $f\in(A_{0},A_{1})_{\theta_{n},p}^{\blacktriangleleft}$. Note that
since the constant $M$ is independent of $n$, we can apply the previous
inequality to the $u_{n}$, 
\begin{equation*}
\Vert u_{n}\Vert_{A_{1}+(A_{0},A_{1})_{\eta,p}^{\blacktriangleleft}}\leq
M\Vert
u_{n}\Vert_{(A_{0},A_{1})_{\theta_{n},p}^{\blacktriangleleft}},\qquad\forall
n\in{\mathbb{N}},
\end{equation*}
which combined with \eqref{vuelta1} yields 
\begin{equation}
\liminf_{n\rightarrow\infty}\Vert
u_{n}\Vert_{A_{1}+(A_{0},A_{1})_{\eta,p}^{\blacktriangleleft}}<\infty.   \label{P5.1}
\end{equation}

Since $T$ satisfies
the assumptions of Cwikel's theorem (cf. Theorem \ref{Cwikelsthm}), 
\begin{equation}\label{16.9.2}
T:(A_{0},A_{1})_{\eta,p}^{\blacktriangleleft} \longrightarrow(B_{0},B_{1})_{\eta,p}^{\blacktriangleleft}
\end{equation}
is a compact operator. It follows from \eqref{16.9.2} and the assumption $T: A_1 \longrightarrow B_1$ is compact  that
\begin{equation}\label{16.9.3}
T:A_{1}+(A_{0},A_{1})_{\eta,p}^{\blacktriangleleft}\longrightarrow B_{1}+(B_{0},B_{1})_{\eta,p}^{\blacktriangleleft} 
\end{equation}
is also compact. Using (\ref{P5.1}), we may select a bounded subsequence (still named $%
\{u_{n}\}_{n\in \mathbb{N}})$ in $A_{1}+(A_{0},A_{1})_{\eta,p}^{\blacktriangleleft}$.  In view of \eqref{16.9.3} we derive that $\{Tu_{n}\}_{n \in \mathbb{N}}$ is precompact in $B_{1}+(B_{0},B_{1})_{\eta,p}^{\blacktriangleleft}$ and
consequently there exists a subsequence $\{u_{n_{k}}\}_{k \in \mathbb{N}}$ such that $%
\{Tu_{n_{k}}\}_{k \in \mathbb{N}}$ converges in $B_{1}+(B_{0},B_{1})_{\eta,p}^{\blacktriangleleft}$.
\end{proof}

Next we record a  variant of Theorem \ref{Thm5 copy(1)} involving only $K$%
-functionals (rather than interpolation norms).

\begin{theorem}
\label{Thm5**} Let $(A_{0},A_{1})$ and $(B_{0},B_{1})$ be
interpolation pairs of normed spaces.  Suppose that $T: A_0 + A_1 \longrightarrow B_0 + B_1$ is a linear operator such that $%
T:A_{0}\longrightarrow B_0$ is bounded and $T:A_{1}\longrightarrow B_1$ is
compact. Let $\{t_{n}\}_{n \in \mathbb{N}}$ be an infinitesimal and let $\{u_{n}\}_{n \in \mathbb{N}}$ be a
sequence such that 
\begin{equation}
\liminf_{n\rightarrow\infty}\frac{K(t_{n},u_{n};A_{0},A_{1})}{t_{n}}%
<\infty.   \label{PF5.1}
\end{equation}
Then there exists $v\in B_1$ such that $Tu_{n}\rightarrow v$
in $B_0 + B_1$ as $n\rightarrow\infty$ (up to a subsequence).
\end{theorem}

\begin{proof}
Let 
$$
	M:= \liminf_{n \to \infty} \frac{K(t_n, u_n; A_0, A_1)}{t_n}.
$$
According to \eqref{PF5.1}, there exist subsequences of $\{t_n\}_{n \in \mathbb{N}}$ and $\{u_n\}_{n \in \mathbb{N}}$ such that 
\begin{equation}\label{PF5.1*}
	\lim_{n \to \infty}  \frac{K(t_n, u_n; A_0, A_1)}{t_n} = M. 
\end{equation}
	For every $n \in \mathbb{N}$, there exists a decomposition $u_n = w_n + v_n \in A_0 + A_1$ such that
	\begin{equation}\label{PF5.2}
		\|w_n\|_{A_0} + t_n \|v_n\|_{A_1} \leq K(t_n, u_n; A_0, A_1) + t_n.
	\end{equation}
	In particular, 
	\begin{equation}\label{PF5.2*}
		\|v_n\|_{A_1} \leq  \frac{K(t_n, u_n; A_0, A_1)}{t_n} + 1
	\end{equation}
	and then, by virtue of \eqref{PF5.1*}, there exists a subsequence of $\{v_n\}_{n \in \mathbb{N}}$ (which is relabelled to simplify the notation) uniformly bounded in $A_1$.  By the compactness of $T: A_1 \longrightarrow B_1$, there exists $v \in B_1$ and a subsequence of $\{v_n\}_{n \in \mathbb{N}}$  such that 
	\begin{equation}\label{PF13.6}
		\lim_{n \to \infty} T v_n = v \quad \text{in} \quad B_1.
		\end{equation}
		 On the other hand, from the assumption that $T: A_0 \longrightarrow B_0$,  \eqref{PF5.2} and \eqref{PF5.1*}, we also derive  
	\begin{equation}\label{20.9.1}
	\|T w_n\|_{B_0}	\lesssim \|w_n\|_{A_0}  \leq K(t_n, u_n; A_0, A_1) + t_n \lesssim t_n 
	\end{equation}
	and hence $T w_n \to 0$ in $B_0$ as $n \to \infty$. Therefore, from the estimate 
	$$
		\|T u_n - v\|_{B_0 + B_1} \leq \|T w_n \|_{B_0} + \|T v_n -v\|_{B_1}
	$$
	and \eqref{PF13.6}, \eqref{20.9.1}, we derive that   $T u_n \to v$ in $B_0+B_1$ as $n \to \infty$. 
\end{proof}

\begin{remark}\label{RemarkThm5**}
	Let us focus on    Theorem \ref{Thm5**} with $B_0 = B_1 = A_0$, i.e., $T: A_0 \longrightarrow A_0$ and $T: A_1 \longrightarrow A_0$.  Assume further that $T: A_1 \longrightarrow A_1$ boundedly and the norm on $A_1$ is lower semicontinuous with respect to the convergence in $A_0$. Then Theorem \ref{Thm5**} can be reinforced to ensure that the limit $v \in A_0 \cap A_1$. Indeed, by taking limits in \eqref{PF5.2*}, we know that $\liminf_{n \to \infty} \|v_n\|_{A_1} < \infty$ and hence $\liminf_{n \to \infty} \|T v_n\|_{A_1} < \infty$. Since $T v_n \to v$ in $A_0$ (cf. \eqref{PF13.6}), we get $\|v\|_{A_1} \leq \liminf_{n \to \infty} \|T v_n\|_{A_1} < \infty$.  
\end{remark}

\subsection{Applications}\label{Section7.1}

Next we write down some immediate applications of Theorems \ref{Thm5 copy(1)}  and %
\ref{Thm5**} for some special choices of interpolation pairs $(A_{0}, A_{1})$. 

\begin{corollary}
\label{CorCmp1} Let $X$ be a r.i. space with absolutely continuous norm,
defined on a cube $Q\subset\mathbb{R}^{N}$. Let $q\in\lbrack1,\infty]$ and $%
k\in{\mathbb{N}}$. Suppose that $\{f_{n}\}_{n \in \N}$ is a bounded sequence in $X$
such that\footnote{%
The Besov spaces $\dot{B}_{q}^{s k}X$ on $Q$ are defined in a similar fashion
as their counterparts on $\mathbb{R}^{N}$ (cf. \eqref{DefBesovD}): $\Vert
f\Vert_{\dot{B}_{q}^{sk}X}=(\int_{\mathbb{R}^{N}}|h|^{-skq-N}\Vert%
\Delta_{h}^{k}f\Vert_{X(Q_{kh})}^{q}\,dh)^{1/q}$, where $Q_{kh}=\{x\in
Q:x+kh\in Q\}$.} 
\begin{equation}
\sup_{n\in{\mathbb{N}}}(1-s_{n})^{1/q}\Vert f_{n}\Vert_{\dot{B}%
_{q}^{s_{n}k}X}<\infty   \label{PF13.3}
\end{equation}
for some sequence $s_{n}\rightarrow1^{-}$. Then there exists $f\in X$ such
that $f_{n}\rightarrow f$ in $X$ as $n\rightarrow\infty$ (up to a
subsequence).
\end{corollary}

\begin{remark}
This result extends the classical BBM compactness criteria for the fractional
Sobolev spaces $\dot{W}^{s,p}$ (cf. \cite[Theorem 4]{BBM2})$,$ that
corresponds to the choice $X=L^{p},p\in\lbrack1,\infty),q=p$ and $k=1$ (cf. %
\eqref{CoincGFS}).
\end{remark}

\begin{proof}[Proof of Corollary \protect\ref{CorCmp1}]
We apply Theorem  \ref{Thm5 copy(1)}  to the pairs $(A_{0},A_{1})=(X,W^{k}X), (B_0, B_1) = (X, X)$ and $T =\text{id}$. Since $X
$ has absolutely continuous norm, the embedding $W^{k}X\hookrightarrow X$ is
compact (cf. \cite{Pu}). It is easy to see that, for any interpolation pair $%
(A_{0},A_{1})$, $t\in(0,1)$ and $f\in A_{0}$, we have 
\begin{equation}
K(t,f;A_{0},A_{0}\cap A_{1})\approx t\Vert
f\Vert_{A_{0}}+K(t,f;A_{0},A_{1}).   \label{PF13.1}
\end{equation}
Under the assumption $A_{1}\hookrightarrow A_{0}$, we have 
\begin{equation}
\Vert f\Vert_{(A_{0},A_{1})_{\theta,q}}\approx\bigg(\int_{0}^{1}[t^{-\theta
}K(t,f;A_{0},A_{1})]^{q}\,\frac{dt}{t}\bigg)^{1/q}   \label{PF13.2}
\end{equation}
with equivalence constants that are independent of $\theta\rightarrow1^{-}$,
cf. \cite[Lemma 3]{M-Lim}. Using \eqref{PF13.1} and \eqref{PF13.2}, we
compute 
\begin{align*}
\Vert f\Vert_{(X,W^{k}X)_{\theta,q}} & \approx\bigg(\int_{0}^{1}[t^{-\theta
}K(t,f;X,W^{k}X)]^{q}\,\frac{dt}{t}\bigg)^{1/q} \\
& \approx\bigg(\int_{0}^{1}[t^{-\theta}K(t,f;X,\dot{W}^{k}X)]^{q}\,\frac {dt%
}{t}\bigg)^{1/q}+\bigg(\int_{0}^{1}t^{(1-\theta)q}\,\frac{dt}{t}\bigg)%
^{1/q}\Vert f\Vert_{X} \\
& \lesssim\Vert f\Vert_{(X,\dot{W}^{k}X)_{\theta,q}}+(1-\theta)^{-1/q}\Vert
f\Vert_{X},
\end{align*}
where the constants on each step are independent of $\theta\rightarrow1^{-}$%
. In particular, with constants independent of $n\in \N,$ 
\begin{equation*}
\Vert f_{n}\Vert_{(X,W^{k}X)_{s_{n},q}}\lesssim\Vert f_{n}\Vert_{(X,\dot {W}%
^{k}X)_{s_{n},q}}+(1-s_{n})^{-1/q}\Vert f_{n}\Vert_{X}. 
\end{equation*}
Using \eqref{PD28.3} adapted to $Q$, the previous estimate can be expressed
as 
\begin{equation*}
\Vert f_{n}\Vert_{(X,W^{k}X)_{s_{n},q}}\lesssim\Vert f_{n}\Vert_{\dot{B}%
_{q}^{s_{n}k}X}+(1-s_{n})^{-1/q}\Vert f_{n}\Vert_{X}
\end{equation*}
and recalling the normalization constants $%
c_{s_{n},q}=(s_{n}(1-s_{n})q)^{1/q}$:  
\begin{equation*}
\Vert
f_{n}\Vert_{(X,W^{k}X)_{s_{n},q}^{\blacktriangleleft}}=c_{s_{n},q}\Vert
f_{n}\Vert_{(X,W^{k}X)_{s_{n},q}}\lesssim(1-s_{n})^{1/q}\Vert f_{n}\Vert _{%
\dot{B}_{q}^{s_{n}k}X}+\Vert f_{n}\Vert_{X}. 
\end{equation*}
In view of \eqref{PF13.3} and the boundedness of $\{f_{n}\}_{n\in \N}$ in $X$%
, condition \eqref{vuelta1} is satisfied and we can apply Theorem \ref{Thm5 copy(1)} 
to conclude.
\end{proof}

The following compactness criterion in terms of the BMO-type functionals $%
\kappa_{\varepsilon}$ (cf. \eqref{BMO-type}) was recently obtained in \cite[%
Theorem 2]{ARBDN}. We give a simpler alternative proof via our extrapolation
method.

\begin{corollary}
\label{CorCompactAR} Assume that $\{\varepsilon_{n}\}_{n \in \N}$ is infinitesimal and
let $\{f_{n}\}_{n\in \N}\subset L_{\emph{loc}}^{1}(\mathbb{R}^{N})$ be such
that 
\begin{equation}
\liminf_{n\rightarrow\infty}\kappa_{\varepsilon_{n}}(f_{n})<\infty. 
\label{PF13.5}
\end{equation}
Then there exist a sequence $\{c_{n}\}_{n\in \N}\subset\mathbb{R}$ and $f\in
L_{\emph{loc}}^{1}(\mathbb{R}^{N})\cap\dot{\emph{BV}}(\mathbb{R}^{N})$ such
that $(f_{n}-c_{n})\rightarrow f$ in $L_{\emph{loc}}^{1}(\mathbb{R}^{N})$
(up to a subsequence).
\end{corollary}

\begin{proof}
For simplicity, we shall restrict our attention to function spaces defined
over a cube $Q$. Consider the pairs, $(A_{0},A_{1})=(L^{1}(Q),\dot{BV}%
(Q)), (B_0, B_1) = (L^1(Q), L^1(Q))$ and the linear operator $Tf=f-f_{Q}$. Clearly, $T$ maps $%
L^{1}(Q)$ into itself. On the other hand, by the classical
Rellich--Kondrachov theorem, $T:\dot{BV}(Q)\longrightarrow L^{1}(Q)$ is
compact. To be able to apply Theorem \ref{Thm5**} \ we need to verify that %
\eqref{PF5.1} holds: 
\begin{equation}
\liminf_{n\rightarrow\infty}\frac{K(\varepsilon_{n},f_{n};L^{1}(Q),\dot {BV}%
(Q))}{\varepsilon_{n}}<\infty.   \label{PF13.4}
\end{equation}
To see this, we can make use of the trivial embedding $\dot{W}%
^{1,1}(Q)\hookrightarrow\dot{BV}(Q)$ and \eqref{K-Funct-DS}
(more precisely, its local counterpart) to get 
\begin{equation*}
\frac{K(\varepsilon_{n},f_{n};L^{1}(Q),\dot{BV}(Q))}{%
\varepsilon_{n}}\lesssim\frac{K(\varepsilon_{n},f_{n};L^{1}(Q),\dot{W}%
^{1,1}(Q))}{\varepsilon_{n}}\approx\kappa_{\varepsilon_{n}}(f_{n}) 
\end{equation*}
and hence \eqref{PF13.4} follows from \eqref{PF13.5}. By virtue of Theorem %
\ref{Thm5**}, there exists $f\in L^{1}(Q)$ such that $%
Tf_{n}=f_{n}-(f_{n})_{Q}\rightarrow f$ in $L^{1}(Q)$ as $n\rightarrow\infty$%
. In addition, since $\Vert Tf\Vert_{\dot{BV}\text{(}Q)}=\Vert f\Vert_{\dot{%
BV}(Q)}$ and the $\dot{BV}(Q)$ norm is lower semicontinuous with respect to
the $L^{1}(Q)$-convergence, it turns out that $f\in\dot{BV}(Q)$, cf. Remark \ref{RemarkThm5**}. This proves the desired result
with scalars $c_{n}=(f_{n})_{Q}$.
\end{proof}

In a similar fashion as Corollary \ref{CorCompactAR}, we can give an
elementary proof of the compactness result in terms of heat kernels from 
\cite[Theorem 1.6]{GS}.

\begin{corollary}
Let $p\in[1,\infty)$. Assume that $\{t_{n}\}_{n \in \N}$ is infinitesimal and let $%
\{f_{n}\}_{n \in \N}$ be a bounded sequence in $L^{p}(\mathbb{R}^{N})$ such that 
\begin{equation}
\liminf_{n\rightarrow\infty}\frac{1}{t_{n}^{p/2}}\int_{\mathbb{R}%
^{N}}H_{t_{n}}(|f_{n}-f_{n}(x)|^{p})(x)\,dx<\infty.   \label{PF6.2}
\end{equation}
Then (up to a subsequence) $f_{n}\rightarrow f$ in $L_{\emph{loc}}^{p}(%
\mathbb{R}^{N})$ for some $f\in L_{\emph{loc}}^{p}(\mathbb{R}^{N})$ (in addition, $f \in W^{1, p}_{\emph{loc}}(\R^N)$ if $p > 1$ and $f \in BV_{\emph{loc}}(\R^N)$ if $p=1$).
\end{corollary}

\begin{proof}
To avoid unnecessary technicalities, we shall restrict our attention to function spaces defined on a fixed cube $Q \subset \R^N$. It follows from \eqref{20.9.3} that 
$$
	\frac{1}{t_n^{\frac{p+N}{2}}}\int_{|\xi|<\sqrt{t_n}}\Vert\Delta_{\xi
}f_n\Vert_{L^{p}(Q_\xi)}^{p}\,d\xi \lesssim \frac{1}{t_{n}^{p/2}}\int_{\mathbb{R%
}^{N}}H_{t_{n}}(|f_{n}-f_{n}(x)|^{p})(x)\,dx, 
$$
which can be rewritten in terms of $K$-functionals (cf. \eqref{PD28.1}) as follows
\begin{equation*}
\bigg(\frac{K(\sqrt{t_{n}},f_{n};L^{p}(Q),\dot{W}^{1,p}(Q))}{\sqrt{t_{n}}}\bigg)^{p}\lesssim\frac{1}{t_{n}^{p/2}}\int_{\mathbb{R%
}^{N}}H_{t_{n}}(|f_{n}-f_{n}(x)|^{p})(x)\,dx 
\end{equation*}
for every $n\in{\mathbb{N}}$ and, in particular (cf. \eqref{PF13.1} and %
\eqref{PF6.2}), 
\begin{align*}
\liminf_{n\rightarrow\infty}\frac{K(\sqrt{t_{n}},f_{n};L^{p}(Q),W^{1,p}(Q))}{\sqrt{t_{n}}} & \\
& \hspace{-4.5cm}\lesssim\sup_{n\in{\mathbb{N}}}\Vert f_{n}\Vert _{L^{p}(Q)}+\liminf_{n\rightarrow\infty}\frac{K(\sqrt{t_{n}}%
,f_{n};L^{p}(Q),\dot{W}^{1,p}(Q))}{\sqrt{t_{n}}} \\
& \hspace{-4.5cm} \lesssim \sup_{n\in{\mathbb{N}}}\Vert f_{n}\Vert _{L^{p}(\R^N)} + \liminf_{n \to \infty}  \frac{1}{t_{n}^{1/2}} \bigg(\int_{\mathbb{R%
}^{N}}H_{t_{n}}(|f_{n}-f_{n}(x)|^{p})(x)\,dx \bigg)^{1/p} < \infty. 
\end{align*}
Then the desired convergence assertion is an immediate consequence of
Theorem \ref{Thm5**} with $(A_{0},A_{1})=(L^{p}(Q),W^{1,p}(Q)), (B_0, B_1) = (L^p(Q), L^p(Q))$ and $T= \text{id}$ on $Q$. Moreover, one can apply Remark \ref{RemarkThm5**} to conclude that $f \in W^{1, p}(Q)$ if $p > 1$. The case $p=1$ needs some further justifications: It is possible to replace the role of $W^{1, 1}(Q)$ by $BV(Q)$ in the above computations since $$G_0 (L^1(Q), W^{1, 1}(Q)) = L^1(Q) \qquad \text{and} \qquad  G_1(L^1(Q), W^{1, 1}(Q)) = BV(Q)$$ and then (cf. \eqref{sec2:5})
\begin{equation}\label{29.9.3}
K(t, f; L^1(Q), BV(Q)) \approx K(t, f; L^1(Q), W^{1, 1}(Q)).
\end{equation}
\end{proof}

Let $(\mathcal{X}, d, \mu)$ be a compact metric measure space satisfying the
assumptions from Section \ref{GConvMetric}. The following result, which has
been recently derived in \cite[Theorem 5.3]{ARB}, is also a consequence of
the extrapolation approach.

\begin{corollary}
Let $p\in\lbrack1,\infty)$. Assume that $\{t_{n}\}_{n \in \N}$ is infinitesimal and $%
\{f_{n}\}_{n \in \N}\subset L^{p}(\mathcal{X})$ such that 
\begin{equation}
\liminf_{n\rightarrow\infty}\bigg[\bigg(\frac{E_{p,t_{n}}(f_{n})}{t_{n}}%
\bigg)^{p}+\Vert f_{n}\Vert_{L^{p}(\mathcal{X})}^{p}\bigg]<\infty. 
\label{PF8.1}
\end{equation}
Then $\{f_{n}\}_{n \in \N}$ contains a subsequence that converges strongly in $L^{p}(%
\mathcal{X})$.
\end{corollary}

\begin{proof}
We wish to apply Theorem \ref{Thm5**} for the choice $(A_{0}, A_{1}) =
(L^{p}(\mathcal{X}),L^{p}(\mathcal{X}) \cap KS^{1, p}(\mathcal{X})), (B_0, B_1) =(
L^{p}(\mathcal{X}), L^{p}(\mathcal{X}))$ and $T = \text{id}$ (the compactness of the embedding 
$L^{p}(\mathcal{X}) \cap KS^{1, p}(\mathcal{X}) \hookrightarrow L^{p}(%
\mathcal{X})$ is known, cf. \cite[(68)]{ARB}). From \eqref{PF13.1} and %
\eqref{PJ4.4}, 
\begin{equation*}
\frac{K(t, f; L^{p}(\mathcal{X}), L^{p}(\mathcal{X}) \cap KS^{1, p}(\mathcal{%
X}))}{t} \approx\|f\|_{L^{p}(\mathcal{X})} + \frac{E_{p, t}(f)}{t}
\end{equation*}
and hence the validity of \eqref{PF8.1} implies \eqref{PF5.1}.
\end{proof}

\section{Variational problems via $\Gamma$-extrapolation}\label{Section8}

\subsection{Proof of Theorem \ref{ThmVarIntro}}
\emph{Step 1.}  We start by showing that 
\begin{equation}\label{20.9.5}
\Gamma_{A_0}- \lim_{\theta \to 1^{-}} E_\theta = E_1.
\end{equation}
Let $u\in A_{0}$ and
let $\{\theta_{n}\}_{n\in \N}$ be a sequence in $(0,1)$ converging to $1$. We wish  to find a sequence $\{u_{n}\}_{n\in \N}$ in $A_{0}$ and
converging to $u$ such that%
\begin{equation}
\limsup_{n\rightarrow\infty}E_{\theta_{n}}(u_{n})\leq E_1(u)   \label{mmvar0}
\end{equation}
and so
\begin{equation}\label{20.9.6}
	\Gamma_{A_0}-\limsup_{n \to \infty} E_{\theta_n} \leq E_1. 
\end{equation}
Recall that (cf. (\ref{sec2:7})) 
\begin{equation}
\lim_{n\rightarrow\infty}\Vert
u\Vert_{(A_{0},A_{1})_{\theta_{n},p}^{\blacktriangleleft}}=\left\Vert
u\right\Vert _{G_{1}(A_{0},A_{1})}, \qquad u \in A_0.   \label{mmvar1}
\end{equation}
We now show that the constant sequence $u_{n}:=u,n \in \N,$ does the job$.$ We
may suppose that the right-hand side of (\ref{mmvar0}) is finite since
otherwise the claim \eqref{mmvar0} becomes trivial, therefore, since $f\in B,$ we must have 
$u\in B.$ Then, by \eqref{mmvar1},  
\begin{align*}
\limsup_{n\rightarrow\infty}E_{\theta_{n}}(u) & =\limsup
_{n\rightarrow\infty}(\left\Vert f-u\right\Vert _{B}^{q}+\lambda\Vert
u\Vert_{(A_{0},A_{1})_{\theta_{n},p}^{\blacktriangleleft}}) \\
& = \left\Vert f-u\right\Vert _{B}^{q}+\lambda\limsup_{n \to \infty} \Vert u\Vert
_{(A_{0},A_{1})_{\theta_{n},p}^{\blacktriangleleft}} \\
& = \left\Vert f-u\right\Vert _{B}^{q}+ \lambda \|u\|_{G_1(A_0, A_1)}  =E_1(u).
\end{align*}
This proves \eqref{mmvar0}. 

Let $u\in A_{0}$ and let $\{u_{n}\}_{n\in \N}$ 
be a sequence converging to $u$ in $A_{0},$ we want to show that%
\begin{equation}\label{20.9.4}
E_1(u)\leq\liminf_{n\to \infty}E_{\theta_{n}}(u_{n}). 
\end{equation}
Without loss of generality, we may assume that $\liminf_{n\to \infty}E_{\theta_{n}}(u_{n})<%
\infty.$ Then there exist subsequences  of $\{\theta_n\}_{n \in \N}$ and $\{u_{n}\}_{n \in \N}$ (which are still denoted by $\{\theta_n\}_{n \in \N}$ and $\{u_{n}\}_{n \in \N}$) such that  
\begin{equation*}
\liminf_{n \to \infty} E_{\theta_{n}}(u_{n})=\lim_{n \to \infty} E_{\theta_{n}}(u_{n})<\infty, 
\end{equation*}
and the sequence $\{f-u_{n}\}_{n \in \N}$ is bounded in $B$.  Now, since $u_{n}\rightarrow u$ in $A_{0},$ we can
extract a further subsequence such that $u_{n}\rightarrow u$ a.e. (cf. \cite[Theorem 1.4, p. 3]{BS}) and $
f-u_{n}\rightarrow f-u$ a.e. Since $B$ satisfies the Fatou
property (cf. \cite[Lemma 1.5, pp. 4-5]{BS}) we have%
\begin{equation}
\left\Vert f-u\right\Vert
_{B}^{q}\leq\liminf_{n\to\infty}\left\Vert f-u_{n}\right\Vert
_{B}^{q}.   \label{mmvar8}
\end{equation}
On account of the $\Gamma_{A_{0}}-\lim_{n\rightarrow\infty}\Vert
u\Vert_{(A_{0},A_{1})_{\theta_{n},p}^{\blacktriangleleft}}=\left\Vert
u\right\Vert _{G_{1}(A_{0},A_{1})}$ (cf. Theorem \ref{Thm1}) we have%
\begin{equation}
\left\Vert u\right\Vert
_{G_{1}(A_{0},A_{1})}\leq\liminf_{n\to\infty }\Vert
u_{n}\Vert_{(A_{0},A_{1})_{\theta_{n},p}^{\blacktriangleleft}}. 
\label{mmvar9}
\end{equation}
Combining (\ref{mmvar8}) and (\ref{mmvar9}),  we obtain%
\begin{align*}
E_1(u) & =\left\Vert f-u\right\Vert _{B}^{q}+\lambda\left\Vert u\right\Vert
_{G_{1}(A_{0},A_{1})} \\
& \leq\liminf_{n\to\infty}\left\Vert f-u_{n}\right\Vert
_{B}^{q}+\lambda\liminf_{n\to\infty}\Vert
u_{n}\Vert_{(A_{0},A_{1})_{\theta_{n},p}^{\blacktriangleleft}} \\
& \leq\liminf_{n\to\infty}(\left\Vert f-u_{n}\right\Vert
_{B}^{q}+\lambda\Vert
u_{n}\Vert_{(A_{0},A_{1})_{\theta_{n},p}^{\blacktriangleleft }}) \\
& =\liminf_{n\rightarrow\infty}E_{\theta_{n}}(u_{n})
\end{align*}
and then \eqref{20.9.4} holds. Taking now the infimum in  \eqref{20.9.4} over all possible sequences $\{u_n\}_{n \in \N}$ with $u_n \to u$ in $A_0$, we derive that
\begin{equation}\label{20.9.7}
	E_1 \leq \Gamma_{A_0}-\liminf_{n \to \infty} E_{\theta_n}. 
\end{equation}

The desired result \eqref{20.9.5} follows as a combination of \eqref{20.9.6} and  \eqref{20.9.7}.

\emph{Step 2.} We are going to show existence of minimizers of $E_1$.  Let $\{u_{n}\}_{n\in \N}$ be a sequence
of near-minimizers of $E_1,$ or in other words, 
\begin{equation}
E_1(u_{n})\leq\inf_{u}E_1(u)+\varepsilon_{n},   \label{tambien}
\end{equation}\label{23.9.2}
where $\varepsilon_{n}>0$ with $\lim_{n \to \infty} \varepsilon_{n}=0.$ In particular, $%
\{u_{n}\}_{n\in \N}$ is uniformly  bounded in $B \cap G_{1}(A_{0},A_{1})$, that is, 
\begin{equation}\label{5.10.3}
	\|u_n\|_B + \frac{K(t, u_n; A_0, A_1)}{t} \lesssim 1, \qquad \forall t > 0, \qquad \forall n \in \N.
\end{equation}
Assume that $\{\theta_n\} \subset (0, 1)$ with $\theta_n \to 1$. Applying \eqref{PF13.2} to the ordered pair $(A_0, A_0 \cap A_1)$  and \eqref{PF13.1}, we get
\begin{align}
	 \|u_n\|_{(A_0, A_0 \cap A_1)_{\theta_n, p}^\blacktriangleleft} & \approx c_{\theta_n, p} \bigg(\int_{0}^{1}[t^{-\theta_n
}K(t,u_n;A_{0},A_0 \cap A_{1})]^{p}\,\frac{dt}{t}\bigg)^{1/p} \nonumber \\
& \approx  c_{\theta_n, p} \bigg(\int_{0}^{1}[t^{-\theta_n
}K(t,u_n;A_{0}, A_{1})]^{p}\,\frac{dt}{t}\bigg)^{1/p}  + c_{\theta_n, p} \bigg(\int_0^1 t^{(1-\theta_n) p} \, \frac{dt}{t} \bigg)^{1/p} \|u_n\|_{A_0} \nonumber \\
& \approx c_{\theta_n, p} \bigg(\int_{0}^{1}[t^{-\theta_n
}K(t,u_n;A_{0}, A_{1})]^{p}\,\frac{dt}{t}\bigg)^{1/p}  + \|u_n\|_{A_0} \nonumber \\
&  \lesssim c_{\theta_n, p} \bigg(\int_{0}^{1}[t^{-\theta_n
}K(t,u_n;A_{0}, A_{1})]^{p}\,\frac{dt}{t}\bigg)^{1/p}  + \|u_n\|_{B}, \label{23.9.3}
\end{align}
where we have used the fact that $B \hookrightarrow A_0$ in the last step. It follows from \eqref{5.10.3} that
\begin{align*}
	 \|u_n\|_{(A_0, A_0 \cap A_1)_{\theta_n, p}^\blacktriangleleft} & \lesssim c_{\theta_n, p} \bigg(\int_0^1 t^{(1-\theta_n) p} \, \frac{dt}{t} \bigg)^{1/p} + 1 \approx 1
\end{align*}
with all underlying equivalence constants independent of $n$. Hence 
$$
	\liminf_{n \to \infty} \|u_n\|_{(A_0, A_0 \cap A_1)_{\theta_n, p}^\blacktriangleleft} < \infty
$$
and we are in a position to apply Theorem \ref{ThmCwiGen} to the pair $(A_0, A_0 \cap A_1)$ (cf. \eqref{23.9.1}). 
Therefore we can find $\bar{u}\in A_{0}$ and a subsequence, still named $%
\{u_{n}\}_{n \in \N},$ such that $u_{n}\rightarrow\bar{u}$ in $A_{0}.$ Passing through
a subsequence if necessary, we can assume that $u_{n}\rightarrow\bar{u}$ a.e.. 
Moreover, (\ref{tambien}) implies that the sequence $\{f-u_{n}\}_{n\in \N}$
is bounded in $B$ and, by Fatou's property,%
\begin{equation}
\left\Vert f-\bar{u}\right\Vert _{B}^{q}\leq\liminf_{n\to\infty
}\left\Vert f-u_{n}\right\Vert _{B}^{q}.   \label{tambien1}
\end{equation}
Since for every $t>0,$ the map $f\in A_{0}\mapsto K(t,f;A_{0},A_{1})$ is
continuous\footnote{Use that $|K(t, f; A_0, A_1)-K(t, g; A_0, A_1)| \leq K(t, f-g; A_0, A_1) \leq \|f-g\|_{A_0}$.} it  readily follows that the norm on $G_{1}(A_{0},A_{1})$ is lower semicontinuous with
respect to convergence in the norm $A_{0}.$ As a consequence 
\begin{equation}
\left\Vert \bar{u}\right\Vert
_{G_{1}(A_{0},A_{1})}\leq\liminf_{n\to \infty}\left\Vert
u_{n}\right\Vert _{G_{1}(A_{0},A_{1})}.   \label{tambien2}
\end{equation}
Adding (\ref{tambien1}) and (\ref{tambien2}),  we obtain (cf. also \eqref{tambien})
\begin{align*}
E_1(\bar{u}) & \leq\liminf_{n\rightarrow\infty}E_1(u_{n})  \leq\inf_{u}E_1(u),    
\end{align*}
establishing $\bar{u}$ as a minimizer of $E_1$.


\emph{Step 3.} Next we are going to analyze  the convergence of near-minimizers. Let $\{\bar{u}_{\theta}\}_{\theta \in (0, 1)}$ be a family
of near-minimizers of $E_{\theta}$, say 
\begin{equation*}
E_{\theta}(\bar{u}_{\theta})\leq\inf_{u}E_{\theta}(u)  + \varepsilon_\theta
\end{equation*}
with $\lim_{\theta \to 1^{-}} \varepsilon_\theta =0$. Noting that $\inf_u E_\theta(u) \leq  \|f\|_B^q$, 
it follows that the family $\{E_{\theta}(\bar{u}_{\theta})\}$ is uniformly bounded provided that $\theta \in (\theta_0, 1)$ for some $\theta_0 \in (0, 1)$. Without loss of generality, we may assume that $\theta_0 =0$.  In
particular, the sets 
\begin{equation}\label{22.9.1}
\{\|\bar{u}_\theta\|_B\}_{\theta \in (0, 1)} \qquad \text{and} \qquad     \{\left\Vert \bar{u}_{\theta
}\right\Vert _{(A_{0},A_{1})_{\theta,p}^{\blacktriangleleft}}\}_{\theta
\in(0,1)} \qquad  \text{are uniformly bounded.}
\end{equation}
Relying on \eqref{23.9.3}, we can estimate
$$
	\|\bar{u}_{\theta}\|_{(A_0, A_0 \cap A_1)_{\theta, p}^{\blacktriangleleft}} \lesssim \|\bar{u}_\theta\|_{(A_0, A_1)_{\theta, p}^{\blacktriangleleft}} + \|\bar{u}_{\theta}\|_B
$$
provided that $\theta \to 1^{-}$. Then, by \eqref{22.9.1}, 
$$
	\liminf_{\theta \to 1^{-}} \|\bar{u}_{\theta}\|_{(A_0, A_0 \cap A_1)_{\theta, p}^{\blacktriangleleft}}  < \infty. 
$$
 Therefore, by Theorem \ref{ThmCwiGen} applied to $(A_0, A_0 \cap A_1)$, there exists a
subsequence $\{\bar{u}_{\theta_{n}}\}_{n\in \N}$ and $%
\bar{u}\in A_{0}$ such that $\bar{u}_{\theta_{n}}\rightarrow\bar{u}$ in $A_{0},$
with $\theta_{n}\rightarrow1^{-}$.  On account of Step 1,  $\Gamma_{A_{0}}-\lim_{n \to \infty}
E_{\theta_{n}}=E_1$ and then  it follows by the fundamental theorem of $\Gamma$-convergence (see e.g. \cite[Theorem 1.21]{Bra} and \cite[Theorem 7.8]{DALM}) that $\bar{u}$ is a minimizer of $E_1$ and%
\begin{equation}
\lim_{n\rightarrow\infty}E_{\theta_{n}}(\bar{u}_{\theta_{n}})=E_1(\bar{u}),
\label{tambien8}
\end{equation}
i.e. \eqref{20.9.9} holds. 

\emph{Step 4.} It remains to check that the convergence assertion obtained in the previous step $\bar{u}_{\theta_n} \to \bar{u}$ in $A_0$ can be updated to say that  $\bar{u}_{\theta_n} \to \bar{u}$ in $B$ under the additional assumption that $B$ is uniformly convex.   In
particular, (\ref{tambien8}) implies that%
\begin{equation}\label{21.9.1}
\left\Vert f-\bar{u}\right\Vert _{B}^{q}+\lambda\left\Vert \bar{u}%
\right\Vert _{G_{1}(A_{0},A_{1})}\geq\limsup_{n\rightarrow\infty}\left\Vert
f-\bar{u}_{\theta_{n}}\right\Vert _{B}^{q}+\lambda\liminf_{n\rightarrow\infty
}\left\Vert \bar{u}_{\theta_{n}}\right\Vert
_{(A_{0},A_{1})_{\theta_{n},p}^{\blacktriangleleft}}.
\end{equation}
Furthermore, from Theorem \ref{Thm1}, we know 
$$
	\|\bar{u}\|_{G_1(A_0, A_1)} = \Gamma_{A_0}-\lim_{\theta \to 1^{-}} \|\bar{u}\|_{(A_0, A_1)_{\theta, p}^{\blacktriangleleft}},
$$ 
and, since $\bar{u}_{\theta_n} \to \bar{u}$ in $A_0$,  we infer that (cf. also \eqref{22.9.1})
$$
	\|\bar{u}\|_{G_1(A_0, A_1)} \leq \liminf_{n \to \infty}  \|\bar{u}_{\theta_n}\|_{(A_0, A_1)_{\theta_n, p}^\blacktriangleleft} < \infty.
$$
Inserting this into \eqref{21.9.1}, we arrive at
\begin{equation}\label{21.9.2}
	\left\Vert f-\bar{u}\right\Vert _{B}^{q} \geq \limsup_{n\rightarrow\infty}\left\Vert
f-\bar{u}_{\theta_{n}}\right\Vert _{B}^{q}
\end{equation}
and, on the other hand, by the Fatou property on $B$ (noting that, possibly passing to a subsequence,  $\bar{u}_{\theta_n} \to \bar{u}$ a.e. since $\bar{u}_{\theta_n} \to \bar{u}$ in $A_0$):
\begin{equation}\label{21.9.3}
\left\Vert f-\bar{u}\right\Vert _{B}^{q} \leq \liminf_{n \to \infty} \|f-\bar{u}_{\theta_n}\|_B^q. 
\end{equation}
Putting together \eqref{21.9.2} and \eqref{21.9.3}, we obtain that
\begin{equation}
\lim_{n \to \infty} \left\Vert f-\bar{u}_{\theta_{n}}\right\Vert _{B}=\left\Vert f-\bar {u}%
\right\Vert _{B}.   \label{tambien9}
\end{equation}
According to the Milman--Pettis theorem \cite[Theorem 3.31, p. 77]{B}, since $B$ is uniformly convex it is reflexive. Hence from the boundedness  of  $%
\{f-\bar{u}_{\theta _{n}}\}_{n\in \N}$ in $B$, it follows that $%
\{f-\bar{u}_{\theta_{n}}\}_{n\in \N}$ is weakly precompact and consequently, going
through a subsequence, there exists $\beta\in B$ such that $%
f-\bar{u}_{\theta_{n}}\rightharpoonup\beta.$ Since $B \hookrightarrow A_0$, we  also derive that $
f-\bar{u}_{\theta_{n}}\rightharpoonup\beta$ in $A_0$. On the other hand,  $f-\bar{u}_{\theta_n} \to f-\bar{u}$ in $A_0$. Uniqueness of weak-limit yields that  $\beta = f-\bar{u}$ and then  (\ref{tambien9}) is now supplemented with $f-\bar{u}_{\theta_{n}}%
\rightharpoonup f-\bar{u}$ in $B$. Since $B$ is uniformly convex it
satisfies the Radon--Riesz property (cf. \cite[Proposition 3.32, p. 78]{B}) which yields%
\begin{equation*}
f-\bar{u}_{\theta_{n}}\rightarrow f-\bar{u} \qquad \text{in} \qquad B. 
\end{equation*}
Consequently%
\begin{equation*}
\lim_{n \to \infty} \left\Vert \bar{u}-\bar{u}_{\theta_{n}}\right\Vert _{B}=\lim_{n \to \infty} \left\Vert
(f-\bar{u}_{\theta_{n}})-(f-\bar{u})\right\Vert _{B}=0 
\end{equation*}
as we wished to show. \qed

\begin{remark}\label{Remark15}
An alternative argument to Step 4 leading to convergence  in certain Banach function spaces that are
not uniformly convex including $L^{1}$, and even Orlicz spaces that satisfy
the $\Delta_{2}$ condition, is to dispense with weak convergence and replace
it by a.e. convergence through the use of the Brezis--Lieb lemma \cite%
{BrLi}. In fact,  the Brezis--Lieb lemma simply states that if $%
f_{n}\rightarrow f$ a.e. and $\{f_n\}_{n \in \N}$ is uniformly bounded in $L^p, p >0,$ then 
\begin{equation*}
\lim_{n\rightarrow\infty}(\left\Vert f_{n}\right\Vert _{L^{p}}^p-\left\Vert
f-f_{n}\right\Vert _{L^{p}}^p)=\left\Vert f\right\Vert _{L^{p}}^p. 
\end{equation*}
In particular, a.e. convergence and $\lim_{n \to \infty} \left\Vert f_{n}\right\Vert
_{L^{1}}=\left\Vert f\right\Vert _{L^{1}}$ yield norm convergence in $L^1$.

In the general setting that $B$ is a Banach function space, the above proof (cf. Step 3) reveals that (up to a subfamily) $\bar{u}_\theta \to \bar{u}$ as $\theta \to 1^{-}$ in $A_0$. 
\end{remark}

\subsection{Applications of Theorem \ref{ThmVarIntro}}

In this subsection, we assume that all function spaces are defined over a
fixed cube $Q$ with sides parallel to the axes of coordinates in $\mathbb{R}%
^{N}$.

\begin{theorem}
\label{ThmGenMinimization} Let $B$ be a Banach function space that is uniformly convex and let $X$ be
a r.i. space with absolutely continuous norm such that $B \hookrightarrow X$. Let $q>0,p\in\lbrack1,\infty),k\in{\mathbb{N%
}}$ and $\lambda>0$. Assume that $f\in B$. For each $s\in(0,1)$, let $\bar{u}%
_{s}$ be a near-minimizer  for the functional 
\begin{equation*}
E_{s}(u):=\Vert f-u\Vert_{B}^{q}+\lambda(1-s)^{1/p}\Vert u\Vert_{\dot{B}%
_{p}^{sk}X}. 
\end{equation*}
Then (up to a subfamily) $\bar{u}_{s}\rightarrow\bar{u}$ in $B$ as $%
s\rightarrow1^{-}$, where $\bar{u}$ is a minimizer of 
\begin{equation*}
E_1(u):=\Vert f-u\Vert_{B}^{q}+\lambda F(u) 
\end{equation*}
and\footnote{%
Note that $F(u)\approx\Vert u\Vert_{\dot{W}^{k}X}$ if $u \in \dot{W}^k X$, cf. \eqref{PF12.1}.} 
\begin{equation*}
F(u):=\bigg(\frac{1}{kp}\int_{\mathbb{S}^{N-1}}\bigg\|\sum_{|\alpha|=k} \frac{k!}{\alpha!} \, 
\omega^{\alpha}D^{\alpha}u\bigg\|_{X}^{p}\,d\sigma^{N-1}(\omega)\bigg)%
^{1/p}. 
\end{equation*}
\end{theorem}

Before going to the proof of Theorem \ref{ThmGenMinimization}, let us 
state the special case when $X=L^{p}$ and $k=1$.

\begin{corollary}
\label{CorMinSob} Let $B$ be a Banach function space that is uniformly convex  with $B \hookrightarrow L^p$ for some $p \in [1, \infty)$ (e.g., $B = L^r$ with\footnote{The case $B=L^1$ is also admissible, cf. Remark \ref{Remark15}.} $r \in [p, \infty)$). Let $q>0$ and $\lambda>0$. Assume that $f \in B$.  For
each $s\in(0,1)$, let $\bar{u}_{s}$ be a near-minimizer of 
\begin{equation*}
E_{s}(u):=\Vert f-u\Vert_{B}^{q}+\lambda(1-s)^{1/p}\Vert u\Vert_{\dot{W}%
^{s,p}}. 
\end{equation*}
Then (up to a subfamily) $\bar{u}_{s}\rightarrow\bar{u}$ in $B$ as $s\rightarrow1^{-}$, where $%
\bar{u}$ is a minimizer of 
\begin{equation*}
E_1(u):=\Vert f-u\Vert_{B}^{q}+\lambda\bigg(\frac{C_{p,N}}{p}\bigg)^{1/p}F(u) 
\end{equation*}
with 
\begin{equation*}
F(u)=\left\{ 
\begin{array}{cl}
\Vert\nabla u\Vert_{L^{p}} & \qquad  \text{if}\qquad p\in(1,\infty), \\ 
\Vert u\Vert_{\dot{BV}} & \qquad  \text{if}\qquad p=1.%
\end{array}
\right. 
\end{equation*}
Here, the constant $C_{p,N}$ is given by \eqref{PF10.1}. 
\end{corollary}

\begin{proof}[Proof of Theorem \protect\ref{ThmGenMinimization}]
One may be tempted to apply directly Theorem \ref{ThmVarIntro} (taking into account
the switch in the roles of the parameters $s$ and $\theta$) to the family of
functionals 
\begin{equation*}
\widetilde{E}_{s}(u):=\Vert f-u\Vert_{B}^{q}+\lambda \Vert u\Vert_{(X,%
\dot{W}^{k}X)_{s,p}^{\blacktriangleleft}},\qquad s\in(0,1), 
\end{equation*}
and 
\begin{equation*}
\widetilde{E}_1(u):=\Vert f-u\Vert_{B}^{q}+\lambda\Vert u\Vert_{G_{1}(X,\dot {W}%
^{k}X)}. 
\end{equation*}
In this regard, recall that $X \cap \dot{W}^k X = W^k X \hookrightarrow X$ compactly (cf. the proof of Corollary \ref{CorCmp1}). Unfortunately, we only have at our disposal the characterizations %
\eqref{PD28.3} and \eqref{PD28.5i}, namely, 
\begin{equation*}
\Vert u\Vert_{(X,\dot{W}^{k}X)_{s,p}}\approx\Vert u\Vert_{\dot{B}%
_{p}^{sk}X}\qquad\text{and}\qquad\Vert u\Vert_{G_{1}(X,\dot{W}%
^{k}X)}\approx\Vert u\Vert_{\dot{W}^{k}X}
\end{equation*}
and hence 
\begin{equation*}
E_{s}(u)\approx\widetilde{E}_{s}(u)\qquad\text{and}\qquad E_1(u)\approx 
\widetilde{E}_1(u). 
\end{equation*}
In principle, the asymptotic behavior of near-minimizers of $\widetilde{E}%
_{s}$ may be quite different from the corresponding behavior of near-minimizers of $E_{s}$. But it is possible to circumvent the obstruction
following, line by line, the proof of Theorem \ref{ThmVarIntro} but now applying the
exact formulae derived in Section \ref{SectionExact} (more precisely, %
\eqref{Pdjjsjas3}) and Corollary \ref{CorCmp1}. 
\end{proof}


\subsection{A variant of Theorem \ref{ThmVarIntro} involving $K$-functionals%
}

Assume that $B,A_{0}$ and $A_{1}$ satisfy the assumptions of Theorem \ref{ThmVarIntro}. Given $f\in B, \lambda, q>0$, and $\varepsilon
\in(0,1)$, consider the functional 
\begin{equation*}
\mathfrak{F}_{\varepsilon}(u):=\Vert f-u\Vert_{B}^{q}+\lambda\,\frac {%
K(\varepsilon,u;A_{0},A_{1})}{\varepsilon}. 
\end{equation*}

\begin{theorem}
\label{Thm7new} Suppose  that $\bar{u}_{\varepsilon}$ is a near-minimizer of $%
\mathfrak{F}_{\varepsilon}$. Then (up to a subfamily) $\bar{u}_{\varepsilon}\rightarrow\bar{u}$
in $B$ as $\varepsilon\rightarrow0^{+}$, where $\bar{u}$ is a minimizer of
the functional $E_1$ defined by \eqref{23.9.4} and $E_1(\bar{u}%
)=\lim_{\varepsilon \rightarrow0^{+}}\mathfrak{F}_{\varepsilon}(\bar{u}%
_{\varepsilon}).$ 
\end{theorem}

The proof of Theorem \ref{Thm7new} consists of a simple modification of the
one given for Theorem \ref{ThmVarIntro}  but now relying on Theorems \ref{ThmGKFunct}
and \ref{Thm5**} (with $T= \text{id}$, $B_0= B_1=A_{0}$, and $A_1$ is replaced by $A_0 \cap A_1$). We
leave further details to the interested reader.

\section{Applications to image processing}

\label{SectionIP}

Unless otherwise stated, in this section we shall work on function spaces  defined on $\mathbb{T}^{2}\equiv\lbrack0,1]^{2}$. The
classical restoration problem in image processing can be stated as follows:
Given an observed image $f$ on $\mathbb{T}^{2}$ and a fidelity parameter $%
\lambda>0$, we seek solutions of 
\begin{equation}
\bar{u}=\text{argmin}_{u}\{\Vert f-u\Vert_{B}^{q}+\lambda F(u)\}, 
\label{IP1}
\end{equation}
where $q>0$ and the functional $F$ is adequately chosen to capture the main
features of images (edges, texture, ...), while the error norm on $B$
measures spurious features of images (noise, blur, ...). Nowadays, it is
well understood that the total variation $F(u)=\Vert u\Vert_{\dot{BV}%
}=\Vert\nabla u\Vert_{\mathcal{M}}$ is a natural candidate for the
regularization term in \eqref{IP1} allowing for discontinuities across
edges. In particular, the  Rudin--Osher--Fatemi model \cite{ROF}
corresponds to the choices $B=L^{2},F(u)=\Vert u\Vert_{\dot{BV}}$ and $q=2$: 
\begin{equation}
\bar{u}=\text{argmin}_{u}\{\Vert f-u\Vert_{L^{2}}^{2}+\lambda\Vert u\Vert_{%
\dot{BV}}\}.   \tag{ROF}  \label{IP2}
\end{equation}
Recall that \eqref{IP2} has a unique solution $\bar{u}$ for $f \in L^2$. However, this model lacks contrast invariance
in the sense that if $\bar{u}$ is the solution of \eqref{IP2}, it may happen
that $c\bar{u}$ is not the solution of \eqref{IP2} associated to the input $%
cf$. This obstruction can be overcome with the well-known Chan--Esedoglu
model \cite{CE}, which corresponds to $B=L^{1},F(u)=\Vert u\Vert_{\dot{BV}}$
and $q=1$: 
\begin{equation}
\bar{u}=\text{argmin}_{u}\{\Vert f-u\Vert_{L^{1}}+\lambda\Vert u\Vert _{\dot{%
BV}}\}.   \tag{CE}  \label{IP3}
\end{equation}
The lack of strong convexity inherited to $L^{1}$
does not ensure uniqueness of solution to \eqref{IP3}.

\subsection{From the non-local to the local setting: asymptotic analysis\label{sec:asymp}}

Both \eqref{IP2} and \eqref{IP3} are examples of local denoising models
whose corresponding regularizing term (namely, $F(u)=\Vert u\Vert_{\dot{BV}}$) is
defined in terms of a local operator. Despite their success, these models
may consider fine structures, textures and details of observed images with
low regularity as noise contribution. This issue served as a motivation  to
introduce denoising models based on non-local operators \cite{BCM, GO, GO2}. In particular, the
non-local counterparts of \eqref{IP2} and \eqref{IP3} given by 
\begin{equation*}
\bar{u}_s=\text{argmin}_{u}\{\Vert f-u\Vert_{L^{2}}^{2}+\lambda\Vert u\Vert_{%
\dot{W}^{s,1}}\}
\end{equation*}
and 
\begin{equation*}
\bar{u}_s=\text{argmin}_{u}\{\Vert f-u\Vert_{L^{1}}+\lambda\Vert u\Vert_{\dot {%
W}^{s,1}}\}
\end{equation*}
where $s\in(0,1),$ have been shown to satisfy remarkable properties, see
e.g. \cite{AK}, \cite{Be} and \cite{NO}.

As a special case of Corollary \ref{CorMinSob} (or more precisely, its
periodic counterpart) with $N=2,p=1$ and $B=L^{r}$, we obtain the following
result.

\begin{corollary}
\label{CoroIP1} Let $f\in L^{r},r\in\lbrack1,\infty),$ and $q, \lambda>0$. For each $%
s\in(0,1)$, let $\bar{u}_{s}$ be a near-minimizer of 
\begin{equation*}
E_{s}(u)=\Vert f-u\Vert_{L^{r}}^{q}+\frac{\lambda}{4}(1-s)\Vert u\Vert _{%
\dot{W}^{s,1}}.   
\end{equation*}
Then there exists $\bar{u}\in L^{r}$ such that $\bar{u}_{s}\rightarrow\bar{u}
$ (up to a subfamily) in $L^{r}$ as $s\rightarrow1^{-}$, where $\bar{u}$ is a
minimizer of 
\begin{equation}
E_1(u)=\Vert f-u\Vert_{L^{r}}^{q}+\lambda\Vert u\Vert_{\dot{BV}}. 
\label{PF15.1}
\end{equation}
In particular, \eqref{PF15.1} with $r=q=2$ coincides with \eqref{IP2} and %
\eqref{PF15.1} with $r=q=1$ coincides with \eqref{IP3}.
\end{corollary}

\begin{remark}
Given $Q_{0}$ a cube in $\mathbb{R}^{N}$, let us denote by $\kappa
_{\varepsilon}$ the counterpart of the functional introduced in %
\eqref{BMO-type} adequately adapted to $Q_{0}$ (in the sense that elements
of $\mathcal{F}_{\varepsilon}$ are contained in $Q_{0}$). Motivated by the $%
\Gamma $-convergence formula \eqref{ARBDN}, the variational analysis of the
following functionals has been carried out in \cite{GuS}: 
\begin{equation*}
\mathfrak{F}_{\varepsilon}(u)=\Vert f-u\Vert_{L^{q}(Q_{0})}^{q}+\kappa
_{\varepsilon}(u), 
\end{equation*}
where $q\in\lbrack1,\infty)$. If $q>1$ then \cite[Theorem 1.3]{GuS} claims
convergence (up to subfamily) in $L^{q}$ of the unique minimizers of $%
\mathfrak{F}_{\varepsilon}$ to the unique minimizer of 
\begin{equation*}
\Vert f-u\Vert_{L^{q}(Q_{0})}^{q}+\frac{1}{4}\Vert u\Vert_{\dot{BV}(Q_{0})}. 
\end{equation*}
If $q=1$ there is an analog in terms of near-minimizers, cf. \cite[Theorem
4.1]{GuS}. In particular, the minimization problems \eqref{IP2} and %
\eqref{IP3} can be understood asymptotically
from $\mathfrak{F}_{\varepsilon}$. Here we observe that this result is
intimately connected with a special case of the general phenomenon based on $%
K$-functionals described in Theorem \ref{Thm7new}. Indeed, let $%
B=L^{q}(Q_{0}),A_{0}=L^{1}(Q_{0})$ and $A_{1}=\dot{W}^{1,1}(Q_{0})$ (and so $A_0 \cap A_1 = W^{1, 1}(Q_0) \hookrightarrow A_0 = L^1(Q_0)$ compactly) and
recall that the functional $\kappa_{\varepsilon}(f)$ is comparable to $%
\frac {K(\varepsilon,f;A_{0},A_{1})}{\varepsilon}$ (cf. \eqref{K-Funct-DS})
and $G_{1}(A_{0},A_{1})=\dot{BV}(Q_{0})$ (cf. \eqref{PJ1.1}).
\end{remark}

\subsection{From the discrete to the continuous setting: rates of convergence via $K$-functionals}\label{Section9.2}

We start by describing the functional framework. Let $e_{j},j=1,2,$ be the two coordinate
vectors in $\mathbb{R}^{2}$. The total variation can
be defined as  
\begin{equation}
\text{TV}(f):=\sup_{h>0}h^{-1}\sum_{j=1}^{2}\Vert\Delta_{he_{j}}f%
\Vert_{L^{1}}=\lim_{h\rightarrow0^{+}}h^{-1}\sum_{j=1}^{2}\Vert\Delta_{he_{j}}f%
\Vert_{L^{1}}.   \label{PF18.2}
\end{equation}
Recall the well-known fact that (cf. \cite[(2.8)]{CDPX}) 
\begin{equation*}
\text{TV}(f)\approx\Vert f\Vert_{\dot{BV}}=\sup\bigg\{\int_{\mathbb{T}^{2}}f%
\text{div}\phi:\phi\in C^{1}(\mathbb{T}^{2};\mathbb{R}^{2}),\quad
\Vert\,|\phi|\,\Vert_{L^{\infty}}\leq1\bigg\}. 
\end{equation*}
Here, $|x|=\sqrt{x_{1}^{2}+x_{2}^{2}}$ for $x=(x_{1},x_{2})\in\mathbb{R}^{2}$%
. Following \cite{BNS, BCCLL}, we are interested in the $\ell_{1}$-anisotropic
version of \eqref{IP2} associated to \eqref{PF18.2}, i.e., 
\begin{equation}
E(u):=\frac{1}{2}\Vert f-u\Vert_{L^{2}}^{2}+\text{TV}(u).   \label{PF18.3}
\end{equation}
For each $n\in\mathbb{N}$, let
\begin{equation}
S_{i,j}^{h}=[ih,(i+1)h)\times\lbrack jh,(j+1)h),\qquad0\leq i,j\leq n-1, 
\label{PF18.7}
\end{equation}
with side length $h=1/n$ and let $X_n = \operatorname{span} \mathbf{1}_{S^h_{i, j}}$. Consider the following discretized total variation%
\footnote{%
As usual, in the periodic setting, $S_{n,j}^{h}$ (respectively, $S_{i,n}^{h}$%
) should be identified with $S_{0,j}^{h}$ (respectively, $S_{i,0}^{h}$).} 
\begin{equation}
\text{TV}_{n}(f):=h%
\sum_{i,j=0}^{n-1}(|f_{S_{i,j}^{h}}-f_{S_{i+1,j}^{h}}|+|f_{S_{i,j}^{h}}-f_{S_{i,j+1}^{h}}|), 
\label{PF18.6}
\end{equation}
where $f_{S_{i,j}^{h}}:=\frac{1}{|S_{i,j}^{h}|}\int_{S_{i,j}^{h}}f$ is the mean value of $f$ on $S^h_{i, j}$. Let $P_n$ denote the projection operator onto $X_n$  given by 
\begin{equation}
P_{n}f(x):=\sum_{0\leq i,j\leq n-1} f_{S_{i, j}^h} \mathbf{1}_{S_{i,j}^{h}}(x).  \label{PF19.2}
\end{equation}
Then  (cf. \cite[(2.12)%
]{CDPX})
\begin{equation}\label{24.9.1}
\text{TV}_{n}(f) = \text{TV}(P_{n}f).  
\end{equation}
Accordingly, given $f \in L^2$,  the corresponding discretized version of \eqref{PF18.3} is
defined as 
\begin{equation*}
E_{n}(u):=\frac{1}{2}\Vert f-u\Vert_{L^{2}}^{2}+\text{TV}_{n}(u). 
\end{equation*}
Let us denote by $\bar{u}$ and $\bar{u}_{n}$ the unique minimizers\footnote{Note that $\bar{u}_n =  \argmin_{u \in X_n} E(u)$.} of $%
E$ and $E_{n}$, more precisely, 
$$
	\bar{u} := \argmin_{u \in \dot{BV}} E(u), \qquad \bar{u}_n := \argmin_{u \in X_n} E_n(u),
$$
with corresponding values
$$
	\bar{E} := E(\bar{u}), \qquad \bar{E}_n := E_n (\bar{u}_n). 
$$



\subsubsection{Proof of Theorem \ref{Thm1Intro}}
It follows from (cf. \eqref{24.9.1} and \cite[(2.11)]{CDPX}) 
\begin{equation}\label{24.9.8}
\text{TV}_{n}(u) = \text{TV}(P_n u) \leq\text{TV}(u) 
\end{equation}
and the fact $P_n^2 = P_n$ that 
\begin{align}  
\bar{E} = \frac{1}{2}\|f-\bar{u}\|_{L^{2}}^{2} + \text{TV}(\bar{u}) &\geq%
\frac{1}{2}\|f-\bar{u}\|_{L^{2}}^{2} + \text{TV}(P_n\bar{u}) \nonumber \\
& = E (P_n \bar{u}) + \frac{1}{2} \Big(\|f-\bar{u}\|^2_{L^2} - \|f-P_n \bar{u}\|^2_{L^2} \Big). \label{24.9.2}
\end{align}
Using that $\bar{u}_n$ is a minimizer of $E$ on $X_n$ and $P_n \bar{u} \in X_n$, we derive (by strong convexity properties of $E$) 
$$
	E(P_n \bar{u}) \geq E(\bar{u}_n) + \frac{1}{2} \|P_n \bar{u}-\bar{u}_n\|^2_{L^2}
$$
or equivalently (noting that $E_n = E$ on $X_n$)
$$
	E(P_n \bar{u}) \geq E_n(\bar{u}_n) + \frac{1}{2} \|P_n \bar{u}-\bar{u}_n\|^2_{L^2} = \bar{E}_n + \frac{1}{2} \|P_n \bar{u}-\bar{u}_n\|^2_{L^2}. 
$$
Inserting this estimate into \eqref{24.9.2}, we arrive at
\begin{align*}
\bar{E}  \geq \bar{E}_n + \frac{1}{2} \Big(\|f-\bar{u}\|^2_{L^2} - \|f-P_n \bar{u}\|^2_{L^2} + \|P_n \bar{u}-\bar{u}_n\|_{L^2}^2 \Big)
\end{align*}
and then, by the triangle inequality, 
\begin{align}
	\frac{1}{4}\|\bar{u}-\bar{u}_n\|_{L^2}^2 &\leq \frac{1}{2} \Big(\|P_n \bar{u} - \bar{u}_n\|_{L^2}^2 + \|P_n \bar{u} - \bar{u}\|_{L^2}^2 \Big) \nonumber \\
	& \leq  \bar{E}-\bar{E}_n  + \frac{1}{2} \|P_n \bar{u} - \bar{u}\|_{L^2}^2  - \frac{1}{2} \|f-\bar{u}\|^2_{L^2} + \frac{1}{2} \|f-P_n \bar{u}\|_{L^2}^2. \label{24.9.4}
\end{align}

Elementary computations enable to write 
\begin{equation}\label{24.9.3}
\Vert f-P_{n}\bar{u}\Vert_{L^{2}}^{2}=\Vert
f-P_{n}f\Vert_{L^{2}}^{2}+\Vert P_{n}f-P_{n}\bar{u}\Vert_{L^{2}}^{2}+2%
\int_{\mathbb{T}^{2}}(f-P_{n}f)(P_{n}f-P_{n}\bar{u}).  
\end{equation}
Observe that $f-P_{n}f$ is orthogonal to $X_n$ and, in particular,
this is the case for $P_{n}f-P_{n}\bar{u} \in X_n$ so that  
\begin{equation*}
\int_{\mathbb{T}^{2}}(f-P_{n}f)(P_{n}f-P_{n}\bar{u})=0. 
\end{equation*}
Hence \eqref{24.9.3} can be updated to 
\begin{equation*}
\Vert f-P_{n}\bar{u}\Vert_{L^{2}}^{2}=\Vert
f-P_{n}f\Vert_{L^{2}}^{2}+\Vert P_{n}f-P_{n}\bar{u}\Vert_{L^{2}}^{2}
\end{equation*}
and applying the basic fact that $P_n$ is a contraction on $L^2$ we obtain
\begin{equation}\label{24.9.5}
	\Vert f-P_{n}\bar{u}\Vert_{L^{2}}^{2} \leq \Vert
f-P_{n}f\Vert_{L^{2}}^{2}+\Vert f-\bar{u} \Vert_{L^{2}}^{2}.
\end{equation}
Now plugging \eqref{24.9.5} into \eqref{24.9.4}, one has that 
\begin{equation}\label{24.9.6}
	\frac{1}{4}\|\bar{u}-\bar{u}_n\|_{L^2}^2 \leq \bar{E}-\bar{E}_n  + \frac{1}{2} \Big( \|P_n \bar{u} - \bar{u}\|_{L^2}^2 + \|f-P_n f\|^2_{L^2} \Big).   
\end{equation}


Next we estimate the term  $\|P_n \bar{u}-\bar{u}\|_{L^2}$. To proceed, we first observe that from  \eqref{24.9.8} and the fact that $P_n^2 = P_n$,  we obtain 
\begin{align*}
	E_{P_n f} (P_n u) &= \frac{1}{2} \|P_n f - P_n u\|_{L^2}^2 + \text{TV} (P_n u)  \\
	& =  \frac{1}{2} \|P_n(P_n f - u)\|_{L^2}^2 + \text{TV} (P_n u)  \\
	&  \leq \frac{1}{2} \|P_n f-u\|_{L^2}^2 + \text{TV} (u) = E_{P_n f}(u)
\end{align*}
for every $u \in \dot{BV}$. Consider the (nonlinear) operator $T$ defined on $L^2$ as 
$$
	Tf := \bar{u} = \argmin_{u \in \dot{BV}} E_f(u),
$$
where $E_f$ is given by \eqref{PF18.3} with given data $f$. As a consequence, if $u = T P_n f$ 
 then 
$$
	E_{P_n f}(T P_n f) \leq E_{P_n f} (P_n (T P_n f) ) \leq E_{P_n f}(T P_n f),
$$
i.e., $E_{P_n f} (T P_n f) = E_{P_n f} (P_n(T P_n f))$ and, by uniqueness of minimizer, we derive that $T P_n f = P_n(T P_n f) \in X_n$. Accordingly, since $P_n$ is an orthogonal projection onto $X_n$, we get
\begin{equation}\label{24.9.9}
	\|\bar{u}-P_n \bar{u}\|_{L^2} = \inf_{v \in X_n} \|\bar{u} -v\|_{L^2} \leq \|\bar{u}-T P_n f\|_{L^2}  = \|T f - T P_n f\|_{L^2}  \leq \|f-P_n f\|_{L^2},
\end{equation}
where the last step is well known, cf.  \cite[Proposition 12.27]{BC}.

Applying \eqref{24.9.9} in \eqref{24.9.6}, we have
\begin{equation}\label{24.9.10}
	\frac{1}{4}\|\bar{u}-\bar{u}_n\|_{L^2}^2 \leq \bar{E}-\bar{E}_n  +  \|f-P_n f\|^2_{L^2}.  
\end{equation}

We claim that for $p\in(1,2]$ it holds 
\begin{equation}
\Vert f-P_{n}f\Vert_{L^{2}}\lesssim K(n^{-2+\frac{2}{p}},f;L^{2},\dot{W}%
^{1,p}).   \label{19.5Nov}
\end{equation}
Indeed, given any $g\in\dot{W}^{1,p}$, it follows from the contraction properties of $P_n$ on $L^2$  that 
\begin{align}
\Vert f-P_{n}f\Vert_{L^{2}} & \leq\Vert f-g\Vert_{L^{2}}+\Vert
P_{n}f-P_{n}g\Vert_{L^{2}}+\Vert g-P_{n}g\Vert_{L^{2}}  \notag \\
& \leq 2\Vert f-g\Vert_{L^{2}}+\Vert g-P_{n}g\Vert_{L^{2}}. 
\label{23.1Nov}
\end{align}
Recall that the classical $2$D Poincar\'{e} inequality asserts that 
\begin{equation}
\bigg(\frac{1}{|Q|}\int_{Q}|g(x)-g_{Q}|^{2}\,dx\bigg)^{1/2}\lesssim \ell(Q)%
\bigg(\frac{1}{|Q|}\int_{Q}|\nabla g(x)|^{p}\,dx\bigg)^{1/p}   \label{PIneq}
\end{equation}
for every cube $Q$ with side length  $\ell(Q)$ and $1\leq p\leq2$. Applying \eqref{PIneq} to each cube $%
S_{i,j}^{h}$ (cf. \eqref{PF18.7}), we derive 
\begin{align*}
\Vert g-P_{n}g\Vert_{L^{2}}^{2} & =\sum_{0\leq i,j\leq n-1}\int_{S_{i,j}^{h}}%
|g(x)- g_{S_{i, j}^h} |^{2}\,dx \\
& \lesssim\sum_{0\leq i,j\leq n-1}|S_{i,j}^{h}|^{1-\frac{2}{p}}\ell
(S_{i,j}^{h})^{2}\bigg(\int_{S_{i,j}^{h}}|\nabla g(x)|^{p}\,dx\bigg)^{2/p} \\
& =n^{-4(1-\frac{1}{p})}\sum_{0\leq i,j\leq n-1}\bigg(\int_{S_{i,j}^{h}}|%
\nabla g(x)|^{p}\,dx\bigg)^{2/p} \\
& \leq n^{-4(1-\frac{1}{p})}\bigg(\sum_{0\leq i,j\leq
n-1}\int_{S_{i,j}^{h}}|\nabla g(x)|^{p}\,dx\bigg)^{2/p} \\
& =n^{-4(1-\frac{1}{p})}\Vert\nabla g\Vert_{L^{p}}^{2}
\end{align*}
where we have also used that $p\leq2$ in the penultimate step. Consequently,
we have (cf. \eqref{23.1Nov}) 
\begin{equation*}
\Vert f-P_{n}f\Vert_{L^{2}}\lesssim\Vert f-g\Vert_{L^{2}}+n^{-2(1-\frac{1}{p}%
)}\Vert\nabla g\Vert_{L^{p}}, 
\end{equation*}
therefore taking the infimum over all $g\in\dot{W}^{1,p}$, we arrive at the
desired estimate \eqref{19.5Nov}.

Putting together \eqref{24.9.10} and \eqref{19.5Nov}, we obtain  (note that $\bar{E} \leq \bar{E}_n$)
\begin{equation*}
\|\bar{u}-\bar{u}_{n}\|_{L^{2}} \lesssim K(n^{-2(1-\frac{1}{p})}, f; L^{2}, 
\dot{W}^{1, p})
\end{equation*}
and
$$
	\bar{E}_n- \bar{E} \lesssim  K(n^{-2(1-\frac{1}{p})}, f; L^{2}, 
\dot{W}^{1, p})^2. 
$$

The proof of \eqref{29.9.1} is a slight modification of the above methodology. Indeed, if $u_1 = T f_1$ and $u_2 = T f_2$ then
$$
	E_{f_1} (u_2) -E_{f_1}(u_1) \geq \frac{1}{2} \|u_1-u_2\|_{L^2}^2
$$
and a similar estimate holds when the roles of $u_1$ and $u_2$ are switched. Summing these two estimates, we arrive at  
\begin{equation}\label{29.9.2}
  \langle f_1-f_2, u_1-u_2 \rangle \geq  \|u_1-u_2\|^2_{L^2}.
\end{equation}
Next we wish to apply \eqref{29.9.2} with $u_1 = \bar{u} = T f$ and $u_2 = \bar{u}_n$. To do this, we make the basic observation from \cite[(26)]{KSe} that $\bar{u}_n = T P_n f$. Then
$$
	\langle f-P_n f, \bar{u}-\bar{u}_n \rangle \geq \|\bar{u}-\bar{u}_n\|^2_{L^2}
$$
and, using orthogonality properties, one may rewrite the last inequality as 
$$
	\langle f-P_n f, \bar{u}-P_n \bar{u} \rangle \geq \|\bar{u}-\bar{u}_n\|^2_{L^2}. 
$$   
By the Cauchy-Schwarz inequality and applying \eqref{19.5Nov} twice, we have 
\begin{equation}\label{29.9.4}
	\|\bar{u}-\bar{u}_n\|^2_{L^2} \lesssim K(n^{-2+\frac{2}{p}}, f; L^2, \dot{W}^{1, p}) K(n^{-1}, \bar{u}; L^2, \dot{W}^{1, 2}).  
\end{equation}
Furthermore, using the characterization of $K$-functionals in terms of moduli of smoothness (cf. \eqref{PD28.1}),  it is easy to see that 
$$
	K(t, \bar{u}; L^2, \dot{W}^{1, 2})^2 \lesssim \|\bar{u}\|_{L^\infty} K(t, \bar{u}; L^1, \dot{W}^{1, 1})
$$
for all $t > 0$. In fact, the last estimate admits the following self-improving  formulation  (cf. \eqref{29.9.3})
$$
	K(t, \bar{u}; L^2, \dot{W}^{1, 2})^2 \lesssim \|\bar{u}\|_{L^\infty} K(t, \bar{u}; L^1, \dot{BV}).
$$
Inserting this into \eqref{29.9.4}, we find that 
\begin{align*}
	\|\bar{u}-\bar{u}_n\|^2_{L^2} &\lesssim  K(n^{-2+\frac{2}{p}}, f; L^2, \dot{W}^{1, p}) \Big( \|\bar{u}\|_{L^\infty} K(n^{-1}, \bar{u}; L^1, \dot{BV}) \Big)^{1/2} \\
	& \lesssim n^{-1/2}K(n^{-2+\frac{2}{p}}, f; L^2, \dot{W}^{1, p})   (\|\bar{u}\|_{L^\infty} \text{TV}(\bar{u}))^{1/2}. 
\end{align*}
This finishes the proof of \eqref{29.9.1}. \qed

\subsubsection{Some consequences of Theorem \ref{Thm1Intro}}

%

\begin{theorem}
\label{Coro8.32} Let $f\in B_{2,\infty}^{\alpha}, \, \alpha \in (0, 1)$. Then 
\begin{equation}\label{27.9.3}
\Vert\bar{u}-\bar{u}_{n}\Vert_{L^{2}}\lesssim \begin{cases} \|f\|_{\dot{B}^\alpha_{2, \infty}} n^{-\alpha} & \qquad  \text{if}\qquad  \alpha \geq \frac{1}{2}, \\
	(\|\bar{u}\|_{L^\infty} \emph{TV}(\bar{u}))^{1/4}	\|f\|_{\dot{B}^\alpha_{2, \infty}}^{1/2} n^{-\frac{\alpha}{2} -\frac{1}{4}} & \qquad  \text{if} \qquad  \alpha < \frac{1}{2}.
\end{cases}
\end{equation}
Furthermore, these estimates are optimal in the following sense:
\begin{enumerate}
\item[(i)] If $\alpha \in [\frac{1}{2}, 1)$ then there exists $f \in B^{\alpha}_{2, \infty}$ such that
\begin{equation}\label{4.10.4}
	\|\bar{u}-\bar{u}_n\|_{L^2}^2 \approx n^{-\alpha}, \qquad \forall n \geq 4. 
\end{equation}
\item[(ii)] If $\alpha \in (0, \frac{1}{2})$ then  there exists $n_0 \in \N$ such that   \begin{equation}\label{1.10.5}
			 	\sup_{f \in \mathcal{K}} \|\bar{u}-\bar{u}_n\|_{L^2} \approx n^{-\frac{\alpha}{2} - \frac{1}{4}}, \qquad \forall n \geq n_0, 
			 \end{equation}
			 where
			 \begin{equation}\label{1.10.3}
			 \mathcal{K} = \{f \in L^2: \|f\|_{\dot{B}^\alpha_{2, \infty}} \leq C, \quad \|\bar{u}\|_{L^\infty} \leq  1, \quad \emph{TV}(\bar{u}) \leq 2 \}
			 \end{equation}
			 for a certain constant $C > 0$ that is independent of $n$. 
\end{enumerate}

\end{theorem}

\begin{proof}
Applying Theorem \ref{Thm1Intro} with $p=2$ and the characterization of the $K
$-functional for $X=L^{2}$ and $k=1$ given in \eqref{PD28.1} we obtain (to simplify the exposition we shall omit the underlying equivalence constants that depend on $f$ and $\bar{u}$)
\begin{align}
\Vert\bar{u}-\bar{u}_{n}\Vert_{L^{2}} &\lesssim K(n^{-1},f;L^{2},\dot{W}%
^{1,2})^{1/2} \min \{n^{-1/4}, K(n^{-1}, f; L^2, \dot{W}^{1, 2})^{1/2}\} \nonumber \\
&\approx \omega(f, n^{-1})_2^{1/2} \min \{n^{-1/4}, \omega(f, n^{-1})_2^{1/2}\}. 
\label{PF19.12}
\end{align}
Combining \eqref{PF19.12} with the elementary estimates 
\begin{equation*}
\Vert f\Vert_{\dot{B}_{2,\infty}^{\alpha}}=\sup_{h \neq0}\frac{\Vert\Delta
_{h}f\Vert_{L^{2}}}{|h|^{\alpha}}\geq\sup_{0<|h|<n^{-1}}\frac{\Vert\Delta_{h}f%
\Vert_{L^{2}}}{|h|^{\alpha}}\geq
n^{\alpha}\sup_{|h|<n^{-1}}\Vert\Delta_{h}f\Vert_{L^{2}} = n^\alpha \omega(f, n^{-1})_2
\end{equation*}
we arrive at the following estimate
$$
	\Vert\bar{u}-\bar{u}_{n}\Vert_{L^{2}}  \lesssim n^{-\alpha/2} \min \{n^{-1/4}, n^{-\alpha/2}\}.
$$

Next we focus on the sharpness assertion of \eqref{27.9.3}.   We will restrict our attention to the case $\alpha \neq 1/2$, since the optimality with $\alpha = 1/2$ is a well-know result (cf. \cite[Remark 6.6]{BNS} and \cite[Remark 4.7]{KSe}). Assume first that  $\alpha \in (1/2, 1)$ and consider the prototypical example of extremal local singularity in $B^\alpha_{2, \infty}$  given by (cf. \cite[Lemma 1, p. 44]{RS})
	$$
		\psi(x_1) =(\text{sgn } x_1) |x_1|^{\alpha-\frac{1}{2}}
	$$
	for $|x_1|$ sufficiently small (say $|x_1| \leq 1/8$). Let 
	$$
		g_1(x_1) = \begin{cases}
			\psi (\frac{1}{8}) &\text {if} \qquad x_1 \in [0, \frac{1}{4}], \\
			\varphi(x_1) & \text{if} \qquad x_1 \in [\frac{1}{4}, \frac{1}{2}], \\
			\psi(-\frac{1}{8}) & \text{if} \qquad x_1 \in [\frac{1}{2}, \frac{3}{4}], \\
			\psi(x_1-\frac{7}{8}) & \text{if} \qquad x_1 \in [\frac{3}{4}, 1],
		\end{cases}
	$$
			where $\varphi \in C^\infty$
			 is decreasing with  $\varphi(\frac{1}{4}) = \psi(\frac{1}{8}) > 0$ and $\varphi(\frac{1}{2}) = \psi(-\frac{1}{8}) < 0$. Let $g(x) = g_1(x_1)$ for $x = (x_1, x_2) \in \T^2$. Assume that $t \in (0, 1/8)$. By construction, the sign of $g_1'$ does not change  on any interval of length $t$ and then 
			 $$
			 	|g_1(x_1+t)-g_1(x_1)| = \bigg|\int_{x_1}^{x_1 + t} g_1'(r) \, dr \bigg| = \int_{x_1}^{x_1+t} |g_1'(r)| \, dr
			 $$
			 leading to 
			 $$
			 	\|\Delta_{t e_1} g\|_{L^1(\T^2)} = \|\Delta_t g_1\|_{L^1(\T)} =  \int_0^1  \int_{0}^{t} |g_1'(x_1 + r)| \, dr \, d x_1 = t \|g_1'\|_{L^1(\T)}.
			 $$
			 Here, we note that $\|g_1'\|_{L^1(\T)} < \infty$ since $\alpha > 1/2$. Moreover, it is clear that $\Delta_{t e_2} g(x) =0$. Hence
			 $$
			 	\text{TV}(g) = \lim_{t \to 0^{+}} t^{-1} \|\Delta_{t e_1} g\|_{L^1(\T^2)} =  \|g_1'\|_{L^1(\T)}. 
			 $$
			Let $z_1 \in C^\infty(\T)$ with $|z_1| \leq 1$ such that $z_1 \equiv -1$ on $[\frac{1}{4}, \frac{1}{2}]$ and $z_1 \equiv 1$ on $[\frac{3}{4}, 1]$. Then we define $z(x) = (z_1(x_1), 0)$ and $\xi(x) = - \text{div } z(x) = -z_1'(x_1)$. Then
			 $$
			 	\langle \xi, g \rangle = - \int_0^1 z_1'(x_1) g_1(x_1) \, dx_1 = \int_0^1 z_1(x_1) g_1'(x_1) \, dx_1 = \|g_1'\|_{L^1(\T)}  = \text{TV} (g). 
			 $$
			 According to the well-known dual characterization of minimizers (see e.g. \cite[(2.11)]{CCCNP}), we derive that  $g$ is the unique minimizer of \eqref{PF18.3} with data $f = g + \xi$. Furthermore, using that $\psi \in B^\alpha_{2, \infty}$, it is easy  to check that $g_1 \in B^\alpha_{2, \infty}(\T)$ and, in particular, $g \in B^\alpha_{2, \infty}(\T^2)$. Then, taking into account that $\xi \in C^\infty(\T^2)$, we also derive that  $f \in B^\alpha_{2, \infty}(\T^2)$. On the other hand, from the fact that $P_n$ is an orthogonal projection onto $X_n$, we infer that
			 \begin{equation}\label{27.9.1}
			 	\|g-\bar{u}_n\|_{L^2(\T^2)} \geq \inf_{u \in X_n} \|g-u\|_{L^2(\T^2)} = \|g-P_n g\|_{L^2(\T^2)}. 
			 \end{equation}
			 Suppose now that $n \geq 4$ and let $I^*_n$ be the unique interval  of the form  $I_{n, i} = [\frac{i-1}{n}, \frac{i}{n}]$ with $1 \leq i \leq n$ that contains $7/8$, that is, there exists $a \in [0, 1/n)$ such that $I^{*}_n = [\frac{7}{8}-a, \frac{7}{8}-a+\frac{1}{n})$. Then 
			 \begin{align}
			 \|g-P_n g\|_{L^2(\T^2)}^2 &\geq \sum_{i=1}^n  \int_{I^*_n \times I_{n, i}} |g(x)-P_n g(x)|^2 \, dx \nonumber \\
			 &  =  \int_{I^*_n } |g_1(x_1)-(g_1)_{I_n^*}|^2 \, dx_1  \nonumber \\
			 & =  \bigg[ \int_{I_n^*} g^2 - \frac{1}{|I^*_n|}\bigg( \int_{I_n^*} g  \bigg)^2 \bigg] \nonumber \\
			 &  = \frac{1}{n^{2 \alpha}} \bigg[ \int_{-n a}^{-n a +1} |x_1|^{2\alpha -1} \, dx_1 - \bigg( \int_{-n a}^{-n a+1} (\text{sgn } x_1)  |x_1|^{\alpha-\frac{1}{2}} \, dx_1  \bigg)^2 \bigg] \approx \frac{1}{n^{2 \alpha}}, \label{27.9.2}
			 \end{align}
			 where the last step follows from periodicity and the fact that $\alpha \neq 1/2$. As a combination of \eqref{27.9.1} and \eqref{27.9.2}, we get
			 $$
			 	\|g -\bar{u}_n\|_{L^2(\T^2)} \gtrsim n^{-\alpha}.
			 $$
			 This proves that \eqref{4.10.4} holds for some $f \in B^\alpha_{2, \infty}(\T^2)$ with $\alpha > 1/2$ and $\bar{u}= g$.

Next we turn our attention to \eqref{1.10.5}. 			Let $h = 1/n$ and $\iota = (\lfloor{n/2} \rfloor + 1/2) h$.  Define $\varepsilon = (2 j +1) \frac{h}{2}$, where $j$ is the largest integer such that $\varepsilon \leq h^{\alpha+\frac{1}{2}}$. The existence of such a value $j$ is guaranteed by the assumption $\alpha < 1/2$. Without loss of generality (taking $n \geq n_0$ for some $n_0$ sufficiently large if necessary), we may assume that   $\iota + \varepsilon < 5/6$. Let $\varphi: [0, 1] \longrightarrow \R$ be the piecewise linear function that interpolates the points 
			 $$
			 	(0, 2), \quad \Big(\iota-\frac{h}{2}, \varepsilon \Big), \quad (\iota, 0),  \quad \Big(\iota+\frac{h}{2}, \varepsilon - \frac{h}{2} \Big), \quad (\iota+\varepsilon, 0), \quad (1, 2)
			 $$
			 and $z_1 = 1 -\varphi$. Since $0 \leq \varphi \leq 2$,  we have $|z| \leq 1$ for $z(x) = (z_1(x_1), 0)$. Define
			 $$
			 	f_n(x) = (\mathbf{1}_{[\iota, 1)} + \varphi') (x_1) = (\mathbf{1}_{[\iota, 1)} - z') (x_1),
			 $$
			 or equivalently,  
			 $$
			 	f_n(x) = \sum_{i=1}^5 c_i \mathbf{1}_{J_i}(x_1),
			 $$
			 where
			 $$
			 c_1 = \frac{\varepsilon-2}{\iota-\frac{h}{2}}, \qquad c_2  = -\frac{2 \varepsilon}{h}, \qquad c_3= \frac{2 \varepsilon}{h}, \qquad c_4 = 0, \qquad c_5 = \frac{2}{1-\iota-\varepsilon} + 1, 
			 $$
			 and 
			 $$
			 	J_1 = \Big(0, \iota-\frac{h}{2} \Big), \quad J_2 = \Big(\iota-\frac{h}{2}, \iota \Big), \quad J_3 = \Big(\iota, \iota+ \frac{h}{2} \Big), 
			 $$  
			 $$
			 	J_4 = \Big( \iota+ \frac{h}{2}, \iota + \varepsilon \Big), \quad J_5 = (\iota+\varepsilon, 1). 
			 $$
			 We also let 
			 $$
			 	g_n(x) = \mathbf{1}_{[\iota, 1)}(x_1). 
			 $$
			 Clearly
			 \begin{equation}\label{1.10.2}
			   \|g_n\|_{L^\infty(\T^2)} = 1 \qquad \text{and} \qquad \text{TV}(g_n) = 2.
			   \end{equation}
			    Moreover, we have 
			 $$
			 	\langle -z_1', g_n  \rangle = z_1(\iota)-z_1(1) = \varphi(1)-\varphi(\iota) = 2 = \text{TV} (g_n), 
			 $$
			 which shows that $g_n$ is the unique minimizer of \eqref{PF18.3} with data $f_n$. 
			 
			 To estimate $\|f_n\|_{\dot{B}^\alpha_{2, \infty}}$, we apply that
			 $$
			 	\|\Delta_{\xi} f_n\|_{L^2(\T^2)} \leq \sum_{i=1}^5 |c_i| \|\Delta_{\xi_1} \mathbf{1}_{J_i}\|_{L^2(\T)} \lesssim \sum_{i=1}^5 |c_i| \min\{|\xi_1|, |J_i|\}^{1/2}
			 $$
			 and we next estimate each of the terms in the latter sum. If $i=1$, we use that $\iota-\frac{h}{2} \geq \frac{1}{6}$ (if $n_0$ is large enough)  to derive that $|c_1| \lesssim 1$ uniformly with respect to $n$. If $i=5$ then $|c_5| = \frac{2}{1-\iota-\varepsilon} + 1 \lesssim 1$ since $\iota+\varepsilon < \frac{5}{6}$. Hence, if $|\xi| \leq 1$ then 
			 $$
			 	\|\Delta_{\xi} f_n\|_{L^2(\T^2)}  \lesssim |\xi_1|^{1/2} + \frac{\varepsilon}{h} \min \Big\{|\xi_1|, \frac{h}{2} \Big\}^{1/2} \lesssim  |\xi_1|^{1/2} + h^{\alpha-1/2} \min \Big\{|\xi_1|, \frac{h}{2} \Big\}^{1/2} \lesssim |\xi|^\alpha. 
			 $$ 
			 This together with the fact that $\|f_n\|_{L^2(\T^2)} \lesssim 1$ uniformly with respect to $n$ yields that 
			 \begin{equation}\label{1.10.4}
			 \sup_{n \geq n_0} \|f_n\|_{\dot{B}^\alpha_{2, \infty}} < \infty. 
			 \end{equation}
			 It follows from \eqref{1.10.2} and \eqref{1.10.4} that $f_n \in \mathcal{K}$ (cf. \eqref{1.10.3}) for every $n \geq n_0$.

			 Note that, by construction, the points $\iota-\frac{h}{2}, \iota + \frac{h}{2}$ and $\iota + \varepsilon$ are  of the form $\frac{i}{n}$ for some $1 \leq i \leq n-1$ and then 
			 $$
			 	P_n f_n (x) = c_1 \mathbf{1}_{J_1}(x_1) + c_5 \mathbf{1}_{J_5}(x_1). 
			 $$
			 We are going to prove that $$\bar{u}_n(x) = \frac{\varepsilon}{\iota+\varepsilon} \,  \mathbf{1}_{[0, \iota + \varepsilon)}(x_1) + \mathbf{1}_{[\iota+\varepsilon, 1)} (x_1)$$
			 is the unique minimizer of \eqref{PF18.3} with data $P_n f_n$. Note that $\bar{u}_n \in X_n$. Let $\zeta_1: [0, 1] \longrightarrow \R$ be such that $\zeta_1(0)=-1$ and
			 \begin{align*}
			 \zeta_1' (x_1)&= \bar{u}_n(x_1, 0)-P_n f_n(x_1, 0) \\
			 & = \Big(\frac{\varepsilon}{\iota+\varepsilon}-c_1 \Big) \mathbf{1}_{J_1}(x_1) + \frac{\varepsilon}{\iota+\varepsilon} \,  \mathbf{1}_{J_2 \cup J_3 \cup J_4}(x_1) + (1-c_5) \mathbf{1}_{J_5}(x_1)
			 \end{align*}
			 or equivalently
			 $$
			 	\zeta_1(x_1) -\zeta_1(0) = \begin{cases}
				\big(\frac{\varepsilon}{\iota+\varepsilon}-c_1 \big)x_1 & \quad \text{if} \quad x_1 \in J_1, \\
					 \big(\frac{\varepsilon}{\iota+\varepsilon}-c_1 \big) \big(\iota-\frac{h}{2} \big) + \frac{\varepsilon}{\iota+\varepsilon} \big(x_1 -\iota+\frac{h}{2} \big) & \quad \text{if} \quad x_1 \in J_2 \cup J_3 \cup J_4, \\
				  \big(\frac{\varepsilon}{\iota+\varepsilon}-c_1 \big) \big(\iota-\frac{h}{2} \big) + \frac{\varepsilon}{\iota+\varepsilon} \big(\varepsilon +\frac{h}{2} \big)  + (1-c_5) (x_1-\iota-\varepsilon) & \quad \text{if} \quad x_1 \in J_5. 	
				\end{cases}
			 $$
			 Observe that $\zeta_1(1) = \zeta_1(0)=-1, \zeta_1(\iota+\varepsilon)=1$, $\zeta_1$ is increasing on $[0, \iota+\varepsilon]$ and decreasing on $[\iota+\varepsilon, 1]$, so $|\zeta_1(x_1)| \leq 1$.  Let $\zeta(x) = (\zeta_1(x_1), 0)$.  Then $|\zeta(x)| \leq 1$ and, in addition, 
			 $$
			 	\langle - \zeta'_1, \bar{u}_n \rangle = 2 \Big(1- \frac{ \varepsilon}{\iota+\varepsilon} \Big) = \text{TV}(\bar{u}_n),
			 $$
			 which yields the desired assertion that $\bar{u}_n$ is the minimizer of \eqref{PF18.3} associated to $P_n f_n$. 
			 
			 Finally we estimate the error between  the continuous solution $g_n$ and the approximated solution, namely, 
			 \begin{equation}\label{1.10.1}
			 	\|g_n - \bar{u}_n\|_{L^2(\T^2)}^2  = \bigg(\frac{\varepsilon}{\iota+\varepsilon} \bigg)^2 \iota + \bigg(1-\frac{\varepsilon}{\iota+\varepsilon} \bigg)^2 \varepsilon = \frac{\varepsilon \iota}{\iota+\varepsilon} \gtrsim \varepsilon.
			 \end{equation}
			 Since $n \geq n_0$ (recall also that $\alpha < 1/2$) and using the definition of $\varepsilon$, we arrive at the following estimates 
			 $$
			 	\frac{h^{\alpha + \frac{1}{2}}}{2} \leq h^{\alpha+\frac{1}{2}} - h \leq \varepsilon,
			 $$
			and then (by \eqref{1.10.1})
			$$
				\|g_n - \bar{u}_n\|_{L^2(\T^2)}  \gtrsim h^{\frac{\alpha}{2} + \frac{1}{4}} = n^{-\frac{\alpha}{2} - \frac{1}{4}}. 
			$$
			 This completes the proof of the desired assertion \eqref{1.10.5}. 
\end{proof}

The proof of Theorem \ref{Thm1Intro} can be modified to establish
extrapolation estimates for the rates of convergence of minimizers in the spirit of the BBM formula \eqref{14.9.4}. Specifically, we establish the following result.  

\begin{corollary}
\label{Thm191.} Let $s\in(0,1)$ and $p\in\lbrack1,2)$. Suppose that $f\in\dot{W}^{s,p}\cap L^{\infty},$ then 
\begin{equation}
\Vert\bar{u}-\bar{u}_{n}\Vert_{L^{2}}\lesssim n^{-\frac{sp}{2}}[(1-s)^{%
\frac {1}{p}}\Vert f\Vert_{\dot{W}^{s,p}}]^{\frac{p}{2}}\Vert
f\Vert_{L^{\infty}}^{1-\frac{p}{2}},   \label{P19.12}
\end{equation}
where the implicit constant is independent of $n$ and $s$. As a
consequence, taking limits as $s\rightarrow1^{-}$ and applying \eqref{14.9.4}: 
\begin{equation*}
\Vert\bar{u}-\bar{u}_{n}\Vert_{L^{2}}\lesssim n^{-\frac{p}{2}}\Vert f\Vert_{%
\dot{W}^{1,p}}^{\frac{p}{2}}\Vert f\Vert_{L^{\infty}}^{1-\frac{p}{2}}\qquad%
\text{if}\qquad p\in(1,2) 
\end{equation*}
and 
\begin{equation*}
\Vert\bar{u}-\bar{u}_{n}\Vert_{L^{2}}\lesssim n^{-\frac{1}{2}}\Vert f\Vert_{%
\dot{BV}}^{\frac{1}{2}}\Vert f\Vert_{L^{\infty}}^{\frac{1}{2}} \qquad \text{if} \qquad p=1. 
\end{equation*}
\end{corollary}

\begin{proof}
We have the trivial estimates 
\begin{equation}
\Vert f-P_{n}f\Vert_{L^{2}}\leq\Vert
f-P_{n}f\Vert_{L^{p}}^{1-\theta} \|f-P_n f\|_{L^\infty}^\theta \leq 2^\theta  \Vert
f-P_{n}f\Vert_{L^{p}}^{1-\theta}  \Vert f\Vert_{L^{\infty}}^{\theta}, 
\label{PF19.3}
\end{equation}
where $P_{n}$ is the projection operator defined by \eqref{PF19.2} and $\theta = 1-\frac{p}{2}$. 
Moreover, by construction, 
\begin{equation}
\Vert f-P_{n}f\Vert_{L^{p}}^{p}=\sum_{0\leq i,j\leq n-1}\Vert
\mathcal{T}_{S_{i,j}^{h}}f\Vert_{L^{p}(S_{i,j}^{h})}^{p}   \label{PF19.4}
\end{equation}
where 
\begin{equation*}
\mathcal{T}_{S_{i,j}^{h}}f(x):=f(x)-f_{S^h_{i, j}}. 
\end{equation*}
On the other hand, the classical Poincar\'{e} inequality applied to the
square $S_{i,j}^{h}$ implies 
\begin{equation}
\Vert \mathcal{T}_{S_{i,j}^{h}}f\Vert_{L^{p}(S_{i,j}^{h})}\lesssim h\Vert f\Vert _{%
\dot{W}^{1,p}(S_{i,j}^{h})}.   \label{PF19.5}
\end{equation}
Interpolating \eqref{PF19.5} with the trivial estimate 
\begin{equation*}
\Vert \mathcal{T}_{S_{i,j}^{h}}f\Vert_{L^{p}(S_{i,j}^{h})}\lesssim\Vert f\Vert
_{L^{p}(S_{i,j}^{h})},   
\end{equation*}
we arrive at 
\begin{equation}
[
\mathcal{T}_{S_{i,j}^{h}}f]_{(L^{p}(S_{i,j}^{h}),L^{p}(S_{i,j}^{h}))_{s,p}}%
\lesssim[ f]_{(L^{p}(S_{i,j}^{h}), h\dot{W}%
^{1,p}(S_{i,j}^{h}))_{s,p}},   \label{PF19.7}
\end{equation}
where $[\cdot]_{(A_0, A_1)_{s, p}}$ refers to the modified interpolation\footnote{ We mention that  \eqref{5.10.2} plays an important role in extrapolation theory as can be seen in \cite{JM, KM, M-Lim}. In particular, working with ordered pairs $(A_0, A_1)$, one has that (cf. \cite[Lemma 3]{M-Lim})   $$[f]_{(A_0, A_1)_{s, p}} \leq \|f\|_{(A_0, A_1)_{s, p}} \leq (1+s^{-1/p} (1-s)^{1/p}) [f]_{(A_0, A_1)_{s, p}}.$$}   (semi-)norm (cf. \eqref{5.10.1}) given by 
\begin{equation}\label{5.10.2}
	 [f]_{(A_0, A_1)_{s,p}} := \bigg(\int_0^1 [t^{-s}K(t, f; A_0, A_1)]^p \, \frac{dt}{t}	 \bigg)^{1/p}
\end{equation}
 and the hidden equivalence constant in \eqref{PF19.7} is independent of $h$ and $S^h_{i, j}$ with $0\leq
i,j\leq n-1$. By virtue of \eqref{sec2:3}, we can compute the left-hand side
of \eqref{PF19.7} as 
\begin{equation}
[
\mathcal{T}_{S_{i,j}^{h}}f]_{(L^{p}(S_{i,j}^{h}),L^{p}(S_{i,j}^{h}))_{s,p}}= ((1-s)p)^{-1/p} \|
\mathcal{T}_{S_{i,j}^{h}}f\|_{L^{p}(S_{i,j}^{h})}.   \label{PF19.9}
\end{equation}
Concerning the right-hand side of \eqref{PF19.7}, we can make use of the
simple fact that\footnote{%
More generally, a simple change of variables gives  $[ f]
_{(A_{0},hA_{1})_{s,p}} =  h^{s}\big( \int_0^h [t^{-s}K(t, f; A_0, A_1)]^p \, \frac{dt}{t} \big)^{1/p}$ for every $%
h > 0$.} 
\begin{equation}
[ f]_{(L^{p}(S_{i,j}^{h}),h\dot{W}^{1,p}(S_{i,j}^{h}))_{s,p}} =  h^{s} \bigg( \int_0^h [t^{-s} K(t, f; L^p(S^h_{i, j}), \dot{W}^{1, p}(S_{i, j}^h))]^p \, \frac{dt}{t} \bigg)^{1/p}
\label{P19.10}
\end{equation}
and, in addition, the computations carried out in \eqref{PD28.3} (more
precisely, their counterparts adapted to cubes which hold \textit{mutatis
mutandis}, in particular, the formula  \eqref{PD28.1} remains valid if $t \in (0, h)$; we refer to \cite[p. 17]{DLTYY} for further details) show that 
\begin{equation*}
\bigg( \int_0^h [t^{-s} K(t, f; L^p(S^h_{i, j}), \dot{W}^{1, p}(S_{i, j}^h))]^p \, \frac{dt}{t} \bigg)^{1/p} \approx%
\Vert f\Vert_{\dot{W}^{s,p}(S_{i,j}^{h})}
\end{equation*}
with equivalence constants that are independent of $s \in (0, 1)$ and $S_{i,j}^{h}$.
Hence \eqref{P19.10} yields 
\begin{equation}
[ f]_{(L^{p}(S_{i,j}^{h}),h\dot{W}^{1,p}(S_{i,j}^{h}))_{s,p}}\approx
h^{s}\Vert f\Vert_{\dot{W}^{s,p}(S_{i,j}^{h})}.   \label{P19.11}
\end{equation}
Plugging \eqref{PF19.9} and \eqref{P19.11} into \eqref{PF19.7}, we obtain 
\begin{equation*}
\Vert \mathcal{T}_{S_{i,j}^{h}}f\Vert_{L^{p}(S_{i,j}^{h})}\lesssim ((1-s)p)^{1/p} h^{s}\Vert
f\Vert_{\dot{W}^{s,p}(S_{i,j}^{h})}
\end{equation*}
and then \eqref{PF19.4} gives 
\begin{equation*}
\Vert f-P_{n}f\Vert_{L^{p}}\lesssim (1-s)^{1/p}h^{s}\bigg(\sum_{0\leq i,j\leq
n-1}\Vert f\Vert_{\dot{W}^{s,p}(S_{i,j}^{h})}^{p}\bigg)^{1/p}\leq
(1-s)^{1/p}h^{s}\Vert f\Vert_{\dot{W}^{s,p}}. 
\end{equation*}
The last estimate combined with \eqref{PF19.3} yields 
\begin{equation*}
\Vert f-P_{n}f\Vert_{L^{2}} \lesssim[(1-s)^{1/p}h^{s}\Vert f\Vert_{\dot{W}%
^{s,p}}]^{1-\theta}\Vert f\Vert_{L^{\infty}}^{\theta}. 
\end{equation*}
This together with \eqref{24.9.10} leads to the desired estimate %
\eqref{P19.12}.
\end{proof}

\subsubsection{Error bounds via  Brezis--Van
Schaftingen--Yung spaces.}

 Let  $\alpha\in(0,1),p\in(1,\infty)$
and $\gamma\in\mathbb{R}\backslash\{0\}$. The Brezis--Van Schaftingen--Yung space is defined through the functional   
\begin{equation}
\Vert f\Vert_{\dot{B}_{p}^{\alpha}(\gamma,\infty)}^{p}:=\sup_{\lambda>0}%
\iint_{|\Delta_h f(x)|>\lambda|h|^{\alpha+\frac{\gamma}{p}}}\frac{\lambda^{p}}{%
|h|^{2-\gamma}}\,dx\,dh;  \label{BVSY}
\end{equation}
cf. \cite{BVY, BVY2}. 
Note the connection with weak-type conditions: 
\begin{equation*}
\Vert f\Vert_{\dot{B}_{p}^{\alpha}(\gamma,\infty)}=\bigg\|\frac{\Delta_h f(x)}{%
|h|^{\alpha+\frac{\gamma}{p}}}\bigg\|_{L^{p,\infty}(\nu_{\gamma})},
\end{equation*}
where $L^{p,\infty}(\nu_{\gamma})$ denotes the Marcinkiewicz space (also
known as weak-$L^{p}$ space) on the product space $\mathbb{T}^{2}\times 
[-\frac{1}{2}, \frac{1}{2})^2$ endowed with the measure 
\begin{equation*}
\nu_{\gamma}(E):=\iint_{E}\frac{1}{|h|^{2-\gamma}}\,dx\,dh. 
\end{equation*}

\begin{proof}[Proof of Theorem \ref{Thm2Intro}]
In light of Theorem \ref{Thm1Intro}, the desired estimate \eqref{25.3Nov} can
be seen  as a simple consequence of the following general question: Given $%
\alpha\in(0,1)$, characterize those $f\in L^{2}$ that satisfy 
\begin{equation}
K(t,f;L^{2},\dot{W}^{1,p})\lesssim t^{\alpha}.   \label{25.2Nov}
\end{equation}
In the language of interpolation theory \eqref{25.2Nov} is equivalent to the
statement 
\begin{equation}
f\in(L^{2},\dot{W}^{1,p})_{\alpha,\infty}.   \label{PF19.13}
\end{equation}
In this way the problem is reduced to characterizing the interpolation spaces
in \eqref{PF19.13}. The Besov spaces $\dot{B}_{2,\infty}^{\alpha}$ provide
the  well-known answer when $p=2$ (cf. Theorem \ref{Coro8.32}). However,
the situation is more delicate when $p\neq2$ where there is no suitable 
characterization of $K(t,f;L^{2},\dot{W}^{1,p}).$ Despite this obstruction,
the interpolation spaces arising in \eqref{PF19.13} with $p \neq 2$ have been recently
described in \cite{DSSVSY} in terms of the functionals \eqref{BVSY}%
. To be more precise, according to \cite[Corollary 1.15]{DSSVSY} (note that the arguments carried out in \cite{DSSVSY} refer to $\R^N$, but similar arguments also work in the periodic setting) the following interpolation formula holds 
\begin{equation}
(L^{2},\dot{W}^{1,p})_{\alpha,\infty}=\dot{B}_{q}^{\alpha}(\gamma ,\infty) 
\label{PF20.1}
\end{equation}
where $\frac{1}{q}=\frac{1-\alpha}{2}+\frac{\alpha}{p}$ and $\frac{1}{\gamma 
}=\frac{1}{2}-\frac{1}{p}.$ Then, by \eqref{enlaprimera}, 
\begin{equation}
\Vert\bar{u}-\bar{u}_{n}\Vert_{L^{2}}\lesssim K(n^{-2+\frac{2}{p}},f;L^{2},%
\dot{W}^{1,p})\lesssim n^{-(2-\frac{2}{p})\alpha} \|f\|_{\dot{B}_{q}^{\alpha}(\gamma,\infty)}   \label{PF20.2}
\end{equation}
provided that $\alpha \geq \frac{p}{4(p-1)}$. 
 Note that, in view
of \eqref{PF20.1} and the Sobolev embedding theorem, 
\begin{equation*}
\dot{B}_{q}^{\alpha}(\gamma,\infty)\subset L^{2}+\dot{W}^{1,p}\subset L^{2}. 
\end{equation*}
Then \eqref{25.3Nov} with $\alpha \geq 1/2$ follows after applying a simple change of variables ($\beta \leftrightarrow \alpha$) in %
\eqref{PF20.2}. The case $\alpha < 1/2$ in \eqref{25.3Nov} can be obtained similarly.  In addition, the validity of \eqref{25.3Nov} for $f\in\dot {W%
}^{\beta,q}$ is an immediate consequence of $\dot {W}%
^{\beta,q}\hookrightarrow\dot{B}_{q}^{\beta}(\gamma,\infty)$, which easily follows from the trivial embedding  $L^p(\nu_\gamma) \hookrightarrow L^{p, \infty}(\nu_\gamma)$.
\end{proof}

\end{document}